\documentclass[a4paper,11pt]{article}
\usepackage[margin=1in]{geometry}

\usepackage{amsmath}
\usepackage{amssymb}
\usepackage{graphicx,psfrag}
\usepackage{amsfonts}
\usepackage{amscd}
\usepackage{enumerate,placeins,wrapfig,booktabs,chngcntr}

\usepackage{indentfirst}

\RequirePackage{algorithm}
\RequirePackage{algorithmic}
\RequirePackage{color}

\usepackage{subcaption} \usepackage{multirow}

\usepackage{nccmath}

\allowdisplaybreaks[4]
\def \x{\mathbf{x}}
\def \y{\mathbf{y}}
\def \yy{y}
\def \v{\mathbf{v}}
\def \vv{v}
\def \zz{\mathbf{z}}
\def \aa{\mathbf{a}}
\def \bb{\mathbf{b}}
\def \ee{\mathbf{e}} \def \rr{\mathbf{r}} 
\def \gg{g}
\def \dd{\mathbf{d}} 

\def \X{\mathcal{X}}
\def \Y{\mathcal{Y}}
\def \F{\mathcal{F}} \def \N{\mathcal{N}}  \def \LL {L}

\def \R{\mathbb{R}}

\def \RR{R} 

\def \mid {\mathrm{mid}}
\def \Proj{\mathtt{Proj}}

\newtheorem{theorem}{Theorem}[section]
\newtheorem{lemma}[theorem]{Lemma}

\newtheorem{proposition}[theorem]{Proposition}
\newtheorem{definition}[theorem]{Definition}
\newtheorem{assumption}[theorem]{Assumption}  \newtheorem{remark}[theorem]{Remark}

\newenvironment{proof}{\bigskip\noindent\textbf{Proof.}\ \ }{\qed}\def\qed{ \hfill $\square$}


\title{\textbf{Optimization with Parametric Variational Inequality Constraints on a Moving Set \\ Based on a Smoothing Method}}
\author{Xiaojun CHEN\thanks{Department of Applied Mathematics, The Hong Kong Polytechnic University, Kowloon, Hong Kong, China, (xiaojun.chen@polyu.edu.hk).}
\and Jin ZHANG\thanks{Department of Mathematics, and National Center for Applied Mathematics Shenzhen, Southern University of Science and Technology, Shenzhen, Guangdong, China, (zhangj9@sustech.edu.cn).}
\and Yixuan ZHANG\thanks{Department of Applied Mathematics, The Hong Kong Polytechnic University, Kowloon, Hong Kong, China;
	National Center for Applied Mathematics Shenzhen, Southern University of Science and Technology, Shenzhen, Guangdong, China, (yi-xuan.zhang@connect.polyu.hk).}}
\date{}

\begin{document}
\maketitle

\begin{abstract}
This paper focuses on optimization problems constrained by Parametric Variational Inequalities (PVI) defined on a moving set.
	Unlike most existing works on mathematical programs with equilibrium constraints, the equilibrium constraints have parameters not only in the function but also in the related set.
	We show that the solution function of the PVI is Lipschitz continuous with respect to the upper-level decision variables and the solution set of
	the optimization problem is nonempty and bounded. Moreover, we prove that the metric regularity of the constraints holds automatically, which allows us to characterize stationary points without any additional assumptions.
	A Smoothing Implicit Gradient Algorithm (SIGA) is proposed based on the smoothing approximation of the PVI
	with a novel scheme for updating the smoothing parameter. We prove the convergence of SIGA to a stationary point and  validate the efficiency of SIGA by numerical examples from synthetic problems and portfolio management with real data.
\end{abstract}



\section{Introduction}
Let $\X \subset \R^m$ be a nonempty, bounded, closed and convex set and let
$\Omega: \X \rightrightarrows \R^n$ be a set valued mapping that is
nonempty, closed and convex for any $\x \in \X$.
A Parametric Variational Inequalities (PVI) problem parameterized by~$\x$, denoted as~VI$(\Omega(\x),F(\x,\cdot))$,
seeks $\y \in \Omega(\x)$ such that
\[
(\y' - \y)^\top F(\x,\y) \ge 0, \,\,\,  \forall \y' \in \Omega(\x),
\]
where $F: \R^m \times \R^n \rightarrow \R^n$ is a continuously differentiable function.
The function $F$ and set valued mapping $\Omega$ capture the interplay between decision variables $\y$ and parameters~$\x$.
It is known that for any fixed $\x \in \X$,
a vector $\y$ solves VI$(\Omega(\x), F(\x,\cdot))$ if and only if
\[
\y = \Proj_{\Omega(\x)} (\y - \delta F(\x,\y)),
\]
where $\delta>0$ is any fixed constant \cite[Proposition 1.5.8]{facchinei2003finite}.

In this paper, we consider the following optimization problem with PVI constraints
\begin{equation}\label{eq:1equation constraint}
	\tag{$\mathrm{P}$}
	\begin{aligned}
		\min\limits_{\x \in \X, \ \y} \  & f(\x,\y)  \\\text{ s.t. }\ & \y =\Psi(\x,\y),
	\end{aligned}
\end{equation}
where $f:\R^m \times \R^n \rightarrow \R$ is continuously differentiable and
\begin{equation}\label{Psi}
	\Psi(\x,\y)=\Proj_{\Omega(\x)} (\Phi(\x,\y)) \,\, \, \, {\rm with} \,\,\,\,
	\Phi(\x,\y)= \y - \delta F(\x,\y).
\end{equation}

Problem~\eqref{eq:1equation constraint} captures
a powerful class of problems that arise in a wide range of applications
including but not limited to machine learning~\cite{liu2023hierarchical,chu2026provably,liu2023averaged},
portfolio management~\cite{chen2018smoothing},
traffic network design~\cite{labbe2021bilevel,marcotte1986network}, and fluid mechanics~\cite{chen2022efficient},
especially considering its close relation to bi-level programming.
Indeed,~\eqref{eq:1equation constraint} encompasses more general forms of problems than traditional bi-level programming~\cite{dempe2002foundations,luo1996mathematical}. However,~\eqref{eq:1equation constraint} is known to be particularly challenging because it combines the complexity of the PVI with the need to optimize an objective function, creating a two-level structure hard to analyze and solve. A key feature of problem \eqref{eq:1equation constraint}  is the moving set, where the feasible region of variables evolves with the problem's parameters, mirroring various application scenarios.
For example,
in machine learning, hyperparameter tuning receives much attention while the range of variables may rely on the hyperparameters;
in portfolio optimization, asset allocation bounds may shift with market conditions;
in traffic network design, route capacities vary with flow dynamics.
See~\cite[Chapter~1.2]{luo1996mathematical} for more detailed applications on Stackelberg game in economics and origin-destination demand adjustment problem in transportation.
The dynamic nature from the moving set introduces significant complexity, as the constraints of problem~\eqref{eq:1equation constraint} involve a nonsmooth function
$\Psi$, which is defined by a projection on a parametric set $\Omega(\x)$.

In response to the challenges associated with a moving set $\Omega(\x)$, existing research primarily concentrates on the simplified scenario where the Variational Inequality (VI) is defined over a fixed set $\Omega(\x)\equiv{\bf \Omega}$ for any $\x \in \X$; see, for instance, \cite{gfrerer2017new, lin2009solving,ye2000constraint}.
Specially, approaches based on Mathematical Programs with Complementarity Constraints (MPCC) generally focus on cases where the VI reduces directly to complementarity constraints \cite{chen2004smoothing, chen2009class,scholtes2001convergence},
or they transform the VI into complementarity constraints \cite{dutta2025nonconvex, facchinei1999smoothing,okuno2021lp, outrata1995numerical,wu2015inexact}, where ${\bf \Omega}=\R^n_+$.
However, the VI$({\bf \Omega}, F(\x,\cdot))$ may have limitations in modeling some real applications. For example, under this model, the constraints of the lower-level optimization problem in a bi-level optimization problem do not depend on upper-level decision variables.
Consequently, research on \eqref{eq:1equation constraint} with the PVI defined on a moving set, especially in the context of algorithmic development, remains relatively scarce.
It is noteworthy that \cite{ye1997necessary} investigates the necessary optimality conditions for \eqref{eq:1equation constraint} but does not explore numerical methods,
while~\cite{cui2023complexity} considers a stochastic version.

Faced with these challenges,
the aforementioned fixed-point reformulation of the PVI allows us to leverage the properties of the projection operator, enabling efficient handling of the moving set.
However, the direct challenge after reformulation stems from the nonsmooth equality constraint caused by the projection operator onto a moving set.
As a result, gradient-based algorithms~\cite{fung2022jfb,liu2022optimization,liu2023hierarchical} designed for smooth fixed-point constraints cannot be directly applied. Furthermore, while some works such as~\cite{mckenzie2024three} address nonsmooth cases, they remain unsuitable for our model as they only handle scenarios where the moving set reduces to a fixed set.
Additionally, smoothing approximation of the projection onto a moving set remains largely unexplored, with some existing on a fixed set~\cite{chen1998global,gabriel1997smoothing,qi2000new}.

The contributions of this work are summarized as follows.

\begin{enumerate}[(i)]
	\item
	We show that $y(\x)$ as a solution to VI$(\Omega(\x),F(\x,\cdot))$ defines a Lipschitz continuous solution function from $\X$ to $\R^n$ and the solution set of problem~\eqref{eq:1equation constraint} is nonempty and bounded. Moreover, we show that the metric regularity of the constraints of (\ref{eq:1equation constraint}) holds automatically, which allows us to characterize stationary points without any additional assumptions.
	\item
	We extend smoothing approximation of the VI over a fixed set ${\bf \Omega}$ to  the PVI over a parametric set $\Omega(\x)$, and use a smoothing function $\Psi_\mu$ of $\Psi$ and
	\begin{equation}\label{eq:1equation constraint smoothing} \tag{$\mathrm{P_{\mu}}$}
		\begin{aligned}
			\min\limits_{\x \in \X, \ \y} \  & f(\x,\y)  \\\text{ s.t. }\ & \y = \Psi_{\mu}(\x,\y)
		\end{aligned}
	\end{equation}
	to approximate problem~\eqref{eq:1equation constraint}.
	We establish important
	properties of the smoothing function $\Psi_\mu$ and   relationship between problem (\ref{eq:1equation constraint smoothing}) and problem~\eqref{eq:1equation constraint}.

	\item
	\textcolor{black}{
	We propose a Smoothing Implicit Gradient Algorithm (SIGA), a framework designed to handle the obstacles posed by a moving constraint set
	with a novel scheme for updating the smoothing parameter.
	With the smoothing parameter gradually diminishing to zero,
	we prove that any accumulation point of a sequence generated by SIGA is a stationary point of~\eqref{eq:1equation constraint} under standing assumptions. Experimental results from numerical examples and portfolio management verify the convergence of SIGA and showcase its efficiency.
}
\end{enumerate}

\textbf{Organization.}
Section~\ref{sec:2} analyzes fundamental properties of~\eqref{eq:1equation constraint}, including the solution set and optimality conditions.
Section~\ref{sec: smoothing} analyzes  the smoothing problem
(\ref{eq:1equation constraint smoothing}).
Section~\ref{sec: SIGA} presents our algorithmic framework SIGA based on the implicit function theorem with convergence guarantees.
\color{black}
Section~\ref{sec: experiments} demonstrates the efficiency of SIGA by numerical examples from synthetic problems and portfolio management with real data.
\color{black}

\textbf{Notations.}
In this paper, $\|\cdot\|$ denotes the $l_2$-norm for vectors and spectral norm for matrices.
{\color{black}
$U_{\X} \subset\R^m$ denotes a fixed bounded open
convex neighborhood of~$\X$,
and $\Psi$ is defined on $U_{\X} \times\R^n$ by a fixed locally
Lipschitz extension from $\X\times\R^n$.
}
For any $\x \in \X$ and $\y \in \R^n$,
$\partial \Psi_i(\x,\y)$ and $\partial_C \Psi(\x,\y)$ respectively represent the generalized gradient (Clarke subdifferential) and generalized Jacobian, following~\cite{chen1998global}:
$$\partial_C \Psi(\x,\y) = \partial \Psi_1(\x,\y) \times \partial \Psi_2(\x,\y) \times \cdots \times \partial \Psi_n(\x,\y),$$
denoting the set of all matrices whose $i$-th column belongs to $\partial \Psi_i(\x,\y)$ for each~$i$, with $C$ signifying the Cartesian product~\cite[(7.5.18)]{facchinei2003finite}.
We use $(\x,\y)$ to denote $(\x^\top,\y^\top)^\top$ for any vectors~$\x$ and $\y$ to simplify notations,
and $\mid\{\cdot,\cdot,\cdot\}$ to represent the componentwise median operator.

\section{Properties of problem~\eqref{eq:1equation constraint}}\label{sec:2} In this section, we present some fundamental properties of problem~\eqref{eq:1equation constraint} under the following blanket assumption of this paper.
\begin{assumption}
({\bf Blanket assumption}) \label{assump: foundation}
	\begin{enumerate}[(i)]
		\item \label{assump: Lipschitz proj}
{\textcolor{black}{
	The mapping
	$(\x,\v)\mapsto\Proj_{\Omega(\x)}(\v)$
	is locally Lipschitz continuous on $\X\times\R^n$.}
}
		\item \label{assump: contraction Psi}
		There exist $\LL \in (0,1)$ and $\delta>0$, such that the function $\Psi$ defined in~\eqref{Psi} satisfies
		\[
		\|\Psi(\x,\y) - \Psi(\x,\y')\| \le \LL \|\y-\y'\|
		\]
		for any $\x \in \X$, $\y,\y' \in \R^n$.

	\end{enumerate}
\end{assumption}

The following two propositions  give some sufficient conditions for Assumption~\ref{assump: foundation}(\ref{assump: Lipschitz proj}) and (\ref{assump: contraction Psi}), respectively.

\begin{proposition}\label{prop: Lipschitz proj} If $\Omega(\x) = \{ \y \in \R^n: A \y \ge b(\x), \y \ge \mathbf{c}\}$ is a nonempty polyhedron for any $\x \in \X$,
	where $A \in \R^{r \times n}, \mathbf{c}\in \R^n$ and $b:\X \rightarrow \R^r$ is $L_b$-Lipschitz continuous,
	then Assumption~\ref{assump: foundation}(\ref{assump: Lipschitz proj}) holds.
\end{proposition}

\begin{proof}
	\color{black}
	By \cite[Theorem 2.1]{yen1995lipschitz}, there exists $L_{\Omega}>0$
	such that
	\[
	\|\Proj_{\Omega(\x_1)}(\v)
	-\Proj_{\Omega(\x_2)}(\v)\|
	\le L_{\Omega} \|b(\x_1)-b(\x_2)\|
	\le L_{\Omega} L_b\|\x_1-\x_2\|
	\]
	for all $\x_1,\x_2\in\X$ and $\v\in\R^n$.
	Then for any $\x_1,\x_2\in\X$ and $\v_1,\v_2\in\R^n$,
	\[
	\begin{aligned}
		&\|\Proj_{\Omega(\x_1)}(\v_1)
		-\Proj_{\Omega(\x_2)}(\v_2)\|\\
		&\le
		\|\Proj_{\Omega(\x_1)}(\v_1)
		-\Proj_{\Omega(\x_2)}(\v_1)\|
		+
		\|\Proj_{\Omega(\x_2)}(\v_1)
		-\Proj_{\Omega(\x_2)}(\v_2)\|\\
		&\le L_{\Omega} L_b\|\x_1-\x_2\|+\|\v_1-\v_2\|\\
		& \le (L_{\Omega} L_b + 1) \left( \|\x_1-\x_2\|+\|\v_1-\v_2\| \right)
	\end{aligned}
	\]
	Thus the mapping
	$(\x,\v)\mapsto\Proj_{\Omega(\x)}(\v)$
	is jointly Lipschitz, and hence locally Lipschitz,
	on $\X\times\R^n$.

\end{proof}

\begin{proposition}\label{prop: contraction Psi}
	Suppose there exists $\bar{L} > 0$ such that
	\begin{equation}\label{LipF}
		\|F(\x,\y) - F(\x,\y')\| \le \bar{L} \| \y - \y' \|
	\end{equation}
	for any $\x \in \X, \y,\y' \in \R^n$.
	If one of the following conditions holds,
	then $\Phi(\x,\cdot)$ is $\LL$-Lipschitz over $\R^n$ for any $\x \in \X$,
	which guarantees Assumption~\ref{assump: foundation}(\ref{assump: contraction Psi}).
	\begin{enumerate}[(i)]
		\item \label{prop: contraction Psi strongly monotone} (Strong monotonicity) There exists $\sigma>0$, such that for any $\x \in \X$ and $ \y,\y' \in \R^n$,
		$$(\y-\y')^\top (F(\x,\y) - F(\x,\y')) \ge \sigma \|\y-\y'\|^2,$$
		and let $0 < \delta < 2 \sigma/\bar{L}^2$.

		\item
		(Uniform P function)
		There exists $\sigma>0$, such that for any $\x \in \X$ and $\y\in \R^n$,
		all eigenvalues $\lambda_i(\x,\y)$, $i=1,2,\cdots,n,$ of the matrix
		$
		\frac{1}{2}\! \left( \nabla_{\! \y} F(\x,\y) \! + \!\! \nabla_{\! \y} F(\x,\y)^{\! \top} \right)
		$
		satisfy $\min\limits_{1\le i\le n} \lambda_i(\x,\y) \ge \sigma$,
		and let $0 < \delta < 2 \sigma / \bar{L}^2$.

		\color{black}

	\end{enumerate}
\end{proposition}
\begin{proof}
	Since $\Omega(\x)$ is nonempty closed and convex for any $\x \in \X$,
	the operator $\Proj_{\Omega(\x)}(\cdot)$ is nonexpansive.
	Hence,
	as $\Psi$ is the composition of the projection operator and $\Phi$,
	it suffices to prove that
	$\Phi(\x,\cdot)$ is $\LL$-Lipschitz over $\R^n$ for any $\x \in \X$ under one of the following conditions.

	(i) Let $$\LL = \sqrt{ 1-2\delta \sigma + \delta^2 \bar{L}^2 }.$$
	Then from
	$$\|\y-\y'\|\|F(\x,\y) - F(\x,\y')\|\ge (\y-\y')^\top (F(\x,\y) - F(\x,\y')) \ge \sigma \|\y-\y'\|^2,$$
	and the Lipschitz condition of $F$, we have $\bar{L} \ge \sigma$, which implies that
	$\LL \in (0,1)$.
	From~\cite[(1.9)]{chen1998global2},
	we have $\Phi(\x,\cdot)$ is $\LL$-Lipschitz over $\R^n$ for any $\x \in \X$.

	(ii)
	Let $\LL = \sqrt{ 1 - 2\delta \sigma + \delta^2 \bar{L}^2 }$.
	For any $\x \in \X$, $\y \in \R^n$, and $\v \in \R^n$ with $\|\v\|=1$,
	it holds that
	\[
	\v^\top \nabla_{\y} F(\x,\y) \v
	= \v^\top \frac{\nabla_{\y} F(\x,\y) + \nabla_{\y} F(\x,\y)^\top}{2} \v \ge \sigma.
	\]
	From the Lipschitz condition of $F$, we have
	$\bar{L} \ge \sigma$,
	which implies that $\LL \in (0,1)$.
	Meanwhile,
	\[
	\begin{aligned}
		& (I - \delta \nabla_{\y} F(\x,\y))^\top (I - \delta \nabla_{\y} F(\x,\y)) \\
		= & I - 2 \delta \cdot \frac{\nabla_{\y} F(\x,\y) + \nabla_{\y} F(\x,\y)^\top}{2} + \delta^2 \nabla_{\y} F(\x,\y)^\top \nabla_{\y} F(\x,\y),
	\end{aligned}
	\]
	which indicates that
	$$
	\begin{aligned}
		& \|\nabla_{\y} \Phi(\x,\y) \|
		= \| I - \delta \nabla_{\y} F(\x,\y) \|
		\le \sqrt{ 1 - 2\delta \sigma + \delta^2 \bar{L}^2 }.
	\end{aligned}
	$$
	This guarantees that $\Phi(\x,\cdot)$ is $\LL$-Lipschitz over $\R^n$ for any $\x \in \X$.
\end{proof}

\subsection{Solution existence to \eqref{eq:1equation constraint}}\label{sec: existence of solution}
In this subsection, we show that problem~\eqref{eq:1equation constraint} has  at least one solution $(\x^*,\y^*)$ under Assumption~\ref{assump: foundation}(\ref{assump: Lipschitz proj})(\ref{assump: contraction Psi}).

\begin{lemma}\label{lem:proj Omega continuous}
	\color{black}
	$\Psi$ is jointly locally Lipschitz continuous
	on $\X\times\R^n$.
\end{lemma}
\begin{proof}
	\color{black}
		Since $\Phi$ is continuously differentiable,
		$(\x,\y)\mapsto(\x,\Phi(\x,\y))$ is locally Lipschitz.
		By Assumption~\ref{assump: foundation}(i), its composition
		with $(\x,\v)\mapsto\Proj_{\Omega(\x)}(\v)$, namely~$\Psi$,
		is locally Lipschitz on $\X\times\R^n$.
\end{proof}

\begin{lemma}\label{lem: y() Lipschitz}
	For any given $\x\in \X$, $\Psi(\x, \cdot)$ has a unique fixed point
	$\y$. Moreover, the solution mapping $y(\cdot)$ is Lipschitz continuous over $\X$.
\end{lemma}
\begin{proof}
	By Assumption~\ref{assump: foundation}(\ref{assump: contraction Psi}) and the Banach fixed-point
	theorem, $\Psi(\x,\cdot)$ has a unique fixed point $y(\x)$
	for each $\x\in\X$.
	Hence we can define a solution mapping
	$y(\cdot):\X\rightarrow\R^n$.

	For any $\x_1,\x_2\in\X$, from
	Assumption~\ref{assump: foundation}(\ref{assump: contraction Psi}),
	\[
	\begin{aligned}
		\|y(\x_1)-y(\x_2)\|
		&=\|\Psi(\x_1,y(\x_1))-\Psi(\x_2,y(\x_2))\|\\
		&\le
		\LL\|y(\x_1)-y(\x_2)\|
		+\|\Psi(\x_1,y(\x_2))-\Psi(\x_2,y(\x_2))\|,
	\end{aligned}
	\]
	which implies that
	\begin{equation}\label{eq: y() Lipschitz}
		\|y(\x_1)-y(\x_2)\|
		\le
		\frac{1}{1-\LL}
		\|\Psi(\x_1,y(\x_2))-\Psi(\x_2,y(\x_2))\|.
	\end{equation}
\color{black}
\color{black}
	Then from Lemma~\ref{lem:proj Omega continuous} and the compactness of $\X$,
	$\{y(\x): \x \in \X\}$ is bounded.
	Hence, there exists a closed ball $\Y$
	such that $\{y(\x): \x \in \X\} \subset \Y \subset\R^n$.
	By Lemma~\ref{lem:proj Omega continuous},
	the compact set $\X\times \Y$ can be covered by finitely many
	neighborhoods on each of which $\Psi$ is Lipschitz.
	Let $L_\Psi>0$ be a common upper bound for these Lipschitz constants.
\color{black}

	For any $\x_1,\x_2\in\X$ and $\y\in \Y$, convexity of $\X\times \Y$ ensures that
	the segment joining $(\x_1,\y)$ and $(\x_2,\y)$ lies in $\X\times \Y$.
	Subdivide this segment into finitely many subsegments, each contained
	in one of the neighborhoods,
	derives that
	\[
	\|\Psi(\x_1,\y)-\Psi(\x_2,\y)\|
	\le L_{\Psi}\|\x_1-\x_2\|.
	\]

	Since $y(\x_2)\in \Y$,
	then~\eqref{eq: y() Lipschitz} yields
	\[
	\|y(\x_1)-y(\x_2)\|
	\le \frac{L_\Psi}{1-L}\|\x_1-\x_2\|,
	\]
	which completes the proof.
\end{proof}

Based on Lemma~\ref{lem:proj Omega continuous} and Lemma~\ref{lem: y() Lipschitz},
we obtain the following theorem.
\begin{theorem}\label{prop:existence of solution}
	The feasible set of
	problem~\eqref{eq:1equation constraint} is bounded and the solution set of problem~\eqref{eq:1equation constraint} is nonempty and bounded.
\end{theorem}
\begin{proof}
	From  Lemma~\ref{lem: y() Lipschitz}, problem~\eqref{eq:1equation constraint} can be reformulated as the following minimization problem
	\begin{equation}\label{hfun}
		\min_{\x\in \X}\,\,	h(\x):= f(\x,y(\x)),
	\end{equation}
	where $y(\x)$ is the fixed point of $\Psi(\x,\cdot)$.

	From the continuity of  $f$,
	the function $h: \X \rightarrow \R$  is continuous over the compact set~$\X$, and hence
	problem~\eqref{hfun} has a solution $\x^*\in \X$. Since $\Psi(\x^*,\cdot)$ is a contraction, it has a unique fixed point $\y^*$, which implies that  $(\x^*,\y^*)$
	is a solution of problem~\eqref{eq:1equation constraint}.

	Moreover,  the Lipschitz continuity of the fixed point mapping $y(\cdot)$ implies that for any feasible point $(\x,\y)$ of problem
	\eqref{eq:1equation constraint}, $\y=y(\x)$  is the unique fixed point of $\Psi(\x, \cdot)$ and there is $L_y$ such that
	$$\|\y\|\le \|\y^*\|+\|y(\x)-y(\x^*)\|\le  \|\y^*\|+L_y\|\x-\x^*\|.$$
	From the boundedness of $\X$, we obtain the boundedness of the feasible set and thus the boundedness of the solution set of problem~\eqref{eq:1equation constraint}.
\end{proof}

\subsection{Optimality condition under metric regularity}\label{sec:exact penalty}

To characterize stationarity for~\eqref{eq:1equation constraint},
we first define its feasible set as
\begin{equation}\label{eq:F def}
	\F:=\left\{
	(\x,\y)\in\X\times\R^n:
	\y=\Psi(\x,\y)
	\right\}.
\end{equation}
\color{black}
Define the set-valued mapping
$\mathcal{H}:\R^m\times\R^n\rightrightarrows\R^n$ by
\[
\mathcal{H}(\x,\y):=
\begin{cases}
	\{\y-\Psi(\x,\y)\}, & \x\in\X,\\[1mm]
	\varnothing, & \x\notin\X.
\end{cases}
\]
Then for any $\zz\in\R^n$,
\[
\mathcal{H}^{-1}(\zz)
=
\left\{
(\x,\y)\in\X\times\R^n:
\y-\Psi(\x,\y)=\zz
\right\}.
\]
In particular, $\mathcal{H}^{-1}(\boldsymbol{0})=\F$.
Next we verify its Metric Regularity (MR)~\cite[Definition~3.1]{mordukhovich2018variational} holds as our Constraint Qualification (CQ) to derive the optimality condition.

\begin{definition}[MR]
	Given $(\x^*,\y^*)\in\F$, the mapping $\mathcal{H}$
	is said to be metrically regular around
	$((\x^*,\y^*),\boldsymbol{0})$, where
	$\boldsymbol{0}\in\R^n$, with modulus $\kappa>0$
	if there exist open neighborhoods $U(\x^*,\y^*) \subset \R^m \times \R^n, U(\boldsymbol{0}) \subset\R^n$ of $(\x^*,\y^*)$  and $\boldsymbol{0}$,
	such that
	\[
	\operatorname{dist}
	\bigl((\x,\y),\mathcal{H}^{-1}(\zz)\bigr)
	\le
	\kappa\,
	\operatorname{dist}
	\bigl(\zz,\mathcal{H}(\x,\y)\bigr)
	\]
	for any $(\x,\y)\in U(\x^*,\y^*)$ and $\zz \in U(\boldsymbol{0})$.
\end{definition}

Now we show that the MR holds automatically for the mapping $\mathcal{H}$.

\begin{proposition}\label{prop: metric regularity}
	The mapping
	$\mathcal{H}$ is metrically regular around
	$((\x^*,\y^*),\boldsymbol{0})$
	for any $(\x^*,\y^*)\in\F$ with modulus $\kappa=\frac{1}{1-\LL}$.
\end{proposition}

\begin{proof}
	Given any $\x \in \X$ and $\zz \in \R^n$,
	by Assumption~\ref{assump: foundation}(\ref{assump: contraction Psi}),
	$\Psi(\x,\cdot)+\zz$ is a contraction over $\R^n$.
	Hence, there exists a unique fixed point $\bar{\y} \in \R^n$ satisfying
	$
	\bar{\y}=\Psi(\x,\bar{\y})+\zz.
	$
	Consequently,
	$(\x,\bar{\y})\in\mathcal{H}^{-1}(\zz)$.
	Since $\X$ is nonempty, $\mathcal{H}^{-1}(\zz)\ne\varnothing$
	for any $\zz\in\R^n$.

	Then from Assumption~\ref{assump: foundation}(\ref{assump: contraction Psi}),
	\[
	\begin{aligned}
		\|\y-\bar{\y}\|
		&=\|\y-\Psi(\x,\bar{\y})-\zz\|\\
		&\le
		\|\y-\Psi(\x,\y)-\zz\|
		+\|\Psi(\x,\y)-\Psi(\x,\bar{\y})\|\\
		&\le
		\|\y-\Psi(\x,\y)-\zz\|
		+\LL\|\y-\bar{\y}\|,
	\end{aligned}
	\]
	and then
	\begin{equation}\label{eq:mr dist}
		\begin{aligned}
			\operatorname{dist}
			\bigl((\x,\y),\mathcal{H}^{-1}(\zz)\bigr)
			&\le
			\|(\x,\y)-(\x,\bar{\y})\|\\
			&\le
			\frac{1}{1-\LL}
			\|\y-\Psi(\x,\y)-\zz\|\\
			&=
			\frac{1}{1-\LL}
			\operatorname{dist}
			\bigl(\zz,\mathcal{H}(\x,\y)\bigr).
		\end{aligned}
	\end{equation}

	If $\x\notin\X$, then
	$\mathcal{H}(\x,\y)=\varnothing$ and
	$\operatorname{dist}(\zz,\mathcal{H}(\x,\y))=+\infty$,
	so~\eqref{eq:mr dist} holds trivially.
	Therefore, MR holds for all
	$(\x,\y)\in\R^m\times\R^n$ and $\zz\in\R^n$.
	In particular, it holds in neighborhoods of
	$(\x^*,\y^*)$ and $\boldsymbol{0}$ for any $(\x^*,\y^*)\in\F$.
\end{proof}
\color{black}

Proposition~\ref{prop: metric regularity} verifies the MR with the moving set $\Omega(\x)$, which serves as an extension to existing works on a fixed set $\Omega(\x)\equiv{\bf \Omega}$ for any $\x\in \X$.
In~\cite{ye2000constraint}, for the MR of the constraint system of~\eqref{eq:1equation constraint} with $\Omega(\x)\equiv{\bf \Omega}$,
the strong regularity of the PVI in the sense of Robinson (see~\cite{robinson1980strongly} and~\cite[Chapter 4.2]{luo1996mathematical}),
is introduced as a sufficient condition for a Lipschitz continuous single-valued solution mapping $y(\cdot)$~\cite[Theorem 2.1, Corollary 2.2]{robinson1980strongly},
and can be guaranteed if $F(\x,\cdot)$ is strongly monotone for any $\x \in \X$.
Specifically,~\cite[Theorem 4.7]{ye2000constraint} proves under the Robinson strong regularity condition, the no nonzero abnormal multiplier CQ (NNAMCQ) holds,
which derives the MR of constraint system~\cite[Remark 2]{ye2000multiplier}.
However,
discussions in~\cite{robinson1980strongly} and~\cite{ye2000constraint} were restricted to $\Omega(\x)\equiv{\bf \Omega}$.
Additionally,~\cite[Chapter 4.4]{luo1996mathematical} introduces the MR with some brief discussion
on its sufficient conditions.
In contrast, our analysis in this work not only covers the discussion from strong monotonicity to the MR in~\cite{ye2000constraint},
but also substantially extends it to~\eqref{eq:1equation constraint} with the PVI on a moving set. To be specific, in Proposition~\ref{prop: contraction Psi}(\ref{prop: contraction Psi strongly monotone}),
the strong monotonicity of $F(\x,\cdot)$ guarantees a contraction $\Psi(\x,\cdot)$ for any $\x \in \X$, then
in Lemma~\ref{lem: y() Lipschitz} the Lipschitz property of the fixed point mapping $\y(\cdot)$ of $\Psi(\x,\cdot)$ is derived,
and in Proposition~\ref{prop: metric regularity} the MR is verified without any additional conditions.

Based on Proposition~\ref{prop: metric regularity}, we can derive the optimality condition of~\eqref{eq:1equation constraint} with the MR as our CQ.
Most literature assumes some CQs on the constraint system to establish optimality conditions, where the MR is a standard assumption~\cite{ye2000multiplier}.
Some CQs on nonlinear constraints are hard to verify.
Fortunately, our blanket Assumption \ref{assump: foundation} is sufficient for
the MR on the constraints of problem~\eqref{eq:1equation constraint}, which is easy to verify.

\begin{definition}[Stationary point of problem \eqref{eq:1equation constraint}]\label{def: stationary point}
	We call $(\x^*,\y^*)$ a stationary point of~\eqref{eq:1equation constraint}
	if $(\x^*,\y^*) \in \F$ and
	there exist a multiplier $\lambda^* \in \R^n$ and a matrix $\Gamma \in \partial_C \Psi(\x^*,\y^*)$ such that
	\begin{equation}\label{eq:stationary point}
		\boldsymbol{\mathrm{0}} \in \nabla f(\x^*,\y^*) +\left(
		\begin{aligned}
			& \boldsymbol{\mathrm{0}} \\& \lambda^*
		\end{aligned}
		\right) - \Gamma \lambda^* + \left(\begin{aligned}
			&\N_{\X}(\x^*) \\&\quad \{\boldsymbol{\mathrm{0}}\}
		\end{aligned}\right).
	\end{equation}
\end{definition}

\begin{theorem}[Optimality condition under the MR] \label{thm: stationary point}
	If $(\x^*,\y^*)$ is a local minimizer of problem~\eqref{eq:1equation constraint}, then it is a stationary point of~\eqref{eq:1equation constraint}.
\end{theorem}
\begin{proof}
{\color{black}
By Lemma~\ref{lem:proj Omega continuous} and Assumption~\ref{assump: foundation}(\ref{assump: contraction Psi}),
$\Psi$ is jointly locally
Lipschitz on $\X\times\R^n$.
}
	\color{black}
	From Proposition~\ref{prop: metric regularity},~\cite[Corollary~2.3.2]{burke1991exact} and~\cite[Theorem~6.1.1, Remark~6.1.2, Proposition~6.4.4]{clarke1990optimization},
there exists
	a multiplier $\lambda^*\in\R^n$ such that
	\[
	\boldsymbol{\mathrm{0}} \in \nabla f(\x^*,\y^*) +  \sum_{i=1}^{n} \lambda_i^* \partial \left( \y_i^* - \Psi_i(\x^*,\y^*) \right) + \N_{\X \times \R^n}(\x^*,\y^*). \]
\color{black}
	By~\cite[Proposition 1.4]{mordukhovich2018variational},
	$\N_{\X \times \R^n}(\x^*,\y^*)  = \N_{\X}(\x^*) \times \{\boldsymbol{\mathrm{0}}\}$.
	According to~\cite[Proposition 2.3.1, 2.3.3, Corollary 1]{clarke1990optimization},
	$$\partial\big(\y_i^* - \Psi_i(\x^*,\y^*) \big) = \left(\begin{aligned}
		& \quad \boldsymbol{\mathrm{0}} \\& \nabla_{\y} \y_i^*
	\end{aligned}\right)
	+ \partial \left( - \Psi_i(\x^*,\y^*) \right)
	= \left(\begin{aligned}
		& \quad \boldsymbol{\mathrm{0}} \\& \nabla_{\y} \y_i^*
	\end{aligned}\right) - \partial \Psi_i(\x^*,\y^*),
	$$
	where $\nabla_{\y} \y_i^*$ is a vector whose i-th component is one and others are zero.
	Then by
	\[
	\begin{aligned}
		\sum_{i=1}^{n} \lambda_i^* \left( \left(\begin{aligned}
			& \quad \boldsymbol{\mathrm{0}} \\& \nabla_{\y} \y_i^*
		\end{aligned}\right)
		- \partial \Psi_i(\x^*,\y^*) \right)
		= \left\{  \left(
		\begin{aligned}
			& \boldsymbol{\mathrm{0}} \\& \lambda^*
		\end{aligned}
		\right) - \Gamma \lambda^*: \Gamma \in \partial_C \Psi(\x^*,\y^*) \right\},
	\end{aligned}
	\]
	we can conclude there exists $\Gamma \in \partial_C \Psi(\x^*,\y^*)$, such that (\ref{eq:stationary point}) holds.
\end{proof}

\begin{remark}\label{remark:cq}
	It is worth noting that a multiplier $\lambda^* \in \R^n$ always exists such that the first-order optimality condition~\eqref{eq:stationary point} holds.
	However, some usual CQ assumptions may fail.

	For instance,
	suppose $\X = \left\{ \x: \gg(\x) \le \boldsymbol{\mathrm{0}} \right\}$ where $\gg:\R^2 \rightarrow \R^3$ is defined as follows: \[
	\gg_1(\x) = \x_1^3 - \x_2,\,\, \gg_2(\x) = \x_2,\,\, \gg_3(\x) = - \x_1^3 - \x_2.
	\]
	Let
	$F: \R^2 \times \R \rightarrow \R$
	be defined as $F(\x,\yy) = \x_1 \x_2 + \yy$
	and $\Omega(\x) = [-1 + \x_1,  1 + \x_1]$.
	Then the feasible set is
	$$
	\begin{aligned}
		\F = \Big\{ (\x,\yy):
		g(\x) \le \boldsymbol{\mathrm{0}},
		\yy= \mid\big\{ -1 + \x_1,  1 + \x_1, \yy - \delta( \x_1 \x_2 + \yy) \big\}\Big\}.
	\end{aligned}
	$$
	When $0 < \delta < 2$, all constraints are active at the only feasible point $(\x^*,\yy^*) = (0,0,0)$.
	It is also easy to check that all constraints are continuously differentiable at the feasible point,
	the corresponding equality constraint is
	\[
	\yy - \left( \yy - \delta( \x_1 \x_2 + \yy) \right) = 0,
	\]
	and the corresponding gradient matrix is
	\[
	\left(\begin{aligned}
		\ 0 && -1 && 0 \ \\
		\ 0 && 1 && 0 \ \\
		\ 0 && -1 && 0\ \\
		\ 0 && 0 && \delta \
	\end{aligned}\right).
	\]

	If we consider the feasible set as
	$\F = \left\{ (\x,\yy): \gg(\x) \le \boldsymbol{\mathrm{0}}, \yy = \Psi(\x,\yy) \right\}$,
	the usual Linear Independence CQ (LICQ), Mangasarian-Fromovitz CQ (MFCQ), and Constant Rank CQ (CRCQ)~\cite{facchinei2003finite} do not hold at the feasible point.
	To be specific,
	since the rows of the gradient matrix are linearly dependent, LICQ does not hold.
	For MFCQ to hold, a strict feasible direction $\dd \in \R^3$
	should satisfy $\nabla \gg_j(\x^*,\yy^*) \dd < 0$  for $j = 1,2,3$,
	deriving $- \dd_2 < 0$ and $\dd_2 < 0$.
	However, such a direction does not exist, which violates MFCQ.
	As for CRCQ, the rank of $\left\{\nabla \gg_j (\x,\yy)\right\}_{j \in \{1,2\}}$
	at $(\x^*,\yy^*)$ is $1$,
	while the rank at $(\omega,0,0)$ is $2$ for any sufficiently small $\omega > 0$,
	according to $\nabla \gg_1 (\omega,0,0) = (3 \omega^2, -1,0)$
	and $\nabla \gg_2 (\omega,0,0) = (0, 1,0)$,
	so the desired neighborhood of $(\x^*,\yy^*)$ for CRCQ does not exist.

	Contrarily, if we consider the feasible set as
	$\F = \left\{ (\x,\yy): \x \in \X, \yy = \Psi(\x,\yy) \right\}$
	where $\X = \left\{ \x: \gg(\x) \le \boldsymbol{\mathrm{0}} \right\} = \{(0,0)\}$,
	the MR holds automatically as our CQ from the contraction property of  $\Psi(\x,\cdot)$.
	This demonstrates the superiority of our characterization of the feasible set with a set-valued mapping.
	Furthermore, it is worth noting that under our model, $\X$ is not necessarily defined by functions.
\end{remark}

\section{Smoothing approximation}\label{sec: smoothing}

Smoothing methods for VI$(\Omega(\x),F(\x,\cdot))$ with $\Omega(\x)\equiv {\bf \Omega}$ can be found in existing literature, for example \cite{chen2012smoothing,chen1998global}.
To address the nonsmoothness of $\Psi$ from the projection operator on a moving set
$\Omega(\x),$
which introduces nonsmoothness of the solution mapping $\y(\cdot)$,
we consider a smoothing approximation~$\Psi_{\mu}$ of $\Psi$
and examine its impact on solution properties and consistency.

\begin{definition}[Smoothing approximation] \label{definition:smoothing}
	For any $\mu \in (0,1]$,
	we define\\
	$\Psi_{\mu}:~\textcolor{black}{U_{\X}} \times~\R^n  \rightarrow~\R^n$ a smoothing approximation of $\Psi$,
	if it satisfies
	\begin{enumerate}[(i)]
		\item \label{assump: Psi mu continuously differentiable}
		{\color{black}
$\Psi_{\mu}(\cdot,\cdot)$ is continuously differentiable on $\textcolor{black}{U_{\X}} \times \R^n$,}
		\textcolor{black}{and the mapping $(\x,\y,\mu)\mapsto\Psi_\mu(\x,\y)$ is continuous on $\mathcal X\times\mathbb R^n\times(0,1]$.}

		\item \label{assump: Psi_mu limit}
		$\lim\limits_{\widetilde{\x} \rightarrow \x, \widetilde{\y} \rightarrow \y, \mu \downarrow 0} \Psi_{\mu}(\widetilde{\x},\widetilde{\y}) = \Psi(\x,\y)$,
		for any $\x \in \X$, $\y \in \R^n$; \item \label{assump: Psi_mu contraction and Lip}
		$\Psi_{\mu}(\x,\cdot)$ is $L_s$-Lipschitz over $\R^n$ for any $\x \in \X$ with $L_s\in (0,1)$. \end{enumerate}

\end{definition}

If for any $\x \in \X$, $\Phi(\x,\cdot)$ is a contraction,
and the smoothing approximation preserves the nonexpansiveness of the projection operator,
then similar to Proposition~\ref{prop: contraction Psi}, $\Psi_{\mu}(\x,\cdot)$ as a composition remains a contraction,
which guarantees Definition~\ref{definition:smoothing}(\ref{assump: Psi_mu contraction and Lip}).
We show that a class of smoothing functions satisfies  Definition~\ref{definition:smoothing}
for $\Omega(\x) = \{ \y \in \R^n: l(\x) \le \y \le u(\x)\}$,
in Lemma~\ref{lem: smooth Psi} of Appendix~\ref{sec: smooth box}.

From the Banach fixed-point theorem, for any $\mu \in (0,1]$ and $\x \in \X$,
the smoothed PVI problem $ \y = \Psi_{\mu}(\x,\y)$
admits a unique solution $y_{\mu} (\x)$, leading to the well-defined mapping $y_{\mu}: \X \rightarrow \R^n$.
For any $\x \in \X, \y \in \R^n, \mu \in (0,1]$, denote
\begin{equation*} H_{\mu}(\x,\y):= \y - \Psi_{\mu}(\x,\y).
\end{equation*}
From Definition~\ref{definition:smoothing}(\ref{assump: Psi_mu contraction and Lip}), all singular values of $\nabla_{\y}H_{\mu}(\x,\y) = I - \nabla_{\y} \Psi_{\mu}(\x,\y)$ have a uniform lower bound $1 - L_s > 0$.
Hence, we have the following lemma as the foundation for our analysis related to the smoothing approximation.
\begin{lemma}\label{lem: nonsingular}
$\nabla_{\y}H_{\mu}(\x,\y)$ is nonsingular for any $\x\in\X$, $\y\in\R^n$, $\mu\in(0,1]$, and
\[
\|\nabla_{\y}H_{\mu}(\x,\y)^{-1}\|\le\frac{1}{1-L_s}.
\]
 \end{lemma}

Based on Lemma~\ref{lem: nonsingular} and the continuous differentiability of $H_{\mu}$,
the implicit function theorem can be applied to obtain the following proposition.

\begin{proposition}[Continuous differentiability of $y_{\mu}(\cdot)$] \label{prop:differentiability of y mu x}
	For any fixed $\mu \in (0,1] $, there exists a unique continuously differentiable function $y_{\mu}(\cdot)$ over $\X$, such that for any $\x \in \X$,
	\begin{equation}H_{\mu}(\x,y_{\mu}(\x)) = \boldsymbol{\mathrm{0}}
		\ \text{ and } \
		\nabla y_{\mu}(\x) = - \left(\nabla_{\y} H_{\mu}(\x,y_{\mu}(\x))\right)^{-1}  \nabla_{\x} H_{\mu}(\x,y_{\mu}(\x)).
	\end{equation}
\end{proposition}

We now examine the consistency between problem~\eqref{eq:1equation constraint} and the smoothed problem~\eqref{eq:1equation constraint smoothing}.  By the uniqueness of $y_{\mu}(\x)$ for any $\x \in \X$ and $\mu \in (0,1]$,~\eqref{eq:1equation constraint smoothing} can be reformulated as
\begin{equation}\label{shfun} \min_{\x\in \X}\,\,	h_\mu(\x):= f(\x,y_\mu(\x)),
\end{equation}
where $y_\mu(\x)$ is the fixed point of $\Psi_\mu(\x,\cdot)$.
The function $h_{\mu}:\X \rightarrow \R$ is smooth from Proposition~\ref{prop:differentiability of y mu x}. Moreover, similar to Theorem~\ref{prop:existence of solution}, the feasible set
\begin{equation}\label{eq:sF}
	\F_\mu:= \left\{ (\x,\y): \x \in \X, \y = \Psi_\mu(\x,\y) \right\}
	= \left\{ (\x,y_{\mu}(\x)): \x \in \X \right\}
\end{equation}
of~(\ref{eq:1equation constraint smoothing}) is bounded and  the solution set of~(\ref{eq:1equation constraint smoothing}) is nonempty and bounded, for any $\mu \in (0,1]$.
In Lemma~\ref{lem:y mu x}(\ref{lem: y mu x bounded}) we further prove that the sets $\F_\mu$ are uniformly bounded for all $\mu \in (0,1]$.
By chain rule and Proposition~\ref{prop:differentiability of y mu x} with the help of implicit function theorem,
we can obtain the gradient of $h_{\mu}$ as follows:
\begin{equation} \label{eq: nabla h mu}
	\begin{split}
		\nabla h_{\mu}(\x)
		& = \nabla_{\x} f(\x,y_{\mu}(\x)) + (\nabla y_{\mu}(\x))^\top \nabla_{\y} f(\x,y_{\mu}(\x)) \\
		& = \nabla_{\x} f(\x,y_{\mu}(\x)) - \left(\nabla_{\x} H_{\mu}(\x,y_{\mu}(\x))\right)^\top
		\left(\nabla_{\y} H_{\mu}(\x,y_{\mu}(\x))\right)^{-\top}
		\nabla_{\y} f(\x,y_{\mu}(\x)),
	\end{split}
\end{equation}
where $\nabla_{\x} H_{\mu}(\x,\y) = - \nabla_{\x}  \Psi_{\mu}(\x,\y)$,
$\nabla_{\y} H_{\mu}(\x,\y) = I - \nabla_{\y}  \Psi_{\mu}(\x,\y)$.~\eqref{eq: nabla h mu} serves as the foundation of our gradient-based algorithmic framework in Section~\ref{sec: SIGA}.

\begin{lemma}\label{lem:y mu x}
	\begin{enumerate}[(i)]
		\item \label{lem: y mu x bounded}
		The sets $\F_\mu$ in~\eqref{eq:sF} are uniformly bounded for all $\mu \in (0,1]$.

		\item \label{lem:y mu x continuous in mu and x}
		For any $\left\{\mu_t\right\} \subset (0,1]$, $\left\{\x^t\right\} \subset \X$, and $\bar{\x} \in \X$,
		$$
		\lim\limits_{\mu_t \downarrow 0, \x^t \rightarrow \bar{\x}} y_{\mu_t}(\x^t) = y(\bar{\x}).
		$$ \end{enumerate}
\end{lemma}
\begin{proof}
(i)
	For any $\x \in \X$ and $\mu \in (0,1]$,
	from Definition~\ref{definition:smoothing}(\ref{assump: Psi_mu contraction and Lip}),
	\[
	\begin{aligned}
		\|y_{\mu}(\x) - y(\x)\| & = \left\| \Psi_{\mu}(\x, y_{\mu}(\x) ) - \Psi(\x, y(\x)) \right\| \\
		 &\le L_s \left\|y_{\mu}(\x) - y(\x) \right\|
		+ \| \Psi_{\mu}(\x, y(\x)) - \Psi(\x, y(\x)) \|,
	\end{aligned}
	\]
	which implies that
	\begin{equation}\label{eq: y mu x norm}
		\|y_{\mu}(\x) - y(\x)\|
		\le 1/(1 - L_s) \| \Psi_{\mu}(\x, y(\x)) - \Psi(\x, y(\x)) \|.
	\end{equation}
\color{black}
	By Lemma~\ref{lem: y() Lipschitz}, $y(\cdot)$ is continuous and bounded on the compact set $\mathcal X$. The continuous convergence condition in Definition~\ref{definition:smoothing}(\ref{assump: Psi_mu limit}) and the compactness of $\mathcal X$ imply
	\[
	\lim_{\mu\downarrow0}\sup_{\x\in\mathcal X}
	\|\Psi_\mu(\x,y(\x))-\Psi(\x,y(\x))\|=0.
	\]
	Hence, there exists $\bar\mu\in(0,1)$ such that the right hand side of~\eqref{eq: y mu x norm}
	is uniformly bounded for any $\x\in\mathcal X$ and $\mu\in(0,\bar\mu]$.
	For $\mu\in[\bar\mu,1]$, the joint continuity in
	Definition~3.1(i) and the continuity of $y(\cdot)$ imply that
	$(\x,\mu)\mapsto\Psi_\mu(\x,y(\x))-y(\x)$ is continuous on the
	compact set $\mathcal X\times[\bar\mu,1]$.
	Since $\Psi(\x,y(\x))=y(\x)$, it follows that
	\[
	\sup_{\substack{x\in\mathcal X\\0<\mu\le1}}
	\|\Psi_\mu(\x,y(\x))-\Psi(\x,y(\x))\|<\infty,
	\]
	indicating that the right-hand side of~\eqref{eq: y mu x norm} is uniformly bounded over $\x \in \X$ and $\mu \in (0,1]$.
	Then together with the boundedness of $y(\x)$,
	we obtain uniform boundedness of $y_\mu(\x)$ over
	$\x\in\mathcal X$ and $\mu\in(0,1]$,
	which derives~(\ref{lem: y mu x bounded}).

\color{black}
(ii)
	Next, considering $\mu_t \downarrow 0$ and $\{\x^t\} \subset \X$,
	from~\eqref{eq: y mu x norm} we have
	\[
	\|y_{\mu_t}(\x^t) - y(\x^t)\|
	\le 1/(1 - L_s) \| \Psi_{\mu_t}(\x^t, y(\x^t)) - \Psi(\x^t, y(\x^t)) \| \rightarrow 0,
	\]
	by Definition~\ref{definition:smoothing}(\ref{assump: Psi_mu limit}).
	Then
	\[
	\begin{aligned}
		& \|y_{\mu_t}(\x^t) - y(\bar{\x})\|
		\le \|y_{\mu_t}(\x^t) - y(\x^t)\| + \|y(\x^t) - y(\bar{\x})\|
		\rightarrow 0
	\end{aligned}
	\]
	as $\x^t \rightarrow \bar{\x}$ for any $\bar{\x} \in \X$.
\end{proof}

\color{black}
\begin{proposition}[Solution consistency]
	\label{lem:global solution relationship}
	Let $\mathcal S^*$ and $\mathcal S^*_\mu$ be the solution sets
	of~\eqref{eq:1equation constraint}
	and~\eqref{eq:1equation constraint smoothing} with $\mu \in (0,1]$, respectively.
	Then we have
	\[
	\left\{
	\lim\limits_{\substack{
			(\x^\mu,\y^\mu)\in\mathcal S^*_\mu,\
			\mu\downarrow0}}
	(\x^\mu,\y^\mu)
	\right\}
	\subset \mathcal S^*,
	\]
	which means that, for any sequence $\mu_t\downarrow0$ and any choice
	$(\x^{\mu_t},\y^{\mu_t})\in\mathcal S^*_{\mu_t}$,
	every accumulation point of
	$\{(\x^{\mu_t},\y^{\mu_t})\}$ belongs to $\mathcal S^*$.
\end{proposition}

\begin{proof}
	Choose any sequence $\mu_t\downarrow0$ with $\mu_t\in(0,1]$
	and $(\x^{\mu_t},\y^{\mu_t})\in\mathcal S^*_{\mu_t}$ for each~$t$.
	The sequence is bounded by
	Lemma~\ref{lem:y mu x}(\ref{lem: y mu x bounded}).
	Without loss of generality, we assume
	$\x^{\mu_t}\rightarrow\bar{\x}$, by passing to an arbitrary
	convergent subsequence.
	Then $\bar{\x}\in\X$ by the closedness of~$\X$.

	Choose $(\x^*,\y^*)\in\mathcal S^*$.
	From $(\x^{\mu_t},\y^{\mu_t})\in\mathcal S^*_{\mu_t}$
	and~\eqref{eq:1equation constraint smoothing} being reformulated as~\eqref{shfun},
	it holds that
	$
	h_{\mu_t}(\x^{\mu_t})\le h_{\mu_t}(\x^*),
	$
	and then taking the limit $t\rightarrow+\infty$ yields
	\[
	h(\bar{\x})
	=
	\lim\limits_{\substack{
			\mu_t\downarrow0,
			\x^{\mu_t}\rightarrow\bar{\x}}}
	h_{\mu_t}(\x^{\mu_t})
	\le
	\lim\limits_{\mu_t\downarrow0}h_{\mu_t}(\x^*)
	=
	h(\x^*),
	\]
	where the equalities hold based on the continuity of~$f$
	and Lemma~\ref{lem:y mu x}(\ref{lem:y mu x continuous in mu and x}).
	Along with $\bar{\x}\in\X$ and
	$(\x^*,\y^*)\in\mathcal S^*$,
	we have $h(\bar{\x})=h(\x^*)$,
	and thus $(\bar{\x},y(\bar{\x}))\in\mathcal S^*$.

	Note that $(\x^\mu,\y^\mu)\in\mathcal S^*_\mu$ implies
	$\y^\mu=y_\mu(\x^\mu)$.
	Taking the limit $t\rightarrow+\infty$ gives
	\[
	\y^{\mu_t}=y_{\mu_t}(\x^{\mu_t})
	\longrightarrow y(\bar{\x})
	\]
	by Lemma~\ref{lem:y mu x}(\ref{lem:y mu x continuous in mu and x}) again.
	Since the convergent subsequence was arbitrary,
	any accumulation point of
	$\{(\x^{\mu_t},\y^{\mu_t})\}$ belongs to~$\mathcal S^*$.
\end{proof}
Next we define a stationary point of~\eqref{eq:1equation constraint smoothing} and consider its relationship with a stationary point of~\eqref{eq:1equation constraint}.
\begin{definition}[First-order stationary point of~\eqref{eq:1equation constraint smoothing}]\label{def: sta 0}
	We call $(\x,\y)$  a stationary point of~\eqref{eq:1equation constraint smoothing}
	if $(\x,\y) \in \F_\mu$ and there exists a multiplier $\lambda \in \R^n$ such that
	\begin{equation}\label{smoothSt}
		\left\{\begin{aligned}
			&\boldsymbol{\mathrm{0}} \in \nabla_{\x} f(\x,\y) - \nabla_{\x} \Psi_{\mu}(\x,\y)^\top \lambda + \N_{\X}(\x), \\
			& \boldsymbol{\mathrm{0}} = \nabla_{\y} f(\x,\y) + \lambda - \nabla_{\y} \Psi_{\mu}(\x,\y)^\top \lambda.
		\end{aligned}
		\right.
	\end{equation}
\end{definition}

\begin{assumption}\label{assumption: jacobian consistency + Lipschitz in mu}

	Suppose for any sequence $\{(\x^t,\y^t)\}\subset \X \times \R^n$
	and $(\x,\y) \in \X \times \R^n$,
	$$
	\lim\limits_{\x^t \rightarrow \x, \y^t \rightarrow \y, \mu_t \downarrow 0} {\rm dist} \left( \nabla \Psi_{\mu_t} (\x^t,\y^t)^\top, \partial_C \Psi (\x,\y) \right) = 0.
	$$
\end{assumption}
In Lemma~\ref{lem: smooth Psi 2} of Appendix~\ref{sec: smooth box}, we show that Assumption \ref{assumption: jacobian consistency + Lipschitz in mu} holds for some specific smoothing function $\Psi_{\mu} $ with  $\Omega(\x) = \{ \y \in \R^n: l(\x) \le \y \le u(\x)\},$
where
$l,u: \textcolor{black}{\R^m} \rightarrow \R^n$ are continuously differentiable.

\begin{proposition}[Stationarity consistency]
	\label{prop: stationarity smoothing}
	Let $\mu_t\downarrow0$ with $\mu_t\in(0,1]$, and let
	$(\x^t,\y^t)$ be a stationary point
	of~\eqref{eq:1equation constraint smoothing}
	with $\mu=\mu_t$ for every $t$.
{\color{black}
	Then $\{(\x^t,\y^t)\}$ and the corresponding multiplier
	sequence are bounded, and every accumulation point of
	$\{(\x^t,\y^t)\}$ is feasible
	for~\eqref{eq:1equation constraint}.
Let $(\x^*,\y^*)$ be such an accumulation point. If the
Jacobian-consistency condition in
Assumption~\ref{assumption: jacobian consistency + Lipschitz in mu}
holds at $(\x^*,\y^*)$, then $(\x^*,\y^*)$ is a stationary point
of~\eqref{eq:1equation constraint}.
\par}
\end{proposition}

\begin{proof}
	As a stationary point of~\eqref{eq:1equation constraint smoothing} with $\mu=\mu_t$,
	$(\x^t,\y^t)$ satisfies
	$\x^t \in \X$, $\y^t = \Psi_{\mu_t}(\x^t,\y^t)$, and the sequence $(\x^t,\y^t)\in \F_{\mu_t}$ is uniformly bounded for any $t$ from Lemma~\ref{lem:y mu x}(\ref{lem: y mu x bounded}).
	From Lemma~\ref{lem: nonsingular}, $\nabla_{\y} H_{\mu_t}(\x^t,\y^t)$ is nonsingular,
	and the MR holds for the smoothed constraint system similar to Proposition~\ref{prop: metric regularity}. Hence,  there exists $\lambda^t \in \R^n$ such that
	\begin{equation}\label{eq:KKT smoothing}
		\boldsymbol{\mathrm{0}} \in \nabla f(\x^t,\y^t) +\left(
		\begin{aligned}
			&\boldsymbol{\mathrm{0}} \\& \lambda^t
		\end{aligned}
		\right) - \nabla\Psi_{\mu_t}(\x^t,\y^t)^\top \lambda^t + \left(
		\begin{aligned}
			&\N_{\X}(\x^t) \\&\quad \{\boldsymbol{\mathrm{0}}\}
		\end{aligned}\right).
	\end{equation}

From Definition~\ref{definition:smoothing}(\ref{assump: Psi_mu contraction and Lip}), $\|\nabla_{\y} \Psi_{\mu_t}(\x^t,\y^t)\|\le L_s$.
From~\eqref{eq:KKT smoothing},
	\begin{displaymath}
	\| \lambda^t \|	= \| \left(I - \nabla_{\y} \Psi_{\mu_t}(\x^t,\y^t) \right)^{-\top} \nabla_{\y} f(\x^t,\y^t)\|
	\le 1/(1 - L_s)\|\nabla_{\y} f(\x^t,\y^t)\|.
	\end{displaymath}
 Then from the continuity of $\nabla_{\y} f$ and the boundedness of $\{(\x^t,\y^t)\}$, there exists $\beta>0$ independent of $(\x,\y,\lambda)$ such that
\begin{equation}\label{lambda}
\|\lambda^t\|\le\beta,\qquad\forall\mu_t\in(0,1].
\end{equation}

\color{black}
Therefore, without loss of generality, we assume $(\x^t,\y^t,\lambda^t) \rightarrow (\x^*,\y^*,\lambda^*)$.
By Assumption~\ref{assumption: jacobian consistency + Lipschitz in mu}
and the compactness of $\partial_C\Psi(\x^*,\y^*)$,
we may pass to a further subsequence such that
\[
\nabla\Psi_{\mu_t}(\x^t,\y^t)^\top
\to\Gamma^*\in\partial_C\Psi(\x^*,\y^*).
\]
Lemma~\ref{lem:y mu x}(\ref{lem:y mu x continuous in mu and x})
and the closedness of $\X$ imply that $(\x^*,\y^*)$ is feasible.
Taking the limit in~\eqref{eq:KKT smoothing}, using the closedness of the graph of
$\N_{\X}$, it derives that $(\x^*,\y^*)$ satisfies Definition~\ref{def: stationary point}.
\end{proof}

Now we consider $\epsilon$ stationary points of problem~\eqref{eq:1equation constraint}. Let $H(\x,\y)=\y - \Psi(\x,\y)$.
\begin{definition}[An $\epsilon$ stationary point of~\eqref{eq:1equation constraint}]\label{epsilon-stationary point}
We call $(\x^*,\y^*)$ an $\epsilon$ stationary point of~\eqref{eq:1equation constraint}
	if $\x^*\in \X$ and
	there exist  $\lambda \in \R^n$, $\Gamma_\x \in \R^{m\times n}, \Gamma_\y\in \R^{n\times n}$,
	such that $\Gamma(\x^*,\y^*):=\left(\begin{array}{l}
		\Gamma_\x\\		\Gamma_\y
	\end{array}\right) \in \partial_C \Psi(\x^*,\y^*)$ and
	\begin{equation}\label{ep-stationary point}
		\begin{aligned}
			{\rm dist} \left(\boldsymbol{\mathrm{0}}, \nabla_\x f(\x^*,\y^*) -\Gamma_\x\lambda +\N_{\X}(\x^*)\right) & \le \epsilon,\\
			\|\nabla_\y f(\x^*,\y^*)+	(I - \Gamma_\y) \lambda\| & \le \epsilon,\\
			\|H(\x^*,\y^*)\| & \le \epsilon.
		\end{aligned}
	\end{equation}
\end{definition}

In the following theorem, we characterize the relationship between a stationary point of~\eqref{eq:1equation constraint smoothing} and an $\epsilon$-stationary point of~\eqref{eq:1equation constraint}.
\begin{theorem}\label{thm: eps stationary}

	Let $(\x_\mu,\y_\mu) \in \X \times \R^n$ be a stationary point of~\eqref{eq:1equation constraint smoothing}.
	Under Assumption~\ref{assumption: jacobian consistency + Lipschitz in mu},
	if
	\begin{equation}\label{smmoothstationary}
		\|H(\x_\mu, \y_\mu)\|\le \epsilon \,\, \, \, {\rm and} \,\, \,\,  {\rm dist} (\nabla\Psi_\mu(\x_\mu, \y_\mu)^\top,
		\partial_C \Psi(\x_\mu,\y_\mu))\le \epsilon/\beta,
	\end{equation}
{\color{black}
where $\beta$ is a positive constant in (\ref{lambda}), then
	$(\x_\mu, \y_\mu)$ is an $\epsilon$-stationary point of~\eqref{eq:1equation constraint}.
\par}
\end{theorem}
\begin{proof}

Let $\Gamma(\x_\mu,\y_\mu)\in \partial_C \Psi(\x_\mu,\y_\mu)$ 		such that
$${\rm dist} (\nabla\Psi_\mu(\x_\mu, \y_\mu)^\top,
	\partial_C \Psi(\x_\mu,\y_\mu))=\|\nabla\Psi_\mu(\x_\mu, \y_\mu)^\top-\Gamma(\x_\mu,\y_\mu)\|.$$
	Let $\lambda_\mu$ be the corresponding multiplier of $(\x_\mu,\y_\mu)$.
	We verify the first inequality and the second inequality in (\ref{ep-stationary point}) as follows:
\[
	\begin{aligned}
		{\rm dist} \left( \boldsymbol{\mathrm{0}}, \nabla_{\x} f(\x_\mu,\y_\mu) - \Gamma_{\x_\mu}\lambda_\mu + \N_{\X}(\x_\mu) \right)&\le  \left\| \left( \nabla_\x \Psi_\mu(\x_\mu, \y_\mu)^\top- \Gamma_{\x_\mu}\right)\lambda_\mu\right\|\\
		&\le \| \nabla_\x \Psi_\mu(\x_\mu, \y_\mu)^\top- \Gamma_{\x_\mu}\|\beta \, \le \ \epsilon
	\end{aligned}
	\]
	and
\[
	\begin{aligned}
		\|\nabla_\y f(\x_\mu,\y_\mu)+	(I - \Gamma_{\y_\mu} )\lambda_\mu\|
		& = \|(\nabla_\y \Psi_\mu(\x_\mu,\y_\mu)^\top- \Gamma_{\y_\mu}) \lambda_\mu\|\\
		& \le \| \nabla_\y \Psi_\mu(\x_\mu, \y_\mu)^\top- \Gamma_{\y_\mu}\|\beta \, \le \, \epsilon,
	\end{aligned}
	\]
	which completes the proof.
\end{proof}

The second system of equations in (\ref{smoothSt}) is linear in $\lambda$. Given $(\x,\y)$, the matrix $\nabla_\y H_\mu(\x,\y)$ is nonsingular and
$\lambda= -\nabla_\y H_\mu(\x,\y)^{-\top}\nabla_\y f(\x,\y)$.
{\color{black}
Instead of the three conditions in~(\ref{ep-stationary point}), we can also consider an $\epsilon$ stationary point of~\eqref{eq:1equation constraint} by the following two conditions
\begin{equation}\label{eq:epsilon-reduced}
	\left\{\begin{aligned}
		& {\rm dist} \left( \boldsymbol{\mathrm{0}}, \nabla_{\x} f(\x,\y) + \Gamma_{\x}(I-\Gamma_{\y})^{-1}\nabla_{\y} f(\x,\y) + \N_{\X}(\x) \right) \le \epsilon, \\
		& \|H(\x,\y)\| \le \epsilon.
	\end{aligned}
	\right.
\end{equation}
\par}

\section{SIGA: algorithmic framework}\label{sec: SIGA}

This section introduces the Smoothing Implicit Gradient Algorithm (SIGA) and proves its convergence  to a stationary point of problem \eqref{eq:1equation constraint}.
The core of SIGA relies on the computation of implicit gradient of $h_\mu(\x)$ in~\eqref{eq: nabla h mu}.
We denote $\vv_{\mu}(\x)$ as the solution to the linear system
\begin{equation}\label{eq:v mu x}
	\nabla_{\y} f(\x,\yy_{\mu}(\x))  +  \left(\nabla_{\y} H_{\mu}(\x,\yy_{\mu}(\x))\right)^\top \vv_{\mu}(\x) = \boldsymbol{\mathrm{0}},
\end{equation}
and then from~\eqref{eq: nabla h mu},
\begin{equation}  \label{eq:nabla h_mu(x) 2}  \begin{split}
		\nabla h_{\mu}(\x)
		& = \nabla_{\x} f(\x,\yy_{\mu}(\x)) + \left(\nabla_{\x} H_{\mu}(\x,\yy_{\mu}(\x))\right)^\top \vv_{\mu}(\x).
	\end{split}
\end{equation}
Indeed, $\vv_{\mu}(\x)$ corresponds to the multiplier in~\eqref{eq:KKT smoothing}.
From Lemma~\ref{lem: nonsingular},\\
the matrix $\nabla_{\y} H_{\mu}(\x,\yy_{\mu}(\x))$ is invertible,
so the linear system has a unique solution $\vv_{\mu}(\x)$.

\begin{algorithm} \caption{Smoothing Implicit Gradient Algorithm (SIGA)} \label{alg}
	\algsetup{indent=2em} \begin{algorithmic}[1]
		\STATE
		\textbf{Input:}  $\x^1\in\X$,
		$\mu_0 \in (0,1], \zeta_0 > 0, \tau_0 > 0, p \in \left(0,\frac{1}{4}\right)$.

		\FOR {$t = 1,2,\cdots,$}

		\STATE
		Set
		\begin{equation}\label{eq: alg parameter}
			\mu_t = \frac{\mu_0}{t^p},\,\,
			\zeta_t = \frac{\zeta_0}{t^{2p}},\,\,
			\tau_t = \frac{\tau_0}{t}.
		\end{equation}

		\STATE Find $\y^t$ such that
		\begin{equation}\label{eq:algorithm sub-step 1}
			\| \y^t - \Psi_{{\mu}_t}(\x^{t},\y^t) \| \le \tau_t.
		\end{equation}

		\STATE Set
		\begin{equation}\label{eq:algorithm sub-step 2}
			\v^t = \arg\min_{\v \in \R^n} \left\| \nabla_{\y} f(\x^t,\y^t)  +  \left(
			\nabla_{\y} H_{\mu_t}(\x^t,\y^t)
			\right)^\top \v  \right\|^2.
		\end{equation}

		\STATE
		Compute
		\begin{equation}\label{eq:x iteration}
			\begin{aligned}
				d_{\x}^t &= \nabla_{\x} f(\x^t,\y^t) + \left(
				\nabla_{\x} H_{\mu_t}(\x^t,\y^t)
				\right)^\top \v^t, \\
				\x^{t+1} &= \Proj_{\X} \left( \x^t - \zeta_t d_{\x}^t \right).
			\end{aligned}
		\end{equation}

		\ENDFOR
	\end{algorithmic}
\end{algorithm}

At each step of Algorithm~\ref{alg}, with fixed $(\mu_t, \zeta_t, \tau_t)$, we first compute the approximate vectors $\y^t$ and $\v^t$ of $\yy_{\mu_t}(\x^t)$ and $\vv_{\mu_t}(\x^t)$, respectively,  to obtain approximate gradient $d_{\x}^t$ of
$\nabla h_{\mu_t}(\x^t)$ expressed in \eqref{eq:nabla h_mu(x) 2}.
Next we perform the core step: compute $\x^{t+1}$ by the projected gradient descent method for problem~\eqref{shfun} with $\mu=\mu_t$.

	It is worth noting that $(\mu_t, \zeta_t, \tau_t)\to (0,0,0)$  with rate
	$\zeta_t/\mu_t\to 0$, $\tau_t/\zeta_t\to 0$, as $t \to \infty$. Moreover,
	SIGA merges the decrease of smoothing parameter into the iteration of $\x^t$
	by providing an adaptive and explicit update criterion for $\mu_t \downarrow 0$.
	This integration ensures that $\mu_t$ decreases in a controlled manner, allowing the algorithm to transition smoothly from the smoothed problem~\eqref{eq:1equation constraint smoothing} to the original problem~\eqref{eq:1equation constraint}. The explicit update criterion eliminates the need of a separate parameter tuning for~$\mu_t$.

	Now we prove the convergence of SIGA  to a stationary point of~\eqref{eq:1equation constraint} under Assumption~\ref{assumption: Psi mu}.

	\begin{assumption}		\label{assumption: Psi mu} There exist constants $\widetilde{C}, \overline{C}, \widehat{C}$ and~$L_f$, such that
		for any $\x, \x' \in \X$, $\y, \y' \in \R^n$, and $\mu, \mu' \in (0,1]$,
		\begin{enumerate}[(i)]
			\item \label{lem: Psi - Psi' x}
			\color{black}
			$\| \nabla_{\x} \Psi_{\mu}(\x,\y) \| \le \widetilde{C}$;
			\color{black}

			\item \label{lem: Psi - Psi' mu}
			$\left\|\Psi_{\mu}(\x,\y) - \Psi_{\mu'}(\x,\y) \right\|
			\le \overline{C} | \mu - \mu' |$;

			\item \label{lem: norm of partial Psi * v 0}
			$
			\left\| \nabla \Psi_{\mu}(\x,\y) - \nabla \Psi_{\mu}(\x',\y') \right\|
			\le \frac{\widehat{C}}{\mu} (\| \x - \x'\| + \|\y - \y'\|);
			$

			\item \label{assumption Lf Lipschitz}
			$
			\left\| \nabla f(\x,\y) - \nabla f(\x',\y') \right\|
			\le L_f (\| \x - \x'\| + \|\y - \y'\|).
			$

		\end{enumerate}
	\end{assumption}

	In Lemma~\ref{lem: Verification of Assumption Psi} of Appendix~\ref{sec: smooth box}, we verify that Assumption \ref{assumption: Psi mu} on $\Psi_\mu$ holds for a smoothing function $\Psi_{\mu} $ with  $\Omega(\x) = \{ \y \in \R^n: l(\x) \le \y \le u(\x)\},$
	where
	$l,u: \textcolor{black}{\R^m} \rightarrow \R^n$ are continuously differentiable.
Moreover, we give properties of smooth functions $y_\mu, v_\mu, h_\mu$ in Appendix~\ref{sec:appendix properties lemmas} and~\ref{sec:appendix descent formula lemmas} for the convergence analysis of SIGA in this~section.

	\begin{proposition}[Construction of the descent property] \label{prop:descent of Lyapunov function}
		Let $\{(\x^t,\y^t,\v^t)\}$ be a sequence generated by Algorithm~\ref{alg}.
		Under Assumption~\ref{assumption: Psi mu},
		\begin{equation}\label{eq: descent h mu x}
			\begin{aligned}
				h_{\mu_{t+1}}(\x^{t+1})  - &  h_{\mu_t}(\x^t) \le
				- \theta_t \|\x^{t+1} - \x^t\|^2 + \RR_t,
			\end{aligned}
		\end{equation}
		where
		\[
		\left\{
		\begin{aligned}
			\theta_{t} & =\frac{1}{2} \left(\frac{1}{\zeta_t} - \frac{C_{h}}{\mu_t} \right), \\
			\RR_t & = \tfrac{\left((1 - L_s)^2 + 4 \widetilde{C}^2\right)\left( L_f^2 + 2 \widehat{C}^2 \overline{M}^2 \right)}{(1 - L_s)^4}  \frac{\zeta_t \tau_t^2}{\mu_t^2}
			+ \overline{C} \, \overline{M}   |\mu_t -\mu_{t+1}| + \tfrac{\overline{C}^2 L_f}{2 (1 - L_s)^2 }|\mu_t -\mu_{t+1}|^2,
		\end{aligned}\right.
		\]
		with $C_{h}$ defined in Lemma~\ref{lem: nabla h(x) - nabla h(x')}
		and $\overline{M}$ in Lemma~\ref{lem: y v bounded}.

		Furthermore,
		there exists $\widetilde{T} > 0$ such that for any $t \ge \widetilde{T}$,
		$\theta_{t} \ge \theta_{\widetilde{T}} > 0$,
		and
		there exists $M_{\RR}>0$ such that
		$ \sum_{t=\widetilde{T}}^{\infty} \RR_t \le M_{\RR}$. \end{proposition}
\begin{proof}
{\color{black}

Combining Lemma~\ref{lem:descent h} and Lemma~\ref{lem:descent v}, we obtain~\eqref{eq: descent h mu x}. Next, from~\eqref{eq: alg parameter},
		\[
		\begin{aligned}
			\theta_{t}
			= \frac{1}{2} \cdot t^{2p} \left(\frac{1}{\zeta_0} - \frac{C_{h}}{\mu_0 \cdot t^{p}} \right),
		\end{aligned}
		\]
 which derives that there exists $\widetilde{T} > 0$ such that
for any $t \ge \widetilde{T}$,
		$\theta_{t} \ge \theta_{\widetilde{T}} > 0$.
As for $\RR_t$, from $\mu_t = \mu_0/t^p$, we have
		\[
		\resizebox{.98\linewidth}{!}{$\displaystyle 0 < \mu_t - \mu_{t+1} = \mu_0 \left( \left( \frac{1}{t} \right)^p \! - \left( \frac{1}{t+1} \right)^p  \right)
		\le \mu_0  \left( \frac{1}{t+1} \right)^{p-1} \!\! \left( \frac{1}{t} \! - \! \frac{1}{t+1} \right)
		= \frac{\mu_0}{t \cdot (t+1)^p}. $}\\
		\]
 Since $p > 0$, we have
		$
		\sum_{t=\widetilde{T}}^{\infty} | \mu_t - \mu_{t+1} | < + \infty$ and $
		\sum_{t=\widetilde{T}}^{\infty} | \mu_t - \mu_{t+1} |^2 < + \infty.
		$
		Also, \begin{displaymath}\sum_{t=\widetilde{T}}^{\infty} \frac{\zeta_t \tau_t^2}{\mu_t^2} = \frac{\zeta_0 \tau_0^2}{\mu_0^2} \sum_{t=\widetilde{T}}^{\infty}  \frac{1}{t^2} < + \infty.\end{displaymath}
 Hence there exists $M_{\RR}>0$ such that
		$ \sum_{t=\widetilde{T}}^{\infty} \RR_t \le M_{\RR}$.
}
\end{proof}

	The above result indicates that~$\theta_{t}$ dominates the residual $\RR_t$ in magnitude as $t \rightarrow \infty$,
	as the basis for our subsequence convergence.

\par\smallskip

\color{black}
\begin{theorem}\label{thm: min converge}
	Under the settings of
	Proposition~\ref{prop:descent of Lyapunov function},
	for each integer $j\ge0$, choose
	\[
	t_j\in\arg\min_{2^j\le t<2^{j+1}}
	\|\x^{t+1}-\x^t\|^2.
	\]
	Then $t_{j+1}>t_j$, $t_j\to\infty$, and the subsequence satisfies
	\[
	\|\x^{t_j+1}-\x^{t_j}\|
	=O(t_j^{-1/2})\to0
	\qquad\text{as }j\to\infty.
	\]
\end{theorem}
{\color{black}
	\begin{proof}
		Denote $\gamma_t:=  h_{\mu_t}(\x^t)$.
		From Lemma~\ref{lem:y mu x}(\ref{lem: y mu x bounded}) and the continuity of $f$,
		it derives that $\gamma_t$ has a lower bound, denoted as $h_0$.
		For any $T >\widetilde{T}$,
		from Proposition~\ref{prop:descent of Lyapunov function},
		\begin{equation}\label{eq: thm extraction revised}
			\begin{aligned}
				\sum_{t=\widetilde T}^{T-1}
				\theta_t\|\x^{t+1}-\x^t\|^2
				&\le \gamma_{\widetilde T}-\gamma_T
				+\sum_{t=\widetilde T}^{T-1}\RR_t
				\le \gamma_{\widetilde T}-h_0+M_{\RR}.
			\end{aligned}
		\end{equation}
		Moreover, $\theta_t\ge\theta_{\widetilde T}>0$
		for all $t\ge\widetilde T$.
		By the definition of $t_j$,
		\[
		2^j\le t_j<2^{j+1}\le t_{j+1},
		\]
		so $t_{j+1}>t_j$ and $t_j\to\infty$.

		For every sufficiently large $j$ such that
		$2^j\ge\widetilde T$,
		by the choice of $t_j$, for every integer $t$ satisfying
		$2^j \le t < 2^{j+1}$, we have
		\[
		\|\x^{t+1}-\x^t\|^2
		\ge \|\x^{t_j+1}-\x^{t_j}\|^2.
		\]
		Together with $\theta_t \ge \theta_{\widetilde T}>0$, this gives
		\[
		\theta_t\|\x^{t+1}-\x^t\|^2
		\ge \theta_{\widetilde T}\|\x^{t_j+1}-\x^{t_j}\|^2.
		\]
		Summing this bound over the $2^{j+1}-2^j=2^j$ integer indices
		in the interval yields
		\[
		\begin{aligned}
			\theta_{\widetilde T}\,2^j
			\|\x^{t_j+1}-\x^{t_j}\|^2
			&\le
			\sum_{t=2^j}^{2^{j+1}-1}
			\theta_t\|\x^{t+1}-\x^t\|^2
			\le
			\sum_{t=\widetilde T}^{2^{j+1}-1}
			\theta_t\|\x^{t+1}-\x^t\|^2
			\le \gamma_{\widetilde T}-h_0+M_{\RR},
		\end{aligned}
		\]
		where the last inequality is from \eqref{eq: thm extraction revised}.

		Since $t_j<2^{j+1}$, we obtain
		\[
		\begin{aligned}
			\|\x^{t_j+1}-\x^{t_j}\|^2
			&\le
			\frac{\gamma_{\widetilde T}-h_0+M_{\RR}}
			{\theta_{\widetilde T}2^j}
			\le
			\frac{2(\gamma_{\widetilde T}-h_0+M_{\RR})}
			{\theta_{\widetilde T}t_j}
			=O(t_j^{-1}),
		\end{aligned}
		\]
		which derives the desired result.
	\end{proof}
}

	Based on this theorem, we can prove the stationarity convergence of SIGA.
	\begin{theorem}[Convergence of feasibility violation and to the stationary point]\label{thm: feasibility-stationarity-convergence}
		Under Assumption~\ref{assumption: jacobian consistency + Lipschitz in mu} and the settings of Proposition~\ref{prop:descent of Lyapunov function},
		it holds that
		\begin{enumerate}[(i)]
			\item \label{thm: feasibility violation convergence}
			$\lim\limits_{t \rightarrow \infty} \| \y^t - \Psi(\x^t, \y^t) \| = 0$.
			\item \label{thm: Rt converge}
\color{black}
			Any accumulation point $(\x^*,\y^*,\v^*)$ of the subsequence $\{(\x^{t_j},\y^{t_j},\v^{t_j})\}$ in Theorem~\ref{thm: min converge}
			satisfies Definition~\ref{def: stationary point},
			with $(\x^*,\y^*)$ being a stationary point of~\eqref{eq:1equation constraint}
			and $\v^*$ corresponding to the multiplier $\lambda^*$.
		\end{enumerate}

	\end{theorem}
	\begin{proof}
{\color{black}
		(i)
		It is easy to obtain that
		\[
		\begin{aligned}
			& \| \y^{t} - \Psi(\x^{t}, \y^{t}) \|
			\le \| \y^{t} - \Psi_{\mu_{t}}(\x^{t}, \y^{t}) \|
			+ \| \Psi_{{\mu_{t}}}(\x^{t}, \y^{t})
			- \Psi(\x^{t}, \y^{t})\|,
		\end{aligned}
		\]
		converges to zero as $t \rightarrow \infty$ according to~\eqref{eq:algorithm sub-step 1} and Definition~\ref{definition:smoothing}(\ref{assump: Psi_mu limit}).
}

		(ii)
		From~\eqref{eq: alg parameter},  with $0< p <1/4$, we have
		$$\frac{\tau_t}{\mu_t} = \frac{\tau_0}{\mu_0 \cdot t^{1-p}} \rightarrow 0,$$
		as $t \rightarrow \infty$.
		Then from Lemma~\ref{lem:descent v}, it holds that
		$\| \y^{t} - \yy_{\mu_t}(\x^t) \| \rightarrow 0$ and $\| \v^{t} - \vv_{\mu_t}(\x^t) \| \rightarrow 0$.
		Along with Lemma~\ref{lem:y mu x}(\ref{lem: y mu x bounded}) and Lemma~\ref{lem: y v bounded},
		$\{\y^t\}$ and $\{\v^t\}$ are bounded.
		Then for the subsequence  $\{{t_j}\}$ in Theorem~\ref{thm: min converge},
		$\{(\x^{t_j},\y^{t_j},\v^{t_j})\}$ is bounded,
		with an accumulation point denoted as $(\x^*,\y^*,\v^*)$, where $\x^* \in \X$.
		According to~(\ref{thm: feasibility violation convergence}),
		we have $\y^* = \Psi(\x^*,\y^*)$
		and $(\x^*,\y^*) \in \F$.

		From the iteration of $\x^t$ in~\eqref{eq:x iteration}, we have $$
		\begin{aligned}
			\x^{t+1} &= \arg\min_{\bar{\x} \in \X} \left\| \x^t - \zeta_t \Big(\nabla_{\x} f(\x^t,\y^t)
			+ \left( \nabla_{\x} H_{\mu_t}(\x^t,\y^t) \right)^\top \v^t \Big)  - \bar{\x} \right\|^2,
		\end{aligned}
		$$
		which derives
		\[
		2 \left( \x^t - \zeta_t \Big(\nabla_{\x} f(\x^t,\y^t)
		- \left( \nabla_{\x} \Psi_{\mu_t}(\x^t,\y^t) \right)^\top \v^t \Big) - \x^{t+1} \right) \in \N_{\X}(\x^{t+1}).
		\]
		Then for the subsequence $\{t_j\}$,
		\begin{equation*} \label{eq:theorem KKT epsilon x}
			\frac{\x^{t_j} - \x^{{t_j}+1}}{\zeta_{t_j}} \in \nabla_{\x} f(\x^{t_j},\y^{t_j}) - \left( \nabla_{\x} \Psi_{\mu_{t_j}}(\x^{t_j},\y^{t_j}) \right)^\top  \v^{t_j} + \N_{\X}(\x^{{t_j}+1}).
		\end{equation*}
		From~\eqref{eq:algorithm sub-step 2} and Lemma~\ref{lem: nonsingular},
		$
		\nabla_{\y} f(\x^{t_j},\y^{t_j})  + \v^{t_j} -  \left(
		\nabla_{\y} \Psi_{\mu_{t_j}}(\x^{t_j},\y^{t_j})
		\right)^\top \v^{t_j} = \boldsymbol{\mathrm{0}},
		$
		and then
		\begin{equation}\label{eq:theorem KKT epsilon}
			\left(\begin{aligned}
				& \tfrac{\x^{t_j} - \x^{{t_j}+1}}{\zeta_{t_j}} \\& \quad \boldsymbol{\mathrm{0}}
			\end{aligned}
			\right)
			\in \nabla f(\x^{t_j},\y^{t_j}) +\left(
			\begin{aligned}
				& \boldsymbol{\mathrm{0}} \\& \v^{t_j}
			\end{aligned}
			\right) - \nabla \Psi_{\mu_{t_j}}(\x^{t_j},\y^{t_j})^\top \v^{t_j} + \left(\begin{aligned}
				&\N_{\X}(\x^{t_j+1}) \\&\quad \{\boldsymbol{\mathrm{0}}\}
			\end{aligned}\right).
		\end{equation}

{\color{black}

		According to Theorem~\ref{thm: min converge}, since $0 < p <\frac{1}{4}$,
		\[
		\frac{\|\x^{t_j} - \x^{{t_j}+1}\|}{\zeta_{t_j}} \le O \left(  t_j^{2p - \frac{1}{2}}
		\right) \rightarrow 0,
		\,  {\rm as} \, t_j \rightarrow \infty.
		\]
		Then by Assumption~\ref{assumption: jacobian consistency + Lipschitz in mu} and the continuous differentiability of $f$, taking limit to~\eqref{eq:theorem KKT epsilon} derives
		\[
		\boldsymbol{\mathrm{0}} \in \nabla f(\x^*,\y^*) +\left(
		\begin{aligned}
			& \boldsymbol{\mathrm{0}} \\			& \v^*
		\end{aligned}
		\right) - \partial_C \Psi (\x^*,\y^*) \ \v^* + \left(\begin{aligned}
			& \N_{\X}(\x^*) \\			& \quad \{\boldsymbol{\mathrm{0}}\}
		\end{aligned}\right)
		\]
		as the desired result.
}
	\end{proof}

\section{Numerical experiments}\label{sec: experiments}

This section evaluates SIGA on a numerical problem with a moving-box PVI constraint and on portfolio-management applications.
All experiments are run in Python 3.11.15 on a PC with an Intel Core Ultra 5 228V CPU and 32 GB RAM.

\subsection{Numerical verification}\label{sec: toy exp}

Extending the numerical experiments in~\cite{jiang2024primal,shen2026single} from a one-dimensional problem to a high-dimensional one, we consider the following bi-level problem:
\begin{equation}\label{eq:new-toy-bilevel}
	\begin{aligned}
		\min_{\x\in\X}\quad & f\bigl(\x,\y(\x)\bigr)
		=\frac{e^{-\frac{\ee_n^\top\y(\x)}{n}+2}}{2+\cos\left(\frac{6\ee_m^\top\x}{m}\right)}
		+\frac{1}{2}\ln\!\left(\left(\frac{4\ee_m^\top\x}{m}-2\right)^2+1\right),\\
		\text{s.t.}\quad &
		\y(\x)=\operatorname*{argmin}_{-0.001\ee_n\leq\y\leq A\x}
		\frac{1}{2}(\y-2A\x)^\top Q(\y-2A\x),
	\end{aligned}
\end{equation}
where $\X=[0,3]^m$, $Q\in\R^{n\times n}$, $A\in\R^{n\times m}$,
and $\ee_m$ and $\ee_n$ respectively denote all-ones vectors in $\R^m$ and $\R^n$.
Requiring $Q$ to be a symmetric positive definite matrix and all entries of $A$ and $QA$ to be nonnegative guarantees that the lower-level problem admits a unique solution $\y=A\x$.
The lower-level problem is equivalent to the PVI with
$\Omega(\x)=\{\y:-0.001\ee_n\leq\y\leq A\x\}$ and $F(\x,\y)=Q(\y-2A\x)$.
Thus,
\begin{equation}\label{eq:new-toy-Psi}
	\Phi(\x,\y)=\y-\delta Q(\y-2A\x), \qquad
	\Psi(\x,\y)
	=\Proj_{[-0.001\ee_n,A\x]}
	\bigl(\y-\delta Q(\y-2A\x)\bigr).
\end{equation}

For each random seed, a positive matrix $A\in\R^{n\times m}$ and a nonnegative matrix $B\in\R^{n\times n}$ are sampled from uniform distributions and normalized so that $A\ee_m=\ee_n$ and $\|B\|_2=1$.
We set $Q=B^\top B+I_n$.

\begin{samepage}
\begin{remark}\label{remark: remark PSD}
\color{black}
	In Proposition~\ref{prop: contraction Psi},
	if there exists $\sigma\in(0,\bar L]$ such that
	$\nabla_\y F(\x,\y)$ is symmetric
	and $
	\sigma I\preceq \nabla_\y F(\x,\y)\preceq \bar L I
	$
	for any $\x\in\X$ and $\y \in \R^n$,
	then \[
	\begin{aligned}
		\|\nabla_\y\Phi(\x,\y)\|
		&=\|I-\delta\nabla_\y F(\x,\y)\|
		\le \max_{\lambda\in[\sigma,\bar L]}|1-\delta\lambda|
		= \max \{ |1 - \delta \sigma|, |1 - \delta \bar{L} | \}.
	\end{aligned}
	\]
	Let $0<\delta<2/\bar L$,
	and then
	Assumption~\ref{assump: foundation}(\ref{assump: contraction Psi}) is satisfied.
\end{remark}
\end{samepage}

\color{black}
Since $\nabla_\y F(\x,\y)=Q$, we take
$\sigma=\lambda_{\min}(Q)$ and $\bar L=\lambda_{\max}(Q)$ in Remark~\ref{remark: remark PSD}.
We choose $\delta=2/(\sigma+\bar L)$, which satisfies
$0<\delta<2/\bar L$ and minimizes the contraction bound
$L=\max\{|1-\delta\sigma|,|1-\delta\bar L|\}$,
yielding $L=(\bar L-\sigma)/(\bar L+\sigma)<1$.
Hence, Assumption~\ref{assump: foundation}(\ref{assump: contraction Psi}) is satisfied, and metric regularity follows from Proposition~\ref{prop: metric regularity}.
\color{black}

On the other hand, the projection in~\eqref{eq:new-toy-Psi} is nonsmooth at its switching points, so classical smooth CQ such as LICQ, MFCQ, and CRCQ (discussed in Remark~\ref{remark:cq}) are not directly applicable to the equality constraint $\y=\Psi(\x,\y)$.

\subsubsection{Methods, parameters, and performance measures}

We compare SIGA with SFLCB~\cite{shen2026single}, BLOCC~\cite{jiang2024primal}, and two standard nonlinear-programming implementations of the MPCC reformulation.

\begin{enumerate}[(i)]
	\item
	SIGA:
	use SIGA (Algorithm~\ref{alg}) with $p = 0.001, \mu_0 = 0.0001, \zeta_0 = 1, \tau_0 = 10^{-8}$. Following~\cite{qi2000new,burke2013gradient},
	we use the CHKS smoothing function with
	density function $\rho(t) = \frac{2}{(t^2+4)^{\frac{3}{2}}}$ (to be discussed in more detail in Appendix~\ref{sec: smooth box}).
	For the projection onto $\Omega(\x)= \left\{\y: l(\x) \le \y \le u(\x) \right\}$,
	the CHKS smoothing function reads as
	\begin{equation}\label{eq: chks}
		\begin{aligned}
			(\Psi_{\mu} (\x,\y))_i
			= \frac{1}{2} \Big( & l_i(\x) + \sqrt{(l_i(\x) - \Phi_i(\x,\y))^2 + 4 \mu^2} \\
			+ & u_i(\x) - \sqrt{(u_i(\x) - \Phi_i(\x,\y))^2 + 4 \mu^2} \Big),
			\text{ for } i = 1,2,\cdots,n,
		\end{aligned}
	\end{equation}
	and here $l(\x)\equiv-0.001\ee_n$ and $u(\x)=A\x$.

	\item
	SFLCB: consider $\Omega(\x)$ as two inequality constraints
	$\y-A\x\leq\boldsymbol{0}$ and $-\y-0.001\ee_n\leq\boldsymbol{0}$,
	with parameters defined in~\cite{shen2026single}
	set as $\delta_{\rm SFLCB}=0.1$, $\eta_x=\eta_y=\eta_z=\eta_u=\eta_v=0.01$, and $\rho_1=\rho_2=1$.

	\item BLOCC: use the same two inequality constraints as SFLCB,
	with parameters defined in~\cite{jiang2024primal}
	set as $\gamma=5$, $\beta=\tau=0.1$, $\alpha=0.01$, and inner tolerance $10^{-4}$.

	\item MPCC:

	Step 1: Given $\x \in \X$, since $Q$ is positive definite and the constraint of $\Omega(\x)$ is linear,
	then $\y$ is a solution to the lower-level problem in~\eqref{eq:new-toy-bilevel}
	if and only if
	there exist $\tilde{\lambda},\bar{\lambda} \in \R^n$ such that
	\begin{equation}\label{eq: mpcc constraint}
		\left\{\begin{aligned}
			F(\x,\y)  - \tilde{\lambda} + \bar{\lambda} = 0, \\
			0 \le \tilde{\lambda} \perp \y - l(\x) \ge 0, \\
			0 \le \bar{\lambda} \perp u(\x) - \y \ge 0.
		\end{aligned}\right.
	\end{equation}
	Inspired by~\cite{kunapuli2008classification},
	we aim to solve
	$
		\min_{\x \in \X, \y, \tilde{\lambda},\bar{\lambda}} \ f(\x,\y)
	$
	under the above constraints.

	Step 2:
	Following~\cite{kunapuli2008classification,ye2023difference},
	we further impose the componentwise Scholtes relaxations~\cite{scholtes2001convergence,hoheisel2013theoretical}
	$\tilde{\lambda}_i(y_i+0.001)\leq10^{-10}$ and
	$\bar{\lambda_i} (A\x-\y)_i\leq10^{-10}$ to
	replace the complementarity constraints, in order to facilitate an early termination and ease the difficulty of dealing with the
	complementarity constraints for a nonlinear programming solver.

	Step 3:
	We solve the relaxed MPCC problem by the package \texttt{cyipopt} based on the interior point method (denoted as MPCC-IP),
	and by the solver \texttt{trust-constr} in the package \texttt{scipy.optimize}  based on the trust region method (denoted as MPCC-TR).
\end{enumerate}

\begin{figure}[htbp]
	\centering
{\begin{subfigure}[t]{0.49\textwidth}
			\centering
			{\setlength{\fboxsep}{0pt}\includegraphics[width=\dimexpr\linewidth-2\fboxrule\relax]{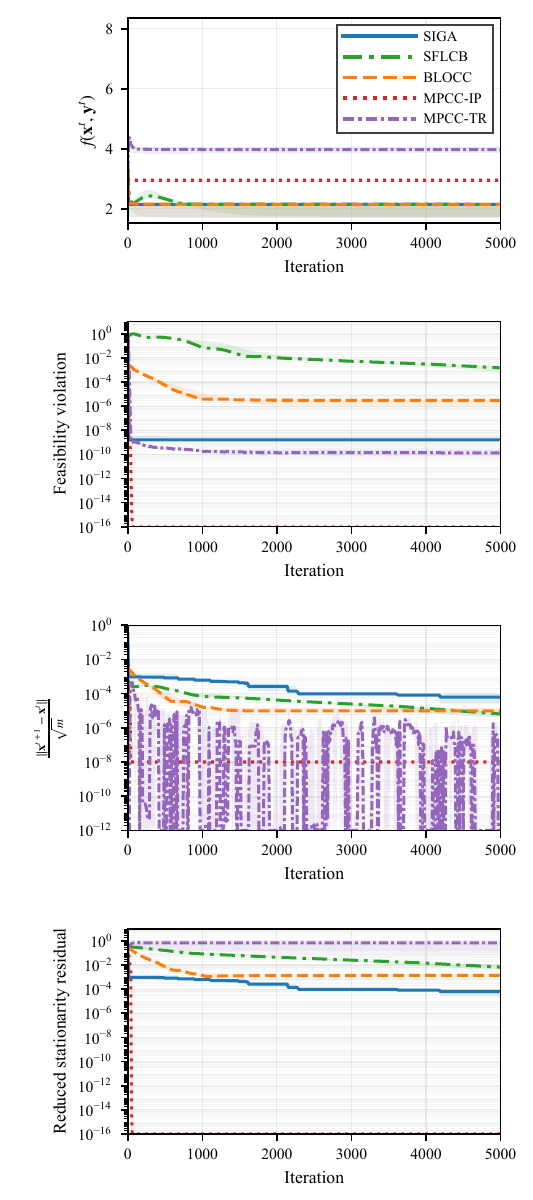}}
			\caption{$(m,n)=(10,100)$}
		\end{subfigure}\hfill
		\begin{subfigure}[t]{0.49\textwidth}
			\centering
			{\setlength{\fboxsep}{0pt}\includegraphics[width=\dimexpr\linewidth-2\fboxrule\relax]{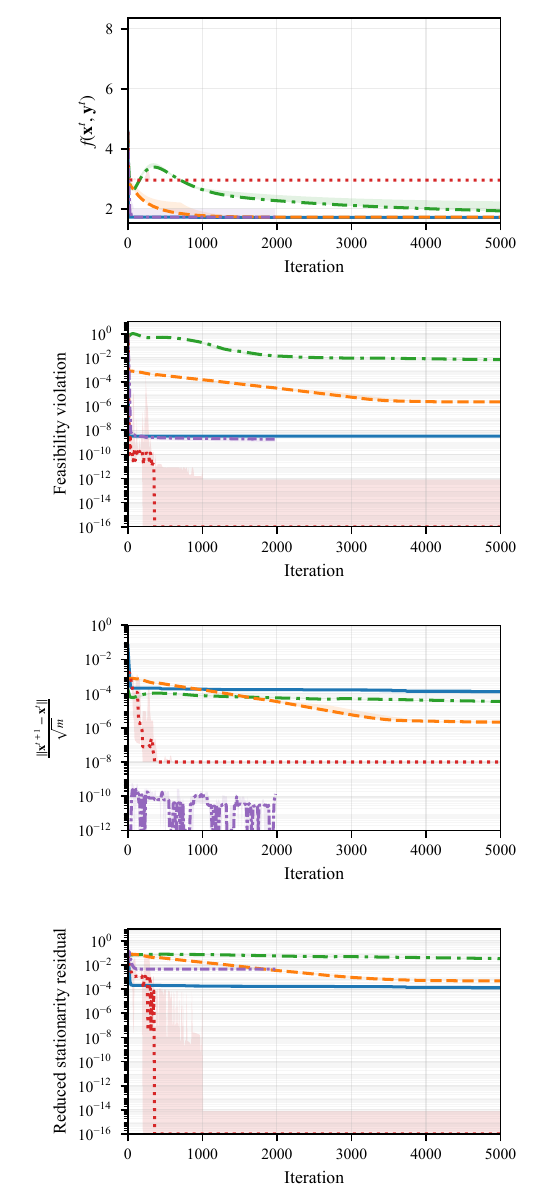}}
			\caption{$(m,n)=(50,500)$}
		\end{subfigure}}
\caption{Convergence curves versus iteration.
		\label{fig:new-toy-iteration}}

\end{figure}

For five random seeds, a new pair $(A,B)$ and a new initial point are generated.
Each component of~$\x^1$ is independently sampled from $\mathcal U(0,3)$. Given $\x^1$, each component $y_i^1$ is independently sampled from $\mathcal U(0,(A\x^1)_i)$.
Within a seed, every method shares the same initial point $(\x^1,\y^1)$.

\begin{figure}[htbp]
	\centering
{\begin{subfigure}[t]{0.49\textwidth}
			\centering
			{\setlength{\fboxsep}{0pt}\includegraphics[width=\dimexpr\linewidth-2\fboxrule\relax]{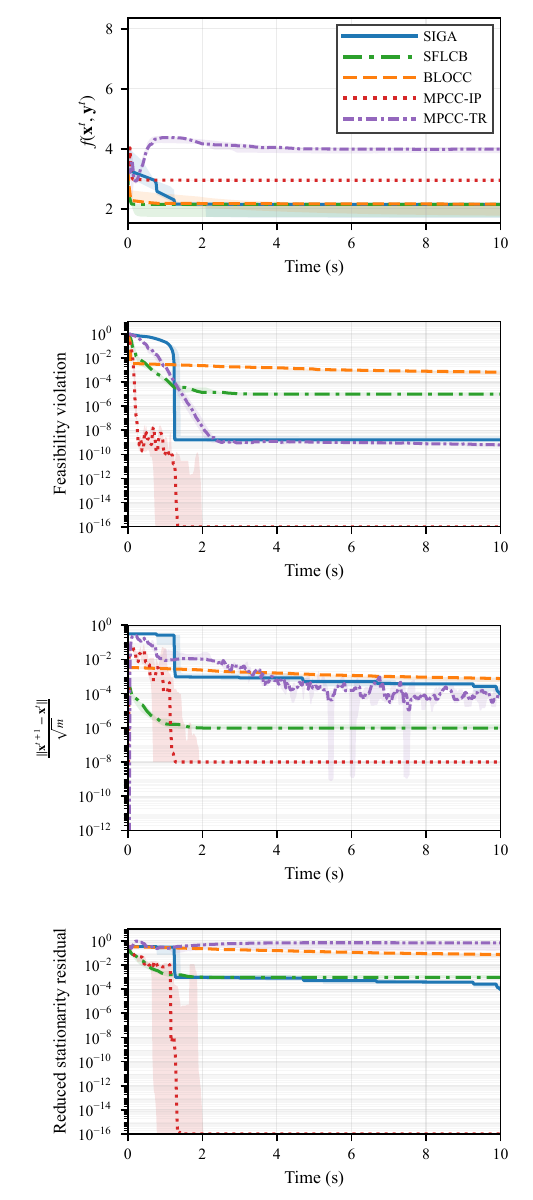}}
			\caption{$(m,n)=(10,100)$}
		\end{subfigure}\hfill
		\begin{subfigure}[t]{0.49\textwidth}
			\centering
			{\setlength{\fboxsep}{0pt}\includegraphics[width=\dimexpr\linewidth-2\fboxrule\relax]{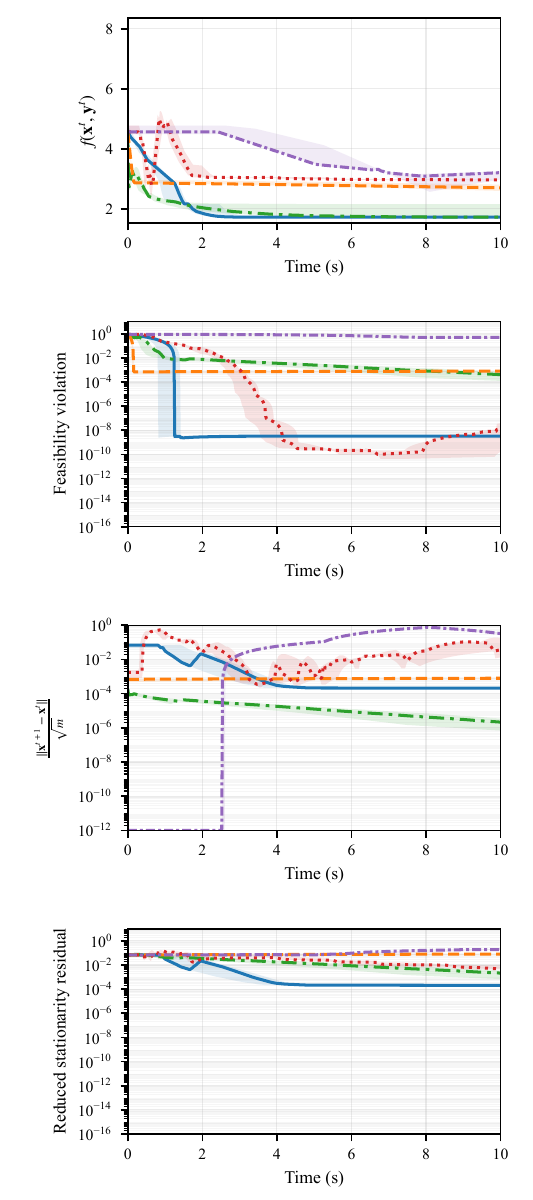}}
			\caption{$(m,n)=(50,500)$}
		\end{subfigure}}
\caption{Convergence curves over the first 10 seconds.
		\label{fig:new-toy-time}}
\end{figure}

We consider the following four performance measures.
\begin{enumerate}[(i)]
	\item Objective value $f(\x^t,\y^t)$; \item Feasibility violation:
	$
	\sqrt{\frac{\|\y^t-\Psi(\x^t,\y^t)\|^2}{n}
		+\frac{{\rm dist}(\x^t,\X)^2}{m}};
	$
	\item Iterate displacement: $
	\frac{\|\x^{t+1}-\x^t\|}{\sqrt m};
	$
	\item Reduced stationarity residual:
	$
	\frac{1}{\sqrt m}{\rm dist}\bigl(\boldsymbol{0},
	\nabla\widehat f(\x^t)+\N_{\X}(\x^t)\bigr),
	$
	where
	$\widehat f(\x):=f(\x,A\x)$, whose gradient is $\nabla\widehat f(\x)=\nabla_{\x}f(\x,A\x)+A^\top\nabla_{\y}f(\x,A\x)$.
\end{enumerate}

\subsubsection{Convergence curves}

For $(m,n)=(10,100)$ and $(50,500)$, every method is run for up to 5000 iterations in Figures~\ref{fig:new-toy-iteration} and~\ref{fig:new-toy-siga-stationarity}, and for 10 seconds in Figure~\ref{fig:new-toy-time}.
For $(m,n)=(50,500)$, the MPCC-TR curve in Figure~\ref{fig:new-toy-iteration} is truncated because its runtime exceeds 3600 seconds.
In Figures~\ref{fig:new-toy-iteration}--\ref{fig:new-toy-siga-stationarity} and~\ref{fig:new-toy-scaling}, each line is the pointwise median over five seeds and each shaded region is bounded by the first and third quartiles.

Figure~\ref{fig:new-toy-iteration} and Figure~\ref{fig:new-toy-time} show satisfactory convergence behavior for SIGA. The decay of $\|\x^{t+1}~-~\x^t\|/\sqrt m$ is consistent with Theorem~\ref{thm: min converge}, and the decay of the feasibility violation is consistent with Theorem~\ref{thm: feasibility-stationarity-convergence}(\ref{thm: feasibility violation convergence}).
MPCC-IP does not obtain a satisfactory decrease of the objective value.
\color{black}
MPCC-TR fails to achieve a decrease of the objective value for $(m,n)=(10,100)$,
and it is much slower than other methods for $(m,n)=(50,500)$.
\color{black}
BLOCC and SFLCB converge more slowly than SIGA, and the gap becomes more pronounced as the dimensions increase.
Figure~\ref{fig:new-toy-scaling} later provides a further comparison of this scaling behavior.

\color{black}
\begin{figure}[htbp]
\centering
{\begin{subfigure}[t]{0.49\textwidth}
		\centering
		{\setlength{\fboxsep}{0pt}\includegraphics[width=0.8\linewidth]{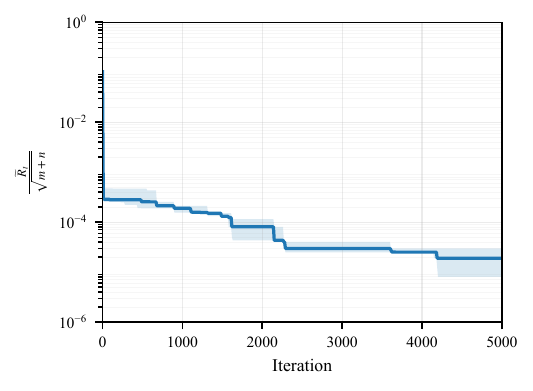}}
		\caption{$(m,n)=(10,100)$}
	\end{subfigure}\hfill
	\begin{subfigure}[t]{0.49\textwidth}
		\centering
		{\setlength{\fboxsep}{0pt}\includegraphics[width=0.8\linewidth]{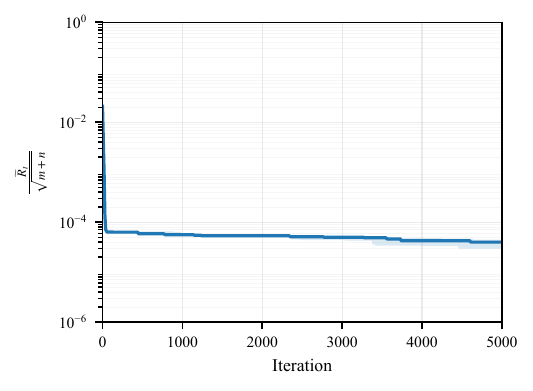}}
		\caption{$(m,n)=(50,500)$}
	\end{subfigure}}
\caption{Convergence curves of SIGA stationarity residual $\widetilde R_t/\sqrt{m+n}$ defined from~\eqref{eq: stationarity residual}.
\label{fig:new-toy-siga-stationarity}}
\end{figure}

For the generated sequence of SIGA, inspired by~\eqref{eq:KKT smoothing},
we define the stationarity residual as

\begin{equation}\label{eq: stationarity residual}
	\widetilde{R}_t:= {\rm dist} \left( \boldsymbol{\mathrm{0}},  \nabla f(\x^t,\y^t) +\left(
	\begin{aligned}
		& \boldsymbol{\mathrm{0}} \\& \v^t
	\end{aligned}
	\right) - \nabla \Psi_{\mu_t}(\x^t,\y^t)^\top \v^t + \left(\begin{aligned}
		&\N_{\X}(\x^t) \\&\quad \{\boldsymbol{\mathrm{0}}\}
	\end{aligned}\right) \right),
\end{equation}
which is specific to SIGA.
Results of Figure~\ref{fig:new-toy-siga-stationarity} are consistent with
the convergence result in Theorem~\ref{thm: feasibility-stationarity-convergence}(\ref{thm: Rt converge}).
\color{black}

\begin{figure}[htbp]
\centering
{\begin{subfigure}[t]{0.49\textwidth}
		\centering
		\includegraphics[width=0.8\linewidth]{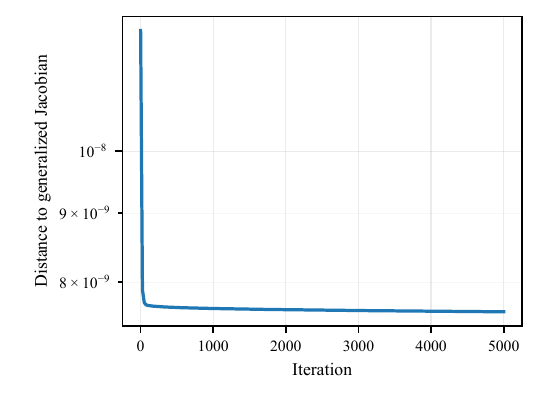}
		\caption{Along the SIGA trajectory. }
	\end{subfigure}\hfill
	\begin{subfigure}[t]{0.49\textwidth}
		\centering
		\includegraphics[width=0.8\linewidth]{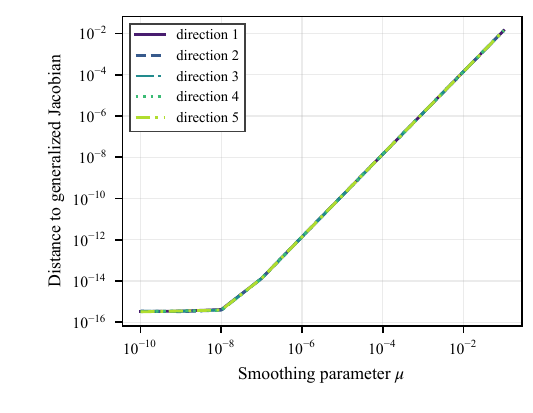}
		\caption{Five generated paths.
		}
	\end{subfigure}}
\caption{Numerical check of Assumption~\ref{assumption: jacobian consistency + Lipschitz in mu}
		regarding ${\rm dist} \left( \nabla \Psi_{\mu_t} (\x^t,\y^t)^\top, \partial_C \Psi (\x,\y) \right)$.
		Both panels show the distance from the smoothed Jacobian to the generalized Jacobian $\partial_C\Psi$ at the corresponding limit.
\label{fig:new-toy-jacobian}}
\end{figure}

Figure~\ref{fig:new-toy-jacobian} provides a numerical consistency check for Assumption~\ref{assumption: jacobian consistency + Lipschitz in mu}.
\textcolor{black}{The matrix distances are computed under the Frobenius norm.}
For the left subfigure, the iterate returned by SIGA is taken as the reference point $(\x,\y)$.
At this point, all components lie strictly on the upper branch of the projection, so $\Psi$ is differentiable, and $\partial_C\Psi(\x,\y)$ is a singleton.

In the right subfigure of Figure~\ref{fig:new-toy-jacobian}, we design a case where $\Psi$ is nonsmooth at the limiting point $(\x,\y)$.
Specifically, $(\x,\y)$ is prescribed as $\x=1.5\ee_m$ and
$\y=A\x-\alpha\boldsymbol{e}^{(1)}$, where
$\alpha=\frac{\delta(QA\x)_1}{1-\delta Q_{11}}$
and $\boldsymbol{e}^{(1)}\in\R^n$ denotes $(1,0,0,\cdots,0)$.
This construction places the first component exactly at the upper-bound switching point, i.e., $\Phi_1(\x,\y)=(A\x)_1$, while all remaining components satisfy $\Phi_i(\x,\y)>(A\x)_i$ and hence lie on the upper branch of the projection.
Five unit vectors $d^k=(d_x^k,d_y^k)$ are generated from a standard Gaussian sample, and the paths are $\x^{k,j}=\x+\mu_jd_x^k$, $\y^{k,j}=\y+\mu_jd_y^k$, with $\mu_j=10^{-j}$.  Hence $(\x^{k,j},\y^{k,j},\mu_j)\to(\x,\y,0)$ along every path.
The distances for $j=1,\ldots,10$ decrease toward zero along all five generated paths,
exhibiting no inconsistency with Assumption~\ref{assumption: jacobian consistency + Lipschitz in mu}.

\subsubsection{Scaling with the dimensions}

We next compare the time required to converge under various dimensions, where the stopping criterion is set as the reduced stationarity residual being smaller than $5\times10^{-3}$,
and the experiment is terminated after 600 seconds if the criterion is still not met.
In the first row of Figure~\ref{fig:new-toy-scaling},
we report the convergence time at each dimension with unsuccessful runs recorded as 600 seconds.
The second and third rows compare the terminal solution quality against SIGA.

\color{black}

\begin{figure}[htbp]
\centering
{\includegraphics[width=\textwidth]{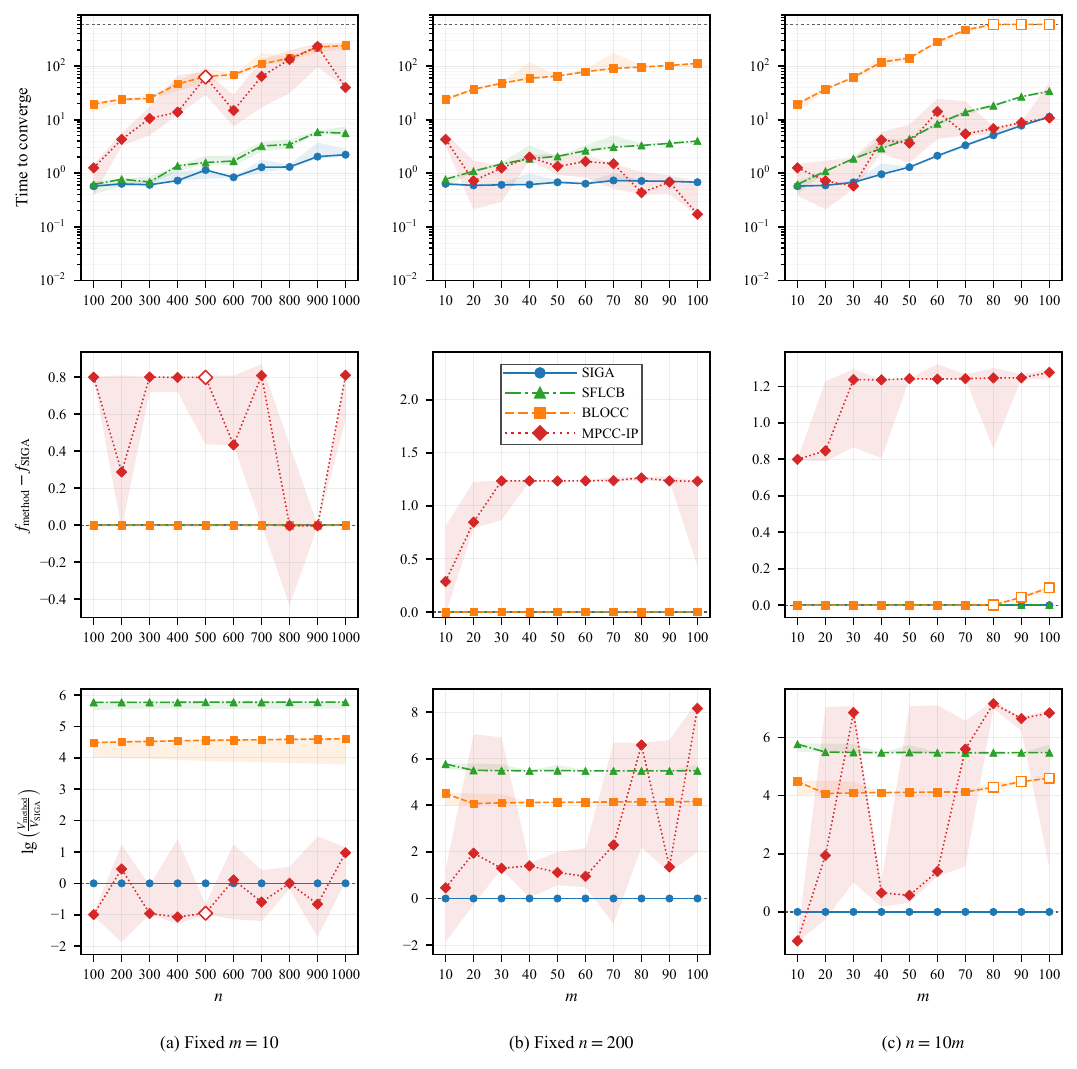}}
\caption{Time to converge and comparison of output quality with SIGA under various dimensions.
	The three rows report, respectively, the convergence time,
	$f_{\mathrm{method}}-f_{\mathrm{SIGA}}$, and
	$\lg\!\left(V_{\mathrm{method}}/V_{\mathrm{SIGA}}\right)$,
	where $V$ denotes the feasibility violation.
	Open markers indicate that at least one of the five runs did not meet the stopping criterion within 600 seconds.
\label{fig:new-toy-scaling}}
\end{figure}

From the three subfigures in Figure~\ref{fig:new-toy-scaling}, SIGA grows only mildly with the lower-level dimension $n$ and is nearly insensitive to the upper-level dimension $m$.
BLOCC is markedly slower,
and fails to converge in most seeds for $(m,n)=(80,800),(90,900),(100,1000),$
while the running time of SFLCB increases with the dimensions more than SIGA.
BLOCC and SFLCB achieve much worse feasibility than SIGA from the third row.
MPCC-IP fails to converge in one seed for $(m,n)=(10,500)$,
does not obtain a satisfactory decrease of the objective function in most experiments,
and takes much more time to converge than SIGA for $(m,n)=(10,800)$ and $(m,n)=(10,900)$.
MPCC-TR is omitted since it fails to satisfy the stopping criterion in most experiments,
while it fails to achieve feasibility in the remaining ones.

\color{black}
\color{black}
For MPCC methods,
the lower-level multipliers $\tilde{\lambda},\bar{\lambda}\in\R^n$
introduced in~\eqref{eq: mpcc constraint} are treated as decision variables~\cite{benko2021implicit}, increasing the number of decision
variables from $m+n$ to $m+3n$.
When solving the relaxed MPCC problem, the nonlinear programming
solvers introduce another set of Lagrange multipliers for its constraints.
These additional variables and constraints enlarge the systems handled
by the solvers, whose sizes grow with $n$.
This increases the cost of each iteration and may explain
the longer running time of the MPCC methods as the dimensions increase.
In comparison, SIGA solves~\eqref{eq:1equation constraint}
 with $m+n$ variables.

\color{black}
\color{black}

 \subsection{Applications in portfolio management}\label{sec: numerical}

Now we apply SIGA to solve the portfolio management problem~\cite{chen2018smoothing}.
Inspired by the standard Markowitz mean-variance model, we first define the portfolio selection model as
\begin{equation}\label{eq: MV model lower level}
	\begin{aligned}
		\min\limits_{\y \in \Omega(\x)} \  & \frac{1}{2} \y^\top \Sigma \y - \eta \rr^\top \y + \frac{\nu}{2} \left(\ee_n^\top \y - 1 \right)^2,
	\end{aligned}
\end{equation}
where $\y \in \R^n$ denotes the portfolio weights of $n$ stocks (risky assets),
$\rr \in \R^n$ and $\Sigma \in \R^{n \times n}$ are the mean vector and covariance matrix of returns on each asset,
and $\eta \ge 0$ is the risk tolerance factor.
Let $m = 2n+1$, $\x = (\aa^\top \ \bb^\top \ \eta)^\top \in \R^{2n+1}$, and
\begin{equation}\label{eq: portfolio X}
	\X = \left\{(\underline{\aa}, \underline{\bb}, 0 )
	\le \x \le (\overline{\aa}, \overline{\bb}, 10^8):
	\underline{\aa} = \boldsymbol{\mathrm{0}}, \
	\overline{\aa} = \tfrac{1}{n+1} \ee_n,
	\underline{\bb} = \tfrac{1}{n-1} \ee_n,
	\overline{\bb} = \ee_n \right\},
\end{equation}
where $\aa,\underline{\aa},\overline{\aa}$, $\bb,\underline{\bb},\overline{\bb} \in \R^n$.
Let $\Omega(\x) = \{ \y \in \R^n: l(\x) \le \y \le u(\x)\}$ be a moving box,
where $l(\x) = \left(I_{n} \ \boldsymbol{\mathrm{0}}_{n \times (n+1)}\right) \x=\aa, $ $u(\x) = \left(\boldsymbol{\mathrm{0}}_{n \times n} \ I_{n} \ \boldsymbol{\mathrm{0}}_{n \times 1}\right) \x=\bb$.
From $\overline{\aa} < \underline{\bb}$ and $\x \in \X$, we have $\aa < \bb$ always holds,
so for any $\x \in \X$, $l(\x) < u(\x)$ and $\Omega(\x)$ is nonempty.
In~\eqref{eq: MV model lower level}, equality constraint $\ee_n^\top \y = 1$ is penalized to the objective with parameter $\nu > 0$ so that partial investment and leverage are allowed.
Since $\eta = \x_m = \left(\boldsymbol{\mathrm{0}}_{1 \times 2n} \ 1 \right) \x$,
the model corresponds to the PVI with
$F(\x,\y)
=  \Sigma \y - \x_m \, \rr + \nu \left( \ee_n^\top \y - 1 \right) \ee_n.$

We consider an optimal parameter selection model based on the above PVI to find the best parameters $\x$ for portfolio selection.
Following~\cite{chen2018smoothing,steinbach2001markowitz},
we focus on the case where $\Sigma$ is positive definite for risky assets.  We use the Sharpe ratio
$
\text{SR} = \tfrac{\rr^\top \y}{\sqrt{\y^\top \Sigma \y}},
$
characterizing the ratio of return and risk,
as the measure of portfolio selection quality.
Hence, our model for a fund manager to optimize weight constraints
and risk-tolerance~$\eta$ to maximize SR,
can be defined as
\begin{equation}\label{eq:SR bilevel}
	\begin{aligned}
		\min\limits_{\x \in \X, \ \y} \  & - \tfrac{\rr^\top \y}{\sqrt{\y^\top \Sigma \y}} \\\text{ s.t. }\ & \y = \Proj_{\Omega(\x)} \left(\y - \delta F(\x,\y) \right).
	\end{aligned}
\end{equation}

\subsubsection{Data sets}

We use data sets from the standard OR-library over 291 weeks (Data 1-2),
China A-share market over 1940 days (Data 3-4),
and CSI300 index over 583 days (Data 5-6),
summarized in Table~\ref{table data sets 0}.
For the A-share and CSI300 data,
we delete the stocks with more than 5\% zero daily returns to improve numerical stability,
obtaining 457 and 131 liquid (actively traded) stocks as Data 3 and Data~5.
We further sort them by average return over the days and take the top 100 and top 75 stocks as Data~4 and Data~6, respectively.

\begin{table}
\centering
\caption{Description of data sets.
\label{table data sets 0}}
{{\setlength{\tabcolsep}{1pt} \begin{tabular}{c||c|c|c|c|c}
				\hline
				Data sets
				& $n$
				& Source
				& Index
				& $J$
				& Description
				\\ \hline
				Data 1 & 31 & \text { HK } & \text { Hang Seng } & \multirow{2}{*}{291} & Weekly prices
				\\
				\cline{1-4}
				Data 2 & 225 & \text { Japan } & \text { Nikkei 225 } &  & (1992 to 1997)
				\\
				\hline
				Data 3 & 457 & \multirow{4}{*}{China} & \multirow{2}{*}{A-share} & \multirow{2}{*}{1940} & Daily prices
				\\
				\cline{1-2}
				Data 4 & 100 &  &  &  & (2010 to 2017)
				\\
				\cline{1-2}
				\cline{4-6}
				Data 5 & 131 &  & \multirow{2}{*}{CSI300} & \multirow{2}{*}{583} & Daily prices
				\\
				\cline{1-2}
				Data 6 & 75 &  & & & (2022 to 2024)
				\\
				\hline
			\end{tabular}
		}}
\end{table}

For each data set, denote the stock closing prices as $\left\{S_{j,i}: \ j \in [J], i \in [n] \right\}$. We calculate the returns on each stock as
$r_{j,i} = \ln \frac{S_{j+1,i}}{S_{j,i}}$ for $j \in [J-1], \ i \in [n]$,
and obtain a return matrix~$\Xi \in \R^{(J-1) \times n}$
with elements $r_{j,i}$. We partition each data set into a training set and a testing set at the ratio of 9:1.
The training set, also called in-sample set, consists of the first~9/10 rows of $\Xi$ as $\Xi_{in}$,
used to compute the $\rr$ and~$\Sigma$ in~\eqref{eq:SR bilevel}
to solve the parameter selection model
with optimal parameter $\x^*$ and the corresponding optimal portfolio selection $\y^*$.
The testing set, also called out-of-sample set, contains the last 1/10 rows of~$\Xi$ as~$\Xi_{out}$,
used to test the performance of the obtained output.

Specifically, for the training set,
we substitute into~\eqref{eq:SR bilevel} $\rr_{in}$ and $\Sigma_{in}$, the mean of each column in $\Xi_{in}$
and the covariance matrix (centered by $\rr_{in}$) plus $ 10^{-4} I_n$, where the regularization is to guarantee the positive definiteness of $\Sigma_{in}$ and improve numerical stability.
For the testing set,
we use $\rr_{out}$ and $\Sigma_{out}$, the mean of each column in $\Xi_{out}$ and the covariance matrix centered by $\rr_{out}$,
to measure performance
in the sense of two metrics: Sharpe ratio (SR) and cumulative return (CR) over the entire testing set,
\color{black}
where CR is calculated by $\ee_{T_{out}}^\top\Xi_{out}\boldsymbol{w}^*$, with $T_{out}$ the number of testing observations and $\boldsymbol{w}^*=\y^*/(\ee_n^\top\y^*)$.
\color{black}
Obviously, superior selection strategies yield higher SR and CR.

\color{black}

\color{black}
In all portfolio experiments, we set $\nu=1$ and $\delta=10^{-3}$.
Similar to Remark~\ref{remark: remark PSD} in Section~\ref{sec: toy exp},
this guarantees that Assumption~\ref{assump: foundation}(\ref{assump: contraction Psi})
holds for each data set.
\color{black}
\color{black}

\subsubsection{Compared methods}

We compare SIGA with five methods: `Naive' and `Fix' as standard portfolio-management baselines; MPCC-IP to solve the typical MPCC reformulation;
BLOCC~\cite{jiang2024primal} and SFLCB~\cite{shen2026single} as recent first-order bilevel methods.  The implementations are as follows.

\begin{table}[h]
	\centering
\caption{Out-of-sample SR and CR, together with the computational time (seconds), after 2000 iterations, while BLOCC is terminated early on Data 2, Data 3, and Data 5 because the runtime exceeds 3600 seconds.
		\label{table:portfolio-results}}
{\begin{tabular}{cc|rrrrrr}
				\toprule
				Metric & Method
				& Data 1 & Data 2 & Data 3 & Data 4 & Data 5 & Data 6 \\
				\midrule
				\multirow{6}{*}{SR}
				& Naive
				& -0.0143 & -0.2313 & -0.0373 & 0.0879 & -0.0282 & 7.037e-05 \\
				& Fix
				& -0.0088 & -0.1030 & -0.0370 & 0.0839 & -0.0233 & -0.0012 \\
				& SFLCB
				& 0.0004 & -0.0888 & -0.0362 & 0.0826 & -0.0072 & -0.0097 \\
				& BLOCC
				& 0.0222 & -0.2307 & -0.0372 & 0.0880 & -0.0282 & 3.522e-05 \\
				& \textbf{MPCC-IP}
				& \textbf{0.0309} & \textbf{0.2523} & \textbf{0.0178}
				& \textbf{0.0960} & \textbf{0.0027} & \textbf{0.0027} \\
				& \textbf{SIGA}
				& \textbf{0.0309} & \textbf{0.2525} & \textbf{0.0960}
				& \textbf{0.0960} & \textbf{0.0027} & \textbf{0.0027} \\
				\midrule
				\multirow{6}{*}{CR}
				& Naive
				& -0.0133 & -0.1870 & -0.0720 & 0.1698 & -0.0224 & 5.140e-05 \\
				& Fix
				& -0.0083 & -0.0629 & -0.0709 & 0.1609 & -0.0180 & -0.0009 \\
				& SFLCB
				& 0.0003 & -0.0579 & -0.0688 & 0.1501 & -0.0053 & -0.0072 \\
				& BLOCC
				& 0.0220 & -0.1861 & -0.0720 & 0.1699 & -0.0224 & 2.572e-05 \\
				& \textbf{MPCC-IP}
				& \textbf{0.0299} & \textbf{0.2195} & \textbf{0.0446}
				& \textbf{0.1928} & \textbf{0.0021} & \textbf{0.0021} \\
				& \textbf{SIGA}
				& \textbf{0.0299} & \textbf{0.2197} & \textbf{0.1928}
				& \textbf{0.1928} & \textbf{0.0021} & \textbf{0.0021} \\
				\midrule
				\multirow{2}{*}{Time}
				& MPCC-IP
				& 6.8087 & 76.1183 & 294.7198 & 18.1251 & 34.4077 & 12.3406 \\
				& \textbf{SIGA}
				& \textbf{0.2742} & \textbf{5.3631} & \textbf{47.7353}
				& \textbf{1.2230} & \textbf{0.8448} & \textbf{0.5025} \\
				\bottomrule
			\end{tabular}}
\end{table}

\begin{enumerate}[(i)]
	\item

 SIGA:
	use SIGA (Algorithm~\ref{alg}) with $p=3\times10^{-4}, \mu_0=5\times10^{-8}, \tau_0=10^{-8}, \zeta_0=0.01$,
	and the same CHKS smoothing function as in Section~\ref{sec: toy exp}.

	\item Naive: the naive evenly weighted portfolio, with all weights equal to $\frac{1}{n}$.

	\item Fix:
	use fixed-point iterations to solve PVI with fixed parameters $\aa = \boldsymbol{\mathrm{0}}$, $\bb = \ee_n$, $\eta = 1$.

	\item SFLCB:
	consider $\Omega(\x)$ as two inequality constraints
	$\y- u(\x) \le 0$ and $l(\x) - \y \le 0$,
	with the same parameters as in Section~\ref{sec: toy exp}.

	 \item BLOCC: same as in Section~\ref{sec: toy exp}.

    \item
 MPCC-IP: because $\Sigma$ is positive definite and the lower-level constraints are linear, $\y$ solves~\eqref{eq: MV model lower level} if and only if
    there exist $\tilde{\lambda},\bar{\lambda} \in \R^n$ such that~\eqref{eq: mpcc constraint} holds.
    Similar to Section~\ref{sec: toy exp},
    we replace complementarity by the componentwise Scholtes relaxation
    $\tilde{\lambda_i}(y_i-l_i(\x))\leq10^{-6}$ and
    $\bar{\lambda_i}(u_i(\x)-y_i)\leq10^{-6}$, and solve the resulting nonlinear program using \texttt{cyipopt}.
Note that MPCC-TR in Section~\ref{sec: toy exp} is not applicable here since it fails to output feasible solutions.

\end{enumerate}

All iterative methods start from
$
\x^1=(\boldsymbol{0},\ee_n,0),
\y^1=\frac{1}{n}\ee_n,
$
and terminate after 2000 iterations or 3600 seconds.

\begin{figure}[htbp]
\centering
{\begin{subfigure}[t]{0.49\textwidth}
	\centering
     {\setlength{\fboxsep}{0pt}\includegraphics[width=\dimexpr\linewidth-2\fboxrule\relax]{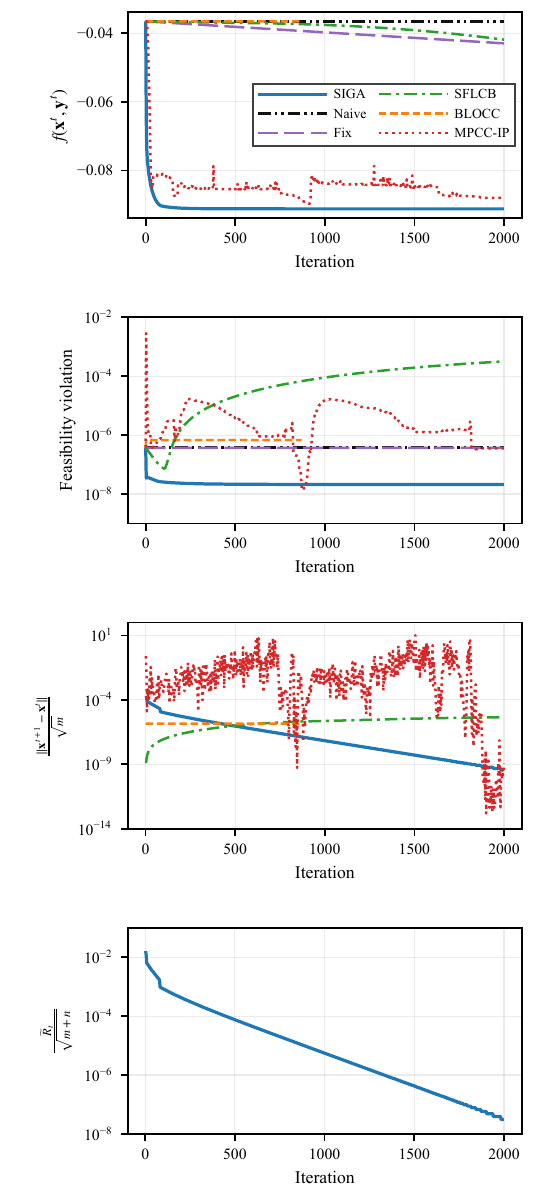}}
\end{subfigure}\hfill
\begin{subfigure}[t]{0.49\textwidth}
	\centering
     {\setlength{\fboxsep}{0pt}\includegraphics[width=\dimexpr\linewidth-2\fboxrule\relax]{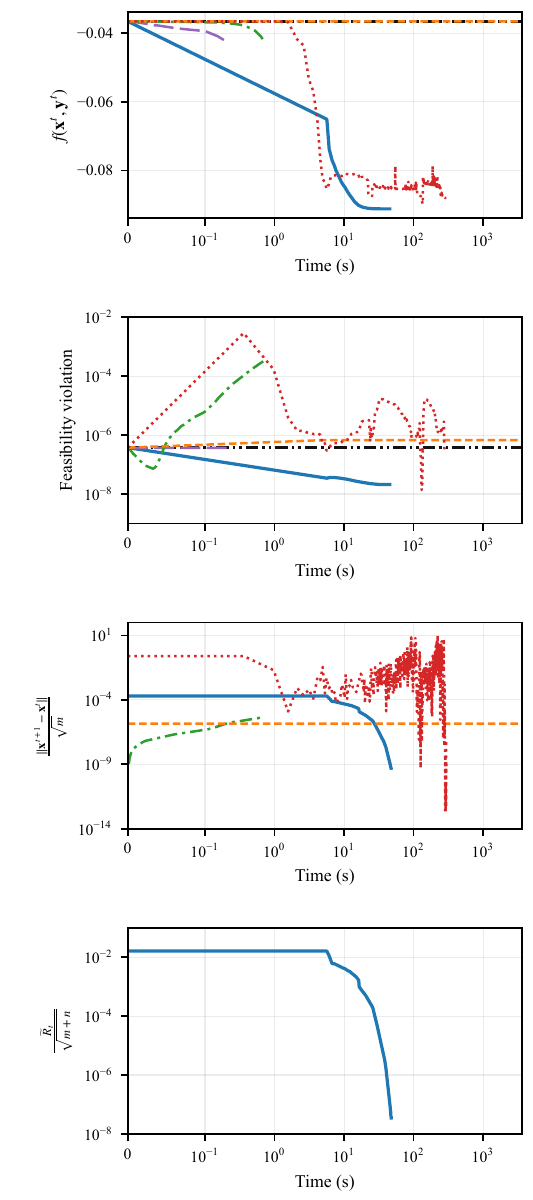}}
 \end{subfigure}}
\caption{Convergence curves on Data~3.
\label{fig:portfolio-data3-profiles}}
\end{figure}

\subsubsection{Experimental results}

Table~\ref{table:portfolio-results} reports the out-of-sample SR and CR,
from which it can be observed that MPCC-IP and SIGA achieve better performance.
Hence, regarding the computational time, we only report MPCC-IP and SIGA,
and SIGA is faster than MPCC-IP on all six data sets.

From Figure~\ref{fig:portfolio-data3-profiles}, it can be observed that SIGA has a better convergence performance, and the last three rows are consistent with the theoretical results in Theorem~\ref{thm: min converge} and Theorem~\ref{thm: feasibility-stationarity-convergence}.
Figure~\ref{fig:portfolio-stationarity-all} collects the SIGA stationarity residual on the six data sets.

\FloatBarrier
\begin{figure}[htbp]
\centering
{{\setlength{\fboxsep}{0pt}\includegraphics[width=\dimexpr\textwidth-2\fboxrule\relax]{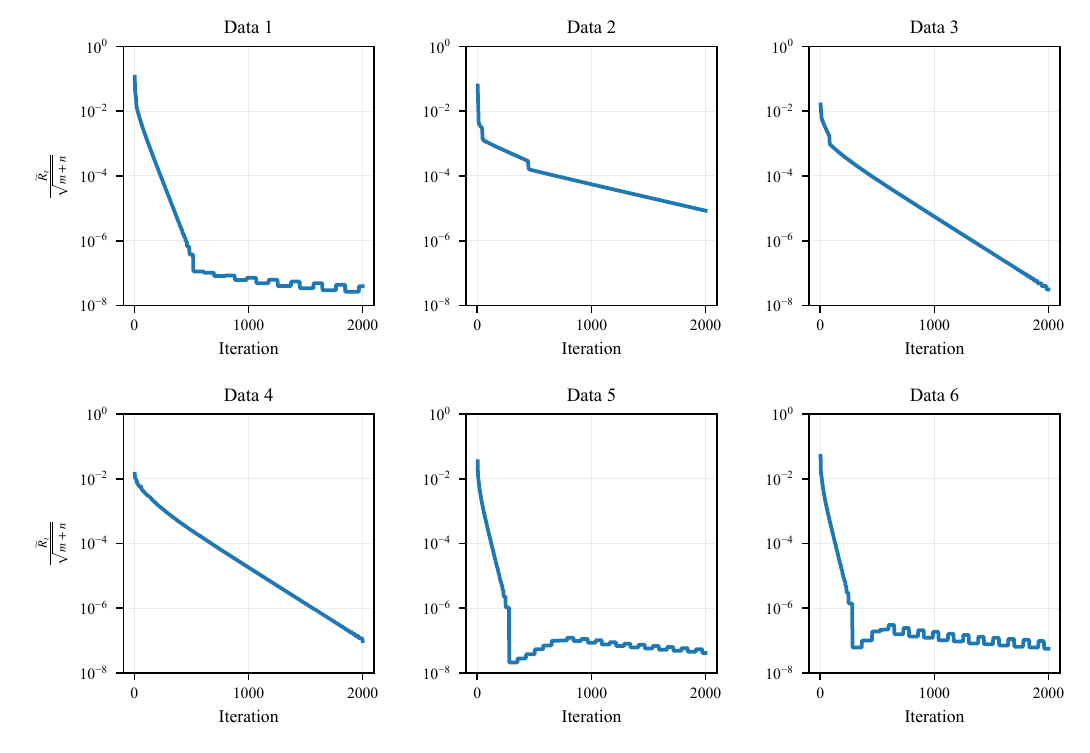}}}
\caption{SIGA stationarity residual $\widetilde R_t/\sqrt{m+n}$ on six data sets.
\label{fig:portfolio-stationarity-all}}
\end{figure}

\FloatBarrier
\section{Conclusions}

This paper investigates problem \eqref{eq:1equation constraint} which has a PVI constraint defined on a moving set. Problem \eqref{eq:1equation constraint} can be classified to MPEC. However, most existing results on MPEC only have parameters in the function of the PVI
with a parameter-independent set.
Problem \eqref{eq:1equation constraint} with the PVI on a moving set has important applications in engineering and finance.
We show that the solution function of the PVI is Lipschitz continuous with respect to the upper-level decision variable $\x$ and the optimal solution set of~\eqref{eq:1equation constraint} is nonempty and bounded. Moreover, Section~\ref{sec:exact penalty} gives interesting results that the constraint system of~\eqref{eq:1equation constraint} satisfies metric regularity automatically, requiring no additional assumptions for stationary point characterization. To overcome the nonsmoothness of the projection operator induced by the moving set, we introduce the smoothing approximation and analyze the consistency between~\eqref{eq:1equation constraint} and the smoothed problem~\eqref{eq:1equation constraint smoothing}.
Based on the smoothing approximation, we propose SIGA with a novel scheme for updating the smoothing parameter and prove the convergence of SIGA to a stationary point of problem~\eqref{eq:1equation constraint}.
Numerical experiments with real data on portfolio management validate the efficient performance of SIGA.

\appendix

\section{PVI on a moving box}\label{sec: smooth box}

To fit our problem setting and illustrate our analysis more clearly,
we consider the  case with the PVI defined on a moving box~$$\Omega(\x) = \{ \y \in \R^n: l(\x) \le \y \le u(\x)\},$$
where
$l,u: \textcolor{black}{\R^m} \rightarrow \R^n$ are continuously differentiable.
By letting
$r = 2n$, $A = \left(I \ -I\right)^\top$,
$\mathbf{c}_i = \min_{\x \in \X} l_i(\x)$ for $i=1,2,\cdots,n$,
and $b(\x) = \left( l(\x) \ -u(\x) \right)$
in Proposition~\ref{prop: Lipschitz proj}, Assumption~\ref{assump: foundation}(\ref{assump: Lipschitz proj}) is satisfied.
Note that we require $l(\x) < u(\x)$ for any $\x \in \textcolor{black}{U_{\X}}$.
Under this moving box scenario,
\begin{equation}\label{eq: moving box mid}
	\Psi(\x,\y) = \Proj_{\Omega(\x)} (\Phi(\x,\y)) = \mid\{l(\x),u(\x),\Phi(\x,\y)\}.
\end{equation}
As discussed in Proposition~\ref{prop: contraction Psi},
if for any $\x \in \X$, $\Phi(\x,\cdot)$ is $\LL$-Lipschitz, where $\LL \in (0,1)$,
then Assumption~\ref{assump: foundation}(\ref{assump: contraction Psi}) is satisfied.

Next we define the specific smoothing approximation of $\Psi$.
To simplify notations,
we respectively denote the scalar-valued functions $\psi_{\mu}$, $\psi$, $\phi$, $\widetilde{l}$, and~$\widetilde{u}$
to be $(\Psi_{\mu})_i$, $\Psi_i$, $\Phi_i$, $l_i$, and~$u_i$
for $i = 1, 2, \cdots, n$.
Then given~$\x$ and~$\y$,
the smoothing approximation $(\Psi_{\mu})_i (\x,\y)$ is defined as
\begin{equation}\label{eq:Psi mu integral}
	\psi_{\mu} (\x,\y)
:= \int_{-\infty}^{\infty} \mid \Big\{ \widetilde{l}(\x), \widetilde{u}(\x), \phi(\x,\y) - \mu t\Big\} \cdot \rho(t) dt.
\end{equation}
Here $\rho(\cdot)$ is a probability density function with finite absolute mean, satisfying that there exist $M_1, M_2 > 0$, such that
\begin{displaymath}	\int_{-\infty}^{\infty} \rho(t) dt = 1, \quad 	\int_{-\infty}^{\infty} |t| \rho(t) dt \le M_1, \quad
	0 \le \rho(t) \le M_2, \forall t \in \R.
\end{displaymath}
\color{black}
Let $\rho(t) = \frac{2}{(t^2+4)^{\frac{3}{2}}}$,
and then~\eqref{eq: chks} as the CHKS smoothing function can be derived in Section~\ref{sec: experiments}~\cite{qi2000new}.
\color{black}

Note that compared to the smoothing method of a VI problem with a fixed box constraint
where $l(\x) \equiv \bar{l} \in \R^n$ and $u(\x) \equiv \bar{u} \in \R^n$ as constant vectors independent of $\x$,
which has been widely studied (see~\cite{chen1998global,gabriel1997smoothing,qi2000new} and~\cite[Chapter 11.8.2]{facchinei2003finite}, for instance),
we extend it to the case where $l$ and $u$ are functions of $\x$.
For such extension,
the approximation enables dynamic handling of moving constraints,
and our analysis also fits the case where the set-valued mapping
$\Omega(\cdot)$ may degenerate to a fixed box. Our analysis also fits the case where $l(\x) \equiv - \infty$ or $u(\x) \equiv + \infty$,
i.e.,
$\Omega(\x) = \{ \y \in \R^n: \y \le u(\x)\}$
or $\Omega(\x) = \{ \y \in \R^n: \y \ge l(\x)\}$.

\color{black}

We denote $L_l$ and $L_u$ as the Lipschitz constants of $l$ and $u$.
To simplify notations, let
\[
\pi_l(\x, \y, \mu):= \frac{\phi(\x,\y) - \widetilde{l}(\x)}{\mu}, \quad
\pi_u(\x, \y, \mu):= \frac{\phi(\x,\y) - \widetilde{u}(\x)}{\mu}.
\]
In the next two lemmas,
we verify Definition~\ref{definition:smoothing} and Assumption~\ref{assumption: jacobian consistency + Lipschitz in mu} on the properties of $\Psi_\mu$ are satisfied under the moving box scenario,
so that the analysis in Section~\ref{sec: smoothing} is applicable to the smoothing function on a moving box as a special case.

\begin{lemma}[Verification of Definition~\ref{definition:smoothing}] \label{lem: smooth Psi}
	Definition~\ref{definition:smoothing} is satisfied
	with the gradient of $ (\Psi_{\mu})_i (\x,\y)$ ($i = 1,2,\cdots,n$) having the form
	\begin{equation*} \left\{
		\begin{aligned}
			\nabla_{\x} \psi_{\mu}(\x,\y)
			= & \int_{- \infty}^{\pi_u(\x, \y, \mu)} \rho(t) dt	\cdot \nabla \widetilde{u}(\x)
			+ \int_{\pi_u(\x, \y, \mu)}^{\pi_l(\x, \y, \mu)}   \rho(t) dt \cdot \nabla_{\x} \phi(\x,\y) \\
			&\;\; + \int_{\pi_l(\x, \y, \mu)}^{+\infty} \rho(t) dt \cdot \nabla \widetilde{l}(\x)
			\ \in \R^m, \\
			\nabla_{\y} \psi_{\mu}(\x,\y)
			= & \int_{\pi_u(\x, \y, \mu)}^{\pi_l(\x, \y, \mu)} \rho(t) dt \cdot \nabla_{\y}  \phi(\x,\y) \ \in \R^n,
		\end{aligned}\right.
	\end{equation*}
	for any $\mu \in (0,1]$, $\x \in \X$, and $\y \in \R^n$.
\end{lemma}
\begin{proof}
	Condition (\ref{assump: Psi mu continuously differentiable}):
	The integral in~\eqref{eq:Psi mu integral} can be separated as
	\[
	\begin{aligned}
		\psi_{\mu} (\x,\y) \! = \! & \int_{- \infty}^{\pi_u(\x, \y, \mu)} \hspace{-1em} \rho(t) dt \! \cdot \! \widetilde{u}(\x)
		\! + \! \int_{\pi_u(\x, \y, \mu)}^{\pi_l(\x, \y, \mu)} \hspace{-0.7em} \left( \phi(\x,\y) - \mu t \right) \! \rho(t) dt
		\! + \! \int_{\pi_l(\x, \y, \mu)}^{+\infty} \hspace{-0.3em} \rho(t) dt \! \cdot \! \widetilde{l}(\x).
	\end{aligned}
	\]
	{\color{black}
		Let
		$
		\mathcal D(s):=\int_{-\infty}^{s}\rho(t)\,dt.
		$
		Since $\rho$ is a probability density function,
		$\mathcal D$ is absolutely continuous and hence continuous.
		Integration by parts gives
		\[
		\psi_\mu(\x,\y)
		=
		\widetilde l(\x)
		+\mu\int_{\pi_u(\x,\y,\mu)}^{\pi_l(\x,\y,\mu)}
		\mathcal D(s)\,ds.
		\]
	\par}

	By calculations based on fundamental theorem of calculus and
	chain rule, after cancellation we obtain the desired result with
	\begin{equation}\label{eq:nabla Psi mu}		\begin{aligned}
			\nabla \psi_{\mu} (\x,\y)
			= & \int_{- \infty}^{\pi_u(\x, \y, \mu)} \rho(t) dt \cdot \left(\nabla \widetilde{u}(\x), \boldsymbol{\mathrm{0}} \right)
			+ \int_{\pi_u(\x, \y, \mu)}^{\pi_l(\x, \y, \mu)}   \rho(t) dt \cdot \nabla \phi(\x,\y) \\
			& \;\; + \int_{\pi_l(\x, \y, \mu)}^{+\infty} \rho(t) dt \cdot \left(\nabla \widetilde{l}(\x), \boldsymbol{\mathrm{0}} \right).
		\end{aligned}
	\end{equation}

	{\color{black}
		Since $\mathcal D$, $\nabla\widetilde l$,
		$\nabla\widetilde u$, and $\nabla\phi$ are continuous,
		the gradient in~\eqref{eq:nabla Psi mu} is jointly continuous
		in $(\x,\y,\mu)$ for $\mu>0$ over $U_{\X}\times\R^n$.
		The integral representation above also gives joint continuity
		of $\psi_\mu$ for $\mu>0$ over $U_{\X}\times\R^n$.

	}

	Condition (\ref{assump: Psi_mu limit}):
	From~\eqref{eq:Psi mu integral}, the condition holds by the
	continuity of $l$, $u$, and $\Phi$.
	{\color{black}
		Indeed, by the nonexpansiveness of the projection onto
		$[\widetilde l(\x),\widetilde u(\x)]$,
		\[
		\left|\psi_\mu(\x,\y)-\psi(\x,\y)\right|
		\le
		\mu\int_{-\infty}^{+\infty}|t|\rho(t)\,dt
		\le M_1\mu.
		\]
		Together with the continuity of $\psi$, this proves
		condition~(\ref{assump: Psi_mu limit}).
	}

	Condition (\ref{assump: Psi_mu contraction and Lip}):
	For any $\mu\in(0,1]$, $\x\in\X$, and $\y,\y'\in\R^n$,
	similar to~\mbox{\cite[Lemma 2]{qi2000new}},
	\[
	\begin{aligned}
		&\left|\psi_\mu(\x,\y)-\psi_\mu(\x,\y')\right|\\
		&\le
		\int_{-\infty}^{+\infty}
		\Big|
		\mid\{\widetilde l(\x),\widetilde u(\x),
		\phi(\x,\y)-\mu t\}
		-
		\mid\{\widetilde l(\x),\widetilde u(\x),
		\phi(\x,\y')-\mu t\}
		\Big|\rho(t)\,dt\\
		&\le
		\int_{-\infty}^{+\infty}
		\left|\phi(\x,\y)-\phi(\x,\y')\right|\rho(t)\,dt\\
		&=
		\left|\phi(\x,\y)-\phi(\x,\y')\right|.
	\end{aligned}
	\]
	Since for any $\x\in\X$, $\Phi(\x,\cdot)$ is $\LL$-Lipschitz,
	then $\Psi_\mu(\x,\cdot)$ is $L_s$-Lipschitz with
	$L_s\le\LL$.
\end{proof}

\begin{remark}
	In Lemma~\ref{lem: smooth Psi},
	the terms related to $\nabla \widetilde{u}(\x)$ and $\nabla \widetilde{l}(\x)$ indicate the impact of moving box compared to the fixed box.
	These terms, along with the terms $\widetilde{l}(\x)$ in $\pi_l(\x, \y, \mu)$
	and $\widetilde{u}(\x)$ in $\pi_u(\x, \y, \mu)$,
	also significantly increase the difficulty of the analysis in the following lemmas.
\end{remark}

\begin{lemma}[Verification of Assumption~\ref{assumption: jacobian consistency + Lipschitz in mu}]
	\label{lem: smooth Psi 2}

	For any $\x \in \X$ and $\y \in \R^n$,
	$$
	\lim\limits_{\x^t \rightarrow \x, \y^t \rightarrow \y, \mu_t \downarrow 0} {\rm dist}\left( \nabla \Psi_{\mu_t} (\x^t,\y^t)^\top, \partial_C \Psi (\x,\y) \right) = 0.
	$$
\end{lemma}
\begin{proof}
\color{black}

	 Inspired by~\mbox{\cite[p.~524]{chen1998global}}\hskip0pt,
	for any $\x \in \X$  and  $\y \in \R^n$,
	from~\eqref{eq: moving box mid},
	\begin{equation}\label{eq: subdifferential}{\color{black}		\partial \psi (\x,\y)=
		\left\{ \begin{aligned}
			&\left\{ \left(\nabla \widetilde{l}(\x), \boldsymbol{\mathrm{0}} \right) \right\} && \hspace{-1em} \text{ if } \phi(\x,\y) < \widetilde{l}(\x)  \\
			&\left\{ \bar{c} \cdot \left(\nabla \widetilde{l}(\x), \boldsymbol{\mathrm{0}} \right)
			+ (1 - \bar{c}) \cdot \nabla \phi(\x,\y): \bar{c} \in [0,1] \right\} && \hspace{-1em} \text{ if } \phi(\x,\y) = \widetilde{l}(\x)  \\
			&\left\{ \nabla \phi(\x,\y) \right\} && \hspace{-4.2em} \text{ if } \widetilde{l}(\x) < \phi(\x,\y) < \widetilde{u}(\x)  \\
			&\left\{ \bar{c} \cdot \left(\nabla \widetilde{u}(\x), \boldsymbol{\mathrm{0}} \right)
			+ (1 - \bar{c}) \cdot \nabla \phi(\x,\y): \bar{c} \in [0,1] \right\} && \hspace{-1em} \text{ if } \phi(\x,\y) = \widetilde{u}(\x) \\
			&\left\{ \left(\nabla \widetilde{u}(\x), \boldsymbol{\mathrm{0}} \right) \right\} && \hspace{-1em} \text{ if } \phi(\x,\y) > \widetilde{u}(\x).
		\end{aligned}
		\right.
	}\end{equation}For any sequence $\{(\x^t,\y^t,\mu_t)\}$ such that $\x^t \rightarrow \x, \y^t \rightarrow \y, \mu_t \downarrow 0$, from~\eqref{eq:nabla Psi mu}
	and the continuous differentiability of $l$, $u$, and $\Phi$,
	we have $\nabla \psi_{\mu_t} (\x^t,\y^t)$ is bounded in $t$.

\color{black}
 Next we prove that any accumulation point of $\left\{ \nabla \psi_{\mu_t} (\x^t,\y^t) \right\}$ belongs to the set $\partial \psi (\x,\y)$.
	Consider the case where $\phi(\x,\y) = \widetilde{l}(\x)$.
	The accumulation point of $\pi_l(\x^t, \y^t, \mu_t)$ 	has three possible values: $\widetilde{c} \in \R$, $+\infty$, and  $-\infty$.
	Suppose $\{(\x^{t_j},\y^{t_j}, \mu_{t_j})\}$ to be any subsequence of $\{(\x^t,\y^t,\mu_t)\}$
	such that
	\begin{equation}\label{eq:jacobian consistency subsequence}		\lim\limits_{\x^{t_j} \rightarrow \x, \y^{t_j} \rightarrow \y, \mu_{t_j} \downarrow 0}
		\pi_l(\x^{t_j}, \y^{t_j}, \mu_{t_j})
	\end{equation}takes these three values.
Since $\widetilde{l}(\x) < \widetilde{u}(\x)$ for any $\x \in \X$, we have
	\[
	\lim\limits_{\x^{t_j} \rightarrow \x, \y^{t_j} \rightarrow \y, \mu_{t_j} \downarrow 0}
	\pi_u(\x^{t_j}, \y^{t_j}, \mu_{t_j})
	= -\infty.
	\]If~\eqref{eq:jacobian consistency subsequence} is $\widetilde{c} \in \R$, then from~\eqref{eq:nabla Psi mu} and the continuous differentiability of $l$, $u$,  and
 $\Phi$, \[
	\begin{aligned}
		\lim_{\x^{t_j} \rightarrow \x, \y^{t_j} \rightarrow \y, \mu_{t_j} \downarrow 0} \!\!\!\!\!
		\nabla \psi_{\mu_{t_j}} (\x^{t_j},\y^{t_j}) = &
		\int_{-\infty}^{\widetilde{c}}  \!\! \rho(t) dt \cdot \nabla \phi(\x,\y)
		+ \int_{\widetilde{c}}^{+\infty} \!\!\! \rho(t) dt \! \cdot \!\! \left(\nabla \widetilde{l}(\x), \boldsymbol{\mathrm{0}} \right) \\= & \ \bar{c} \cdot \left(\nabla \widetilde{l}(\x), \boldsymbol{\mathrm{0}} \right)
		+ (1 - \bar{c}) \cdot \nabla \phi(\x,\y),
	\end{aligned}
	\]
 where $\bar{c}:= \int_{\widetilde{c}}^{+\infty} \rho(t) dt \in [0,1]$
	and $\int_{-\infty}^{\widetilde{c}}   \rho(t) dt  = 1 - \bar{c}$;
	similarly, if~\eqref{eq:jacobian consistency subsequence} is $+\infty$,
	the limit
takes $\nabla \phi(\x,\y)$;
	and if~\eqref{eq:jacobian consistency subsequence} is $-\infty$, the limit takes $\left(\nabla \widetilde{l}(\x), \boldsymbol{\mathrm{0}} \right)$.
To sum up, if $\phi(\x,\y) = \widetilde{l}(\x)$, $
	\lim\limits_{\x^{t_j} \rightarrow \x, \y^{t_j} \rightarrow \y, \mu_{t_j} \downarrow 0} \nabla \psi_{\mu_{t_j}} (\x^{t_j},\y^{t_j})
	\in \partial \psi (\x,\y),
	$
	and the other four cases in~\eqref{eq: subdifferential} parallel similarly.
	Hence, for $i = 1,2,\cdots,n$, \[
	\lim_{\x^t \rightarrow \x, \y^t \rightarrow \y, \mu_t \downarrow 0} {\rm dist}\left( \nabla (\Psi_{\mu_t})_i (\x^t,\y^t), \partial \Psi_i (\x,\y) \right) = 0,
	\]which derives the desired result.
\end{proof}

\color{black}

Next in Lemma~\ref{lem: Verification of Assumption Psi}, we verify Assumption~\ref{assumption: Psi mu}  under the moving box scenario.
Before that,
we provide a preliminary lemma.

\begin{lemma}\label{lem:int - int}
	Suppose for any $\y \in \R^n$, $\Phi(\cdot,\y)$ is $L_1$-Lipschitz over $\X$.
	Then for any $\x, \x' \in~\X, \y, \y' \in \R^n, \mu \in (0,1]$,
	\begin{enumerate}[(i)]
		\item \label{lem:int - int x}
		$\begin{aligned}[t]
			\left| \int_{\pi_u(\x, \y, \mu)}^{\pi_l(\x, \y, \mu)} \rho(t) dt
			- \int_{\pi_u(\x', \y, \mu)}^{\pi_l(\x', \y, \mu)} \rho(t) dt \right|
			\le & \frac{M_2 \left( L_l + L_u + 2 L_1 \right)}{\mu}  \| \x - \x' \|;
		\end{aligned}$
		\vspace{0.5em}
		\item \label{lem:int - int y}
		$\begin{aligned}[t]
			& \left| \int_{\pi_u(\x, \y, \mu)}^{\pi_l(\x, \y, \mu)} \rho(t) dt - \int_{\pi_u(\x, \y', \mu)}^{\pi_l(\x, \y', \mu)} \rho(t) dt \right|
			\le \frac{2 M_2 L}{\mu}  \| \y - \y' \|.
		\end{aligned}$

	\end{enumerate}
\end{lemma}
\begin{proof}

	The result~(\ref{lem:int - int x}) comes from
	$$\begin{aligned}
		& \left| \int_{\pi_u(\x, \y, \mu)}^{\pi_l(\x, \y, \mu)}  \rho(t) dt
		- \int_{\pi_u(\x', \y, \mu)}^{\pi_l(\x', \y, \mu)} \rho(t) dt \right|
		= \left| \int_{\pi_u(\x, \y, \mu)}^{\pi_u(\x', \y, \mu)} \rho(t) dt
		- \int_{\pi_l(\x, \y, \mu)}^{\pi_l(\x', \y, \mu)} \rho(t) dt
		\right| \\
		\le & \frac{M_2}{\mu} \Big( 2 \left| \phi(\x,\y) - \phi(\x',\y) \right|
		+ \left| \widetilde{u}(\x) - \widetilde{u}(\x') \right|
		+ | \widetilde{l}(\x) - \widetilde{l}(\x') | \Big),
	\end{aligned}
	$$
 while~(\ref{lem:int - int y}) follows similarly.
\end{proof}

\begin{lemma}[Verification of Assumption~\ref{assumption: Psi mu}(\ref{lem: Psi - Psi' x})(\ref{lem: Psi - Psi' mu})(\ref{lem: norm of partial Psi * v 0})]\label{lem: Verification of Assumption Psi}

	\sloppy Suppose for any $\y \in \R^n$, $\Phi(\cdot,\y)$ is $L_1$-Lipschitz, matrix-valued function $\nabla \Phi(\cdot,\cdot)$ is $L_{\Phi}$-Lipschitz,
	$\nabla l(\cdot)$ is $L_{l_\x}$-Lipschitz and $\nabla u(\cdot)$ is $L_{u_\x}$-Lipschitz over $\X$, under spectral norm.
	Then
	Assumption~\ref{assumption: Psi mu}(\ref{lem: Psi - Psi' x}) holds
	with $\widetilde{C} = n (L_l + L_u + L_1)$,
	(\ref{lem: Psi - Psi' mu}) holds with
	$\overline{C} = n M_1$,
	and (\ref{lem: norm of partial Psi * v 0}) holds with
	\[
	\begin{aligned}
		\widehat{C} = & \ n M_2 \cdot
		\Big( 2 \LL \left( L_l + L_u + 2 L_1 \right) + 2 \LL^2 + \left(L_u^2 + 2 L_u L_1 + 2 L_1^2 + 2 L_l L_1 + L_l^2\right) \Big) \\
		& + n \cdot \big(4 L_{\Phi} + L_{l_\x} + L_{u_\x} \big).
	\end{aligned}
	\]
\end{lemma}
\begin{proof}

\color{black}

 Verifying  Assumption~\ref{assumption: Psi mu}(\ref{lem: Psi - Psi' mu}) is similar to Lemma~\ref{lem: smooth Psi} to verify Definition~\ref{definition:smoothing}(\ref{assump: Psi_mu contraction and Lip}),
 while Assumption~\ref{assumption: Psi mu}(\ref{lem: Psi - Psi' x}) holds,
since it is easy to obtain that
given any two sets of numbers $\{a_1,a_2,a_3\}$ and $\{\tilde{a}_1,\tilde{a}_2,\tilde{a}_3\}$,
 \[
 \Big| \mid \{a_1,a_2,a_3\} - \mid \{\tilde{a}_1,\tilde{a}_2,\tilde{a}_3\} \Big|
 \le |a_1 - \tilde{a}_1| + |a_2 - \tilde{a}_2| + |a_3 - \tilde{a}_3|.
 \]

	For Assumption~\ref{assumption: Psi mu}(\ref{lem: norm of partial Psi * v 0}),
	from Lemma~\ref{lem: smooth Psi},
	\[
	\begin{aligned}
		&\| \nabla_{\y}  \psi_{\mu} (\x,\y)  - \nabla_{\y}  \psi_{\mu}  (\x',\y) \| \\
		= & \left\| \int_{\pi_u(\x, \y, \mu)}^{\pi_l(\x, \y, \mu)} \rho(t) dt \cdot \nabla_{\y}  \phi(\x,\y)
		- \int_{\pi_u(\x', \y, \mu)}^{\pi_l(\x', \y, \mu)} \rho(t) dt \cdot \nabla_{\y}  \phi(\x',\y) \right\| \\
		\le & \left\|  \nabla_{\y} \phi(\x,\y) -  \nabla_{\y} \phi(\x',\y)
		\right\|
		+ \left| \int_{\pi_u(\x, \y, \mu)}^{\pi_l(\x, \y, \mu)} \rho(t) dt - \int_{\pi_u(\x', \y, \mu)}^{\pi_l(\x', \y, \mu)} \rho(t) dt \right| \left\| \nabla_{\y} \phi(\x',\y)
		\right\| \\
		\le & L_{\Phi} \| \x - \x' \|
		+ \tfrac{M_2 \left( L_l + L_u + 2 L_1 \right)}{\mu}  \| \x - \x' \| \cdot \LL,
	\end{aligned}
	\]
	where the last inequality is from Lemma~\ref{lem:int - int}(\ref{lem:int - int x}).
	Then we have
	\[
	\left\| \nabla_{\y} \Psi_{\mu}(\x,\y) - \nabla_{\y} \Psi_{\mu}(\x',\y) \right\| \le n \left( L_{\Phi} + \tfrac{M_2 \LL \left( L_l + L_u + 2 L_1 \right)}{\mu}  \right)  \| \x - \x'\|.
	\]
	Similarly, from Lemma~\ref{lem: smooth Psi} and
	Lemma~\ref{lem:int - int}, we further obtain
	\[
	\begin{aligned}
		&\left\| \nabla_{\y} \Psi_{\mu}(\x,\y) - \nabla_{\y} \Psi_{\mu}(\x,\y') \right\|
		\le n \left( L_{\Phi} + \tfrac{2 M_2 L^2}{\mu} \right) \| \y - \y'\|, \\
		&\left\| \nabla_{\x} \Psi_{\mu}(\x,\y) - \nabla_{\x} \Psi_{\mu}(\x',\y) \right\|
		\le \! n \! \left( \!\!  L_{l_\x} \!\! + \!\!  L_{\Phi} \!\! + \!\!  L_{u_\x} \!\!\!  + \!\!  \tfrac{M_2 \left( L_u^2 \! + 2 L_u \! L_1 + 2 L_1^2 + 2 L_l \! L_1 + L_l^2 \right) }{\mu}
		\!\!  \right) \!\! \| \x \! - \! \x'\|, \\
		&\left\| \nabla_{\x} \Psi_{\mu}(\x,\y) - \nabla_{\x} \Psi_{\mu}(\x,\y') \right\|
		\le  n \Big( L_{\Phi} + \tfrac{M_2 L \left( L_l + L_u +  2 L_1 \right)}{\mu}  \Big) \| \y - \y'\|,
	\end{aligned}
	\]
 which derives the desired result.

\end{proof}

\section{Lemmas for convergence analysis in Section~\ref{sec: SIGA}}

\subsection{Lemmas of properties of $\yy_{\mu}(\x)$, $\vv_{\mu}(\x)$, and $h_{\mu}(\x)$} \label{sec:appendix properties lemmas}

To begin with, Assumption~\ref{assumption: Psi mu}(\ref{lem: Psi - Psi' x}) and Definition~\ref{definition:smoothing}(\ref{assump: Psi_mu contraction and Lip}) easily derive the following lemma.
\begin{lemma} \label{lem: blanket Lip nabla Psi * v} Under Assumption~\ref{assumption: Psi mu},
	for any $\x \in \X, \y \in \R^n, \v \in \R^n$ and $\mu \in (0,1]$, the following three inequalities hold:
	\[
	\begin{aligned}
		\left\| \left(\nabla_{\x} \Psi_{\mu} (\x,\y) \right)^\top \v \right\|
		&\le \widetilde{C} \| \v \|,  \\
		\left\| \left(\nabla_{\y} \Psi_{\mu} (\x,\y)\right)^\top \v \right\|
		&\le L_s \| \v \|, \\
		\left\| \left( I - \nabla_{\y} \Psi_{\mu} (\x,\y) \right)^\top \v \right\|
		&\ge (1 - L_s) \| \v \|.
	\end{aligned}
	\]
\end{lemma}

\begin{lemma}\label{lem: y mu x}

	Under Assumption~\ref{assumption: Psi mu},
	for any $\x, \x' \in \X$ and $\mu, \mu' \in (0,1]$,
	\begin{enumerate}[(i)]
		\item \label{lem: y(x) - y(x')}
		$\| \yy_{\mu}(\x) - \yy_{\mu}(\x') \| \le \frac{\widetilde{C}}{1 - L_s} \| \x - \x' \|$;
		\item \label{lem: y_mu - y_mu'}
		$\|\yy_\mu(\x) - \yy_{\mu'}(\x)\|
		\le
		\frac{\overline{C}}{1 - L_s }
		|\mu - \mu'|.$
	\end{enumerate}
\end{lemma}
\begin{proof}
\color{black}

	To prove~(\ref{lem: y(x) - y(x')}),
	from Definition~\ref{definition:smoothing}(\ref{assump: Psi_mu contraction and Lip})
	and Assumption~\ref{assumption: Psi mu}(\ref{lem: Psi - Psi' x}), we have
	\[
	\begin{aligned}
		\| \yy_{\mu}(\x) - \yy_{\mu}(\x') \| \le \ & \| \Psi_{\mu}(\x, \yy_{\mu}(\x)) \! - \! \Psi_{\mu}(\x, \yy_{\mu}(\x')) \|
		\! + \! \| \Psi_{\mu}(\x, \yy_{\mu}(\x')) \! - \! \Psi_{\mu}(\x', \yy_{\mu}(\x')) \| \\\le \ & L_s \| \yy_{\mu}(\x) - \yy_{\mu}(\x') \| + \widetilde{C} \| \x - \x' \|,
	\end{aligned}
	\]
	which derives the desired result.
 The result  (\ref{lem: y_mu - y_mu'}) parallels from Definition~\ref{definition:smoothing}(\ref{assump: Psi_mu contraction and Lip}) and Assumption~\ref{assumption: Psi mu}(\ref{lem: Psi - Psi' mu}).

\relax \relax \end{proof}
\begin{lemma}\label{lem: y v bounded}
	Under Assumption~\ref{assumption: Psi mu},
	there exists $\overline{M} > 0$ such that for any $\x \in \X$ and $\mu \in (0,1]$, $\|\vv_{\mu}(\x)\| \le \tfrac{1}{1 - L_s} \left\| \nabla_{\y} f(\x,\yy_{\mu}(\x)) \right\|
	\le \overline{M}.$ \end{lemma}
\begin{proof}
\color{black}

	From~\eqref{eq:v mu x} and Lemma~\ref{lem: blanket Lip nabla Psi * v},
	\[
	\left\| \nabla_{\y} f(\x,\yy_{\mu}(\x)) \right\|
	= \left\| \left(I - \nabla_{\y} \Psi_{\mu}(\x,\yy_{\mu}(\x))\right)^\top \vv_{\mu}(\x) \right\|
	\ge (1 - L_s) \| \vv_{\mu}(\x) \|.
	\]
 Then the desired result is obtained from the continuity of $\nabla_{\y} f$ and Lemma~\ref{lem:y mu x}(\ref{lem: y mu x bounded}).
\relax

\relax \end{proof}

\begin{lemma}\label{lem: v mu x}

	Under Assumption~\ref{assumption: Psi mu},
	for any $\x, \x' \in \X$ and $\mu, \mu' \in (0,1]$,
	$$\|\vv_{\mu}(\x) - \vv_{\mu}(\x')\|
	\le \tfrac{C_{\v}}{\mu}
	\| \x - \x' \|,$$
	where $
	C_{\v}:=
	\tfrac{L_f + \widehat{C} \overline{M} }{1 - L_s} \left(1 + \tfrac{\widetilde{C}}{1 - L_s} \right),
	$
	with $\overline{M}$ defined in Lemma~\ref{lem: y v bounded}.
\end{lemma}
\begin{proof}
\color{black}

	From~\eqref{eq:v mu x}, we have
	 \[
	\begin{aligned}
		&\left(I - \nabla_{\y} \Psi_{\mu}(\x,\yy_{\mu}(\x))\right)^\top \left(\vv_{\mu}(\x)  - \vv_{\mu}(\x')\right) \\
		= & - \left(\nabla_{\y} f(\x,\yy_{\mu}(\x)) \! - \! \nabla_{\y} f(\x',\yy_{\mu}(\x')) \right)
		+ \left(\nabla_{\y} \Psi_{\mu}(\x,\yy_{\mu}(\x)) \!
		- \! \nabla_{\y} \Psi_{\mu}(\x',\yy_{\mu}(\x'))\right)^\top \! \vv_{\mu}(\x').
	\end{aligned}
	\]
	Thus, from Lemma~\ref{lem: blanket Lip nabla Psi * v} and
	Assumption~\ref{assumption: Psi mu}(\ref{lem: norm of partial Psi * v 0})(\ref{assumption Lf Lipschitz}),
	\[
	\begin{aligned}
		(1 - L_s) \| \vv_{\mu}(\x) - &\vv_{\mu}(\x') \|
		\le L_f \left(\| \x-\x' \| + \| \yy_\mu(\x) - \yy_\mu(\x') \| \right) \\
		& + \left( \tfrac{\widehat{C}}{\mu} \left(\| \x - \x'\| + \| \yy_\mu(\x) - \yy_\mu(\x') \| \right)   \right) \|\vv_{\mu}(\x')\|.
	\end{aligned}
	\]
	 Then the desired result comes from Lemma~\ref{lem: y mu x}  (\ref{lem: y(x) - y(x')}) and Lemma~\ref{lem: y v bounded}.

\end{proof}

\begin{lemma}\label{lem: nabla h(x) - nabla h(x')}

	Under Assumption~\ref{assumption: Psi mu},
	for any $\x,\x' \in \X$ and $\mu \in (0,1]$,
	$$
	\| \nabla h_{\mu} (\x) - \nabla h_{\mu} (\x') \| \le \tfrac{C_{h}}{\mu}
	\|\x - \x'\|,
	$$
	where $C_h:=
	\widetilde{C} C_{\v} + \left( L_f + \widehat{C} \overline{M} \right)\left(1 + \tfrac{\widetilde{C}}{1-L_s} \right)$,
	with $C_{\v}$ defined in Lemma~\ref{lem: v mu x}
	and $\overline{M}$ in Lemma~\ref{lem: y v bounded}.
\end{lemma}
\begin{proof}
\color{black}

	From~\eqref{eq:nabla h_mu(x) 2} and Assumption~\ref{assumption: Psi mu}(\ref{assumption Lf Lipschitz}),
	we have
	\[
	\begin{aligned}
		& \| \nabla h_{\mu} (\x) - \nabla h_{\mu} (\x') \| \\
		\le & L_f \|\x - \x'\| + L_f \| \yy_{\mu}(\x) - \yy_{\mu}(\x')\|
		+ \left\| \left(\nabla_{\x} \Psi_{\mu}(\x,\yy_{\mu}(\x))\right)^\top (\vv_{\mu}(\x) - \vv_{\mu}(\x'))  \right\| \\
		& + \left\|
		\left(\nabla_{\x} \Psi_{\mu}(\x,\yy_{\mu}(\x)) - \nabla_{\x} \Psi_{\mu}(\x',\yy_{\mu}(\x'))\right)^\top \vv_{\mu}(\x') \right\|,
	\end{aligned}
	\]
	where the second term can be upper bounded by Lemma~\ref{lem: y mu x}(\ref{lem: y(x) - y(x')}),
	the third term by Lemma~\ref{lem: blanket Lip nabla Psi * v} and  Lemma~\ref{lem: v mu x},
	and the last term by Assumption~\ref{assumption: Psi mu}(\ref{lem: norm of partial Psi * v 0}),
	Lemma~\ref{lem: y mu x}(\ref{lem: y(x) - y(x')}), and Lemma~\ref{lem: y v bounded}.
 Then the desired result is obtained.

\end{proof}

\subsection{Lemmas of descent formula}\label{sec:appendix descent formula lemmas}

\begin{lemma}
	\label{lem:descent h}
	Let $\{(\x^t,\y^t,\v^t)\}$ be a sequence generated by Algorithm~\ref{alg}.
	Under Assumption~\ref{assumption: Psi mu},
	\[
	\begin{aligned}
		&h_{\mu_{t+1}}(\x^{t+1})  -  h_{\mu_t}(\x^t)  \\
		\le &- \frac{1}{2} \left(\frac{1}{\zeta_t} - \frac{C_{h}}{\mu_t} \right) \|\x^{t+1} - \x^t\|^2
		+ \left( L_f^2 + 2 \widehat{C}^2 \overline{M}^2 \right) \cdot \frac{\zeta_t }{\mu_t^2} \cdot  \| \y^t - \yy_{\mu_t}(\x^t) \|^2 \\
		& + 2 \widetilde{C}^2 \zeta_t \| \v^t - \vv_{\mu_t}(\x^t) \|^2
		+ \overline{C} \, \overline{M}   |\mu_t -\mu_{t+1}| + \tfrac{\overline{C}^2 L_f}{2 (1 - L_s)^2 }
		|\mu_t -\mu_{t+1}|^2,
	\end{aligned}
	\]
	with $C_{h}$ defined in Lemma~\ref{lem: nabla h(x) - nabla h(x')}
	and $\overline{M}$ in Lemma~\ref{lem: y v bounded}.
\end{lemma}
\begin{proof}
	To begin with, from~\cite[Lemma 5.7]{beck2017first} and Assumption~\ref{assumption: Psi mu}(\ref{assumption Lf Lipschitz}),
	\begin{equation}\label{eq:descent h_mu(x) 1}
		\begin{split}
			& h_{\mu_{t+1}}(\x^{t+1}) - h_{\mu_t}(\x^{t+1})
			= f(\x^{t+1}, \yy_{\mu_{t+1}}(\x^{t+1})) - f(\x^{t+1}, \yy_{\mu_t}(\x^{t+1}))  \\
			\le & \left\langle \nabla_{\y} f\left(\x^{t+1}, \yy_{\mu_{t+1}}(\x^{t+1})\right), \yy_{\mu_{t+1}}(\x^{t+1}) \! - \! \yy_{\mu_t}(\x^{t+1}) \right\rangle
			+ \! \tfrac{L_f}{2} \! \| \yy_{\mu_{t+1}}(\x^{t+1}) \! - \! \yy_{\mu_t}(\x^{t+1}) \|^2 \\
			\le & \overline{C} \, \overline{M}  |\mu_t -\mu_{t+1}| + \tfrac{\overline{C}^2 L_f}{2 (1 - L_s)^2 }
			|\mu_t -\mu_{t+1}|^2,
		\end{split}
	\end{equation}
	where
	the last inequality is from Lemma~\ref{lem: y mu x}(\ref{lem: y_mu - y_mu'})
	and Lemma~\ref{lem: y v bounded}.

	According to Lemma~\ref{lem: nabla h(x) - nabla h(x')}, for any given $\mu$,
	$h_{\mu}(\cdot)$ is $\frac{C_{h}}{\mu}$-smooth over $\X$,
	so from~\cite[Lemma 5.7]{beck2017first} again,
	we have
	\begin{equation}\label{eq:descent h_mu(x) 2}
		h_{\mu_t}(\x^{t+1}) - h_{\mu_t}(\x^t)
		\le \left\langle \nabla h_{\mu_t} (\x^t), \x^{t+1} - \x^t\right\rangle + \frac{C_{h}}{2 \mu_t} \|\x^{t+1} - \x^t\|^2.
	\end{equation}
	From~\eqref{eq:x iteration}, since $\X$ is a convex set,
	by the property of the projection operator,
	$$
	\begin{aligned}
		0 & \ge
		\left\langle \x^t - \zeta_t \Big(\nabla_{\x} f(\x^t,\y^t)
		+	\left( \nabla_{\x} H_{\mu_t}(\x^t,\y^t)
		\right)^\top \v^t \Big)  - \x^{t+1}, \x^t - \x^{t+1}  \right\rangle, \end{aligned}
	$$
	which derives
	\[
	0 \le - \left\langle \nabla_{\x} f(\x^t,\y^t)
	+	\left( \nabla_{\x} H_{\mu_t}(\x^t,\y^t)
	\right)^\top \v^t, \x^{t+1} - \x^t \right\rangle
	- \frac{1}{\zeta_t} \| \x^{t+1} - \x^t \|^2. \\
	\]
	Combining the above inequality with~\eqref{eq:descent h_mu(x) 2}, we have
	\begin{gather*}\label{eq:descent h_mu(x) 3}
		\begin{aligned}
			& h_{\mu_t}(\x^{t+1}) - h_{\mu_t}(\x^t)  \\
			\le & \left\langle \nabla h_{\mu_t} (\x^t) - \! \nabla_{\x} f(\x^t, \y^t )
			\! - \! \left( \nabla_{\x} H_{\mu_t}(\x^t,\y^t)
			\right)^\top \v^t, \x^{t+1} \! - \! \x^t \right\rangle
			\! + \!  \left( \! \frac{C_{h}}{2 \mu_t}  - \! \frac{1}{\zeta_t} \! \right) \! \|\x^{t + 1} \! - \! \x^t\|^2 \\
			\le
			& \frac{\zeta_t}{2} \left\| \nabla h_{\mu_t} (\x^t) - \nabla_{\x} f(\x^t, \y^t )
			\! - \! \left( \nabla_{\x} H_{\mu_t}(\x^t,\y^t)
			\right)^\top \! \v^t \right\|^2
			+ \left( \frac{C_{h}}{2 \mu_t} - \frac{1}{2 \zeta_t} \right) \! \|\x^{t + 1} \! - \! \x^t\|^2,
		\end{aligned}
	\end{gather*}
	where the last inequality follows from Young's inequality with the step size $\zeta_t>0$ for any $t$.
	From Assumption~\ref{assumption: Psi mu}(\ref{lem: norm of partial Psi * v 0})(\ref{assumption Lf Lipschitz}), Lemma~\ref{lem: blanket Lip nabla Psi * v},
	and Lemma~\ref{lem: y v bounded},
	we further have
	\[
	\begin{aligned}
		& \left\| \nabla h_{\mu_t} \! (\x^t) \! - \! \nabla_{\! \x} f(\x^t \!, \y^t )
		\! - \! \left( \nabla_{\! \x} H_{\mu_t}(\x^t \!,\y^t)
		\right)^\top \! \v^t \right\|^2  \\
		= & \big\| \nabla_{\x} f(\x^t\!, \yy_{\mu_t}\! (\x^t)) \! + \! \left( \nabla_{\x} H_{\mu_t}(\x^t \!, \yy_{\mu_t}\! (\x^t))
		\right)^\top \! \vv_{\mu_t}\! (\x^t) \! - \!  \nabla_{\x} f(\x^t \!, \! \y^t)
		\! - \! \left( \nabla_{\x} H_{\mu_t}(\x^t \!, \! \y^t)
		\right)^\top \!\! \v^t \big\|^2 \\
		\le & 2 L_f^2 \| \yy_{\mu_t}(\x^t) - \y^t \|^2
		+ 2 \left\| \left( \nabla_{\x} H_{\mu_t}(\x^t, \yy_{\mu_t}(\x^t))
		\right)^\top  \vv_{\mu_t}(\x^t) - \left( \nabla_{\x} H_{\mu_t}(\x^t, \y^t)
		\right)^\top \v^t \right\|^2 \\
		\le & 2 L_f^2 \| \yy_{\mu_t}(\x^t) - \y^t \|^2
		+ 4 \left\|  \left( \nabla_{\x} \Psi_{\mu_t}(\x^t,\yy_{\mu_t}(\x^t)) - \nabla_{\x} \Psi_{\mu_t}(\x^t,\y^t)
		\right)^\top \vv_{\mu_t}(\x^t)    \right\|^2 \\
		& + 4 \left\|  \left( \nabla_{\x} \Psi_{\mu_t}(\x^t,\y^t)
		\right)^\top \left( \vv_{\mu_t}(\x^t) - \v^t \right)   \right\|^2
		\\
		\le & \left( 2 L_f^2 + \frac{4 \widehat{C}^2 \overline{M}^2}{\mu_t^2}  \right) \| \yy_{\mu_t}(\x^t) - \y^t \|^2
		+ 4 \widetilde{C}^2 \| \vv_{\mu_t}(\x^t) - \v^t \|^2.
	\end{aligned}
	\]
	The proof is completed by combining this with~\eqref{eq:descent h_mu(x) 1} and $\mu_t \le 1$ for any $t$.
\end{proof}

\begin{lemma}\label{lem:descent v}
	Let $\{(\x^t,\y^t,\v^t)\}$ be a sequence generated by Algorithm~\ref{alg}.
	Under Assumption~\ref{assumption: Psi mu},
	\[
	\begin{aligned}
		&\| \y^{t} - \yy_{\mu_t}(\x^t) \|^2
		\le \frac{\tau_t^2}{(1 - L_s)^2},
		&\| \v^{t} - \vv_{\mu_t}(\x^t) \|^2
		\le \frac{2 ( L_f^2 + \widehat{C}^2 \overline{M}^2) }{(1 - L_s)^4} \cdot \frac{\tau_t^2}{\mu_t^2},
	\end{aligned}
	\]
	with $\overline{M}$ defined in Lemma~\ref{lem: y v bounded}.
\end{lemma}
\begin{proof}

\color{black}

	From~\eqref{eq:algorithm sub-step 1}, we have
	\[
	\begin{aligned}
		\| \y^{t} - \yy_{\mu_t}(\x^t) \| = \| \y^t - \Psi_{{\mu}_t}(\x^t, \yy_{\mu_t}(\x^t)) \|
		\le \tau_t + L_s \| \y^{t} - \yy_{\mu_t}(\x^t) \|,
	\end{aligned}
	\]
	and thus obtain the first result.
	For the second result, from~\eqref{eq:v mu x},
	\[
	\nabla_{\y} f(\x^t,\yy_{\mu_t}(\x^t))  +  \left(\nabla_{\y} H_{\mu_t}(\x^t,\yy_{\mu_t}(\x^t))\right)^\top \vv_{\mu_t}(\x^t) = \boldsymbol{\mathrm{0}},
	\]
	and from~\eqref{eq:algorithm sub-step 2} and Lemma~\ref{lem: nonsingular},
 $
	\nabla_{\y} f(\x^t,\y^t)  +  \left(
	\nabla_{\y} H_{\mu_t}(\x^t,\y^t)
	\right)^\top \v^t = \boldsymbol{\mathrm{0}}.
	$
	Then we have
	\[
	{\begin{aligned}
			& \left(I - \nabla_{\y} \Psi_{\mu_t}(\x^t,\y^t)\right)^\top \left(\v^t  - \vv_{\mu_t}(\x^t)\right) \\
			= & - \left(\nabla_{\y} f(\x^t,\y^t) \! - \! \nabla_{\y} f(\x^t,\yy_{\mu_t}(\x^t)) \right)
			\! + \! \left(\nabla_{\y} \Psi_{\mu_t}(\x^t,\y^t) \!
			- \! \nabla_{\y} \Psi_{\mu_t}(\x^t,\yy_{\mu_t}(\x^t))\right)^\top \!\! \vv_{\mu_t}(\x^t).
		\end{aligned}
	}\]
	Thus, from Lemma~\ref{lem: blanket Lip nabla Psi * v} and
	Assumption~\ref{assumption: Psi mu}(\ref{lem: norm of partial Psi * v 0})(\ref{assumption Lf Lipschitz}),
	\[
	{\begin{aligned}
			(1 - L_s) \| \v^t  - \vv_{\mu_t}(\x^t) \|
			\le L_f \| \y^{t} - \yy_{\mu_t}(\x^t) \| + \frac{\widehat{C}}{\mu_t} \left\|\y^{t} - \yy_{\mu_t}(\x^t) \right\| \|\vv_{\mu_t}(\x^t)\|,
		\end{aligned}
	}\]
	and further from Lemma~\ref{lem: y v bounded},
	we have
	\[
	\| \v^{t} - \vv_{\mu_t}(\x^t) \|^2
	\le \frac{2}{(1 - L_s)^2} \left( L_f^2 +  \frac{\widehat{C}^2 \overline{M}^2 }{\mu_t^2} \right) \| \y^{t} - \yy_{\mu_t}(\x^t) \|^2,
	\]
	which derives the desired result along with $\mu_t \le 1$ for any $t$.

\end{proof}

\textbf{Declaration:}
In the preparation of this manuscript, the authors used an AI tool to search for relevant references and enhance readability. All AI-assisted content was carefully reviewed and revised. The authors take full responsibility for all content of this manuscript.

\bibliographystyle{siam}

\bibliography{bibbib_new_toy_202610201644}

@article{benko2021implicit,
	title={On implicit variables in optimization theory},
	author={Benko, Mat{\'u}{\v{s}} and Mehlitz, Patrick},
	journal={J. Nonsmooth Anal. Optim.},
	volume={2},
	eid={7215},
	year={2021}
}

@article{burke2013gradient,
	title={Gradient consistency for integral-convolution smoothing functions},
	author={Burke, James V and Hoheisel, Tim and Kanzow, Christian},
	journal={Set-Valued Var. Anal.},
	volume={21},
	number={2},
	pages={359--376},
	year={2013}
}

@article{hoheisel2013theoretical,
	title={Theoretical and numerical comparison of relaxation methods for mathematical programs with complementarity constraints},
	author={Hoheisel, Tim and Kanzow, Christian and Schwartz, Alexandra},
	journal={Math. Program.},
	volume={137},
	number={1--2},
	pages={257--288},
	year={2013}
}

@article{burke1991exact,
	title={An exact penalization viewpoint of constrained optimization},
	author={Burke, James V},
	journal={SIAM J. Control Optim.},
	volume={29},
	number={4},
	pages={968--998},
	year={1991},
	publisher={SIAM}
}

@inproceedings{jiang2024primal,
	title={A primal-dual-assisted penalty approach to bilevel optimization with coupled constraints},
	author={Jiang, Liuyuan and Xiao, Quan and Tenorio, Victor M and Real-Rojas, Fernando and Marques, Antonio G and Chen, Tianyi},
	booktitle={Advances in Neural Information Processing Systems},
	volume={37},
	pages={95026--95066},
	year={2024}
}

@inproceedings{shen2026single,
	title={A single-loop first-order algorithm for linearly constrained bilevel optimization},
	author={Shen, Wei and Zhang, Jiawei and Huang, Minhui and Shen, Cong},
	booktitle={Advances in Neural Information Processing Systems},
	volume={38},
	pages={57432--57479},
	year={2025}
}

@article{steinbach2001markowitz,
	title={Markowitz revisited: Mean-variance models in financial portfolio analysis},
	author={Steinbach, Marc C},
	journal={SIAM Rev.},
	volume={43},
	number={1},
	pages={31--85},
	year={2001},
	publisher={SIAM}
}

@article{ye2000multiplier,
	title={Multiplier rules under mixed assumptions of differentiability and {L}ipschitz continuity},
	author={Ye, Jane J},
	journal={SIAM J. Control Optim.},
	volume={39},
	number={5},
	pages={1441--1460},
	year={2000},
	publisher={SIAM}
}

@book{mordukhovich2018variational,
	title={Variational Analysis and Applications},
	author={Mordukhovich, Boris Sholimovich},
	volume={30},
	year={2018},
	publisher={Springer}
}

@article{lin2009solving,
	title={Solving stochastic mathematical programs with equilibrium constraints via approximation and smoothing implicit programming with penalization},
	author={Lin, Gui-Hua and Chen, Xiaojun and Fukushima, Masao},
	journal={Math. Program.},
	volume={116},
	number={1},
	pages={343--368},
	year={2009},
	publisher={Springer}
}

@incollection{labbe2021bilevel,
	title={Bilevel network design},
	author={Labb{\'e}, Martine and Marcotte, Patrice},
	booktitle={Network Design with Applications to Transportation and Logistics},
	pages={255--281},
	year={2020},
	publisher={Springer}
}

@article{ye1997necessary,
	title={Necessary optimality conditions for optimization problems with variational inequality constraints},
	author={Ye, Jane J and Ye, X Y},
	journal={Math. Oper. Res.},
	volume={22},
	number={4},
	pages={977--997},
	year={1997},
	publisher={INFORMS}
}

@article{robinson1980strongly,
	title={Strongly regular generalized equations},
	author={Robinson, Stephen M},
	journal={Math. Oper. Res.},
	volume={5},
	number={1},
	pages={43--62},
	year={1980},
	publisher={INFORMS}
}

@article{ye2000constraint,
	title={Constraint qualifications and necessary optimality conditions for optimization problems with variational inequality constraints},
	author={Ye, Jane J},
	journal={SIAM J. Optim.},
	volume={10},
	number={4},
	pages={943--962},
	year={2000},
	publisher={SIAM}
}

@article{gfrerer2017new,
	title={New constraint qualifications for mathematical programs with equilibrium constraints via variational analysis},
	author={Gfrerer, Helmut and Ye, Jane J},
	journal={SIAM J. Optim.},
	volume={27},
	number={2},
	pages={842--865},
	year={2017},
	publisher={SIAM}
}

@article{ye2023difference,
	title={Difference of convex algorithms for bilevel programs with applications in hyperparameter selection},
	author={Ye, Jane J and Yuan, Xiaoming and Zeng, Shangzhi and Zhang, Jin},
	journal={Math. Program.},
	volume={198},
	number={2},
	pages={1583--1616},
	year={2023},
	publisher={Springer}
}

@article{chen2018smoothing,
	title={A smoothing direct search method for {Monte Carlo-based} bound constrained composite nonsmooth optimization},
	author={Chen, Xiaojun and Kelley, C T and Xu, Fengmin and Zhang, Zaikun},
	journal={SIAM J. Sci. Comput.},
	volume={40},
	number={4},
	pages={A2174--A2199},
	year={2018},
	publisher={SIAM}
}

@article{kunapuli2008classification,
	title={Classification model selection via bilevel programming},
	author={Kunapuli, Gautam and Bennett, Kristin P and Hu, Jing and Pang, Jong-Shi},
	journal={Optim. Method Softw.},
	volume={23},
	number={4},
	pages={475--489},
	year={2008},
	publisher={Taylor \& Francis}
}

@article{yen1995lipschitz,
	title={Lipschitz continuity of solutions of variational inequalities with a parametric polyhedral constraint},
	author={Yen, Nguyen Dong},
	journal={Math. Oper. Res.},
	volume={20},
	number={3},
	pages={695--708},
	year={1995},
	publisher={INFORMS}
}

@book{clarke1990optimization,
	title={Optimization and Nonsmooth Analysis},
	author={Clarke, Frank H},
	year={1990},
	publisher={SIAM}
}

@inproceedings{liu2023averaged,
	title={Averaged method of multipliers for bi-level optimization without lower-level strong convexity},
	author={Liu, Risheng and Liu, Yaohua and Yao, Wei and Zeng, Shangzhi and Zhang, Jin},
	booktitle={International Conference on Machine Learning},
	pages={21839--21866},
	year={2023},
	organization={PMLR},
	journal={arXiv preprint arXiv:2302.03407},
}

@article{qi2000new,
	title={A new look at smoothing {Newton} methods for nonlinear complementarity problems and box constrained variational inequalities},
	author={Qi, Liqun and Sun, Defeng and Zhou, Guanglu},
	journal={Math. Program.},
	volume={87},
	number={1},
	pages={1--35},
	year={2000},
	publisher={Springer}
}

@article{chen2004smoothing,
	title={A smoothing method for a mathematical program with {P-matrix} linear complementarity constraints},
	author={Chen, Xiaojun and Fukushima, Masao},
	journal={Comput. Optim. Appl.},
	volume={27},
	number={3},
	pages={223--246},
	year={2004},
	publisher={Springer}
}

@article{outrata1995numerical,
	title={A numerical approach to optimization problems with variational inequality constraints},
	author={Outrata, Jir{\'\i} and Zowe, Jochem},
	journal={Math. Program.},
	volume={68},
	number={1},
	pages={105--130},
	year={1995},
	publisher={Springer}
}

@article{marcotte1986network,
	title={Network design problem with congestion effects: A case of bilevel programming},
	author={Marcotte, Patrice},
	journal={Math. Program.},
	volume={34},
	number={2},
	pages={142--162},
	year={1986},
	publisher={Springer}
}

@article{gabriel1997smoothing,
	title={Smoothing of mixed complementarity problems},
	author={Gabriel, Steven A and Mor{\'e}, Jorge J},
	journal={Complementarity and Variational Problems: State of the Art},
	volume={92},
	pages={105--116},
	year={1997},
	publisher={SIAM Philadelphia}
}

@article{chen2012smoothing,
	title={Smoothing methods for nonsmooth, nonconvex minimization},
	author={Chen, Xiaojun},
	journal={Math. Program.},
	volume={134},
	number={1},
	pages={71--99},
	year={2012},
	publisher={Springer}
}

@article{chen1998global,
	title={Global and superlinear convergence of the smoothing {Newton} method and its application to general box constrained variational inequalities},
	author={Chen, Xiaojun and Qi, Liqun and Sun, Defeng},
	journal={Math. Comput.},
	volume={67},
	number={222},
	pages={519--540},
	year={1998}
}

@article{chen1998global2,
	title={Global and superlinear convergence of inexact {Uzawa} methods for saddle point problems with nondifferentiable mappings},
	author={Chen, Xiaojun},
	journal={SIAM J. Numer. Anal.},
	volume={35},
	number={3},
	pages={1130--1148},
	year={1998},
	publisher={SIAM}
}

@article{dutta2025nonconvex,
	title={Nonconvex Quasi-Variational Inequalities: Stability Analysis and Application to Numerical Optimization},
	author={Dutta, Joydeep and Lafhim, Lahoussine and Zemkoho, Alain and Zhou, Shenglong},
	journal={J. Optimiz. Theory App.},
	volume={204},
	number={2},
	pages={1--43},
	year={2025},
	publisher={Springer}
}

@article{liu2023hierarchical,
	title={Hierarchical Optimization-Derived Learning},
	author={Liu, Risheng and Liu, Xuan and Zeng, Shangzhi and Zhang, Jin and Zhang, Yixuan},
	journal={IEEE T. Pattern. Anal.},
	volume={45},
	number={12},
	pages={14693--14708},
	year={2023},
	publisher={IEEE}
}

@inproceedings{liu2022optimization,
	title={Optimization-Derived Learning with Essential Convergence Analysis of Training and Hyper-training},
	author={Liu, Risheng and Liu, Xuan and Zeng, Shangzhi and Zhang, Jin and Zhang, Yixuan},
	booktitle={International Conference on Machine Learning},
	pages={13825--13856},
	year={2022},
	organization={PMLR}
}

@article{mckenzie2024three,
	title={Three-operator splitting for learning to predict equilibria in convex games},
	author={McKenzie, Daniel and Heaton, Howard and Li, Qiuwei and Fung, Samy Wu and Osher, Stanley and Yin, Wotao},
	journal={SIAM J. Math. Data Sci.},
	volume={6},
	number={3},
	pages={627--648},
	year={2024},
	publisher={SIAM}
}

@inproceedings{fung2022jfb,
	title={{JFB: Jacobian-free} backpropagation for implicit networks},
	author={Fung, Samy Wu and Heaton, Howard and Li, Qiuwei and McKenzie, Daniel and Osher, Stanley and Yin, Wotao},
	booktitle={AAAI Conference on Artificial Intelligence}, 
	volume={36},
	pages={6648--6656},
	year={2022}
}

@book{beck2017first,
	title={First-order Methods in Optimization},
	author={Beck, Amir},
	year={2017},
	publisher={SIAM}
}

@book{facchinei2003finite,
	title={Finite-dimensional Variational Inequalities and Complementarity Problems},
	author={Facchinei, Francisco and Pang, Jong-Shi},
	year={2003},
	publisher={Springer}
}

@book{luo1996mathematical,
	title={Mathematical Programs with Equilibrium Constraints},
	author={Luo, Zhi-Quan and Pang, Jong-Shi and Ralph, Daniel},
	year={1996},
	publisher={Cambridge University Press}
}

@article{okuno2021lp,
	title={On $\ell_p$-hyperparameter Learning via Bilevel Nonsmooth Optimization},
	author={Okuno, Takayuki and Takeda, Akiko and Kawana, Akihiro and Watanabe, Motokazu},
	journal={J. Mach. Learn. Res.},
	volume={22},
	number={245},
	pages={1--47},
	year={2021}
}

@article{facchinei1999smoothing,
	title={A smoothing method for mathematical programs with equilibrium constraints},
	author={Facchinei, Francisco and Jiang, Houyuan and Qi, Liqun},
	journal={Math. Program.},
	volume={85},
	number={1},
	pages={107--134},
	year={1999},
	publisher={Citeseer}
}

@article{cui2023complexity,
	title={Complexity guarantees for an implicit smoothing-enabled method for stochastic {MPECs}},
	author={Cui, Shisheng and Shanbhag, Uday V and Yousefian, Farzad},
	journal={Math. Program.},
	volume={198},
	number={2},
	pages={1153--1225},
	year={2023},
	publisher={Springer}
}

@article{scholtes2001convergence,
	title={Convergence properties of a regularization scheme for mathematical programs with complementarity constraints},
	author={Scholtes, Stefan},
	journal={SIAM J. Optim.},
	volume={11},
	number={4},
	pages={918--936},
	year={2001},
	publisher={SIAM}
}

@article{chen2009class,
	title={A class of quadratic programs with linear complementarity constraints},
	author={Chen, Xiaojun and Ye, Jane J},
	journal={Set-Valued Var. Anal.},
	volume={17},
	number={2},
	pages={113--133},
	year={2009},
	publisher={Springer Verlag},
}

@book{dempe2002foundations,
	title={Foundations of Bilevel Programming},
	author={Dempe, Stephan},
	year={2002},
	publisher={Springer Science \& Business Media}
}

@article{wu2015inexact,
	title={An inexact {Newton} method for stationary points of mathematical programs constrained by parameterized quasi-variational inequalities},
	author={Wu, Jia and Zhang, Liwei and Zhang, Yi},
	journal={Numer. Algorithms},
	volume={69},
	number={4},
	pages={713--735},
	year={2015},
	publisher={Springer}
}

@article{chen2022efficient,
	title={An efficient threshold dynamics method for topology optimization for fluids},
	author={Chen, Huangxin and Leng, Haitao and Wang, Dong and Wang, Xiaoping},
	journal={CSIAM T. Appl. Math.},
	pages={26-56},
	volume = {3},
	number = {1},
	year={2022}
}

@article{chu2026provably,
	title={A provably convergent plug-and-play framework for stochastic bilevel optimization},
	author={Chu, Tianshu and Xu, Dachuan and Yao, Wei and Yu, Chengming and Zhang, Jin},
	journal={J. Mach. Learn. Res.},
	volume={27},
	number={171},
	pages={1--41},
	year={2026}
}

\end{document}